\documentclass{amsart}
\usepackage[utf8]{inputenc}
\usepackage{microtype,mathrsfs,amssymb,mathtools,eucal,anyfontsize}

\usepackage{tikz-cd}
\usepackage{pgf,tikz,pgfplots}
\pgfplotsset{compat=1.15}
\usetikzlibrary{arrows}
\usetikzlibrary{patterns}
\usetikzlibrary{calc}
\makeatletter
\@tfor\next:=abcdefghijklmnopqrstuvwxyzABCDEFGHIJKLMNOPQRSTUVWXYZ\do{%
	\def\command@factory#1{%
		\expandafter\def\csname cal#1\endcsname{\mathcal{#1}}
		\expandafter\def\csname frak#1\endcsname{\mathfrak{#1}}
		\expandafter\def\csname scr#1\endcsname{\mathscr{#1}}
		\expandafter\def\csname bb#1\endcsname{\mathbb{#1}}
		\expandafter\def\csname rm#1\endcsname{\mathrm{#1}}
		\expandafter\def\csname bf#1\endcsname{\mathbf{#1}}
	}
	\expandafter\command@factory\next
}
\makeatother

\usepackage{hyperref}
\hypersetup{colorlinks,%
citecolor=black,%
filecolor=black,%
linkcolor=blue,%
urlcolor=black,%
}
\usepackage[capitalize,nameinlink,noabbrev]{cleveref}
\newtheorem{theorem}{Theorem}[section]
\newtheorem{prop}[theorem]{Proposition}
\newtheorem{lemma}[theorem]{Lemma}
\newtheorem{cor}[theorem]{Corollary}

\theoremstyle{definition}
\newtheorem{defn}[theorem]{Definition}
\newtheorem{question}[theorem]{Question}

\theoremstyle{remark}
\newtheorem{remark}[theorem]{Remark}

\newcommand{\cat}{\mathrm{CAT(0)}}
\newcommand{\wt}{\widetilde}

\newcommand{\leftQ}[2]{\left.\raisebox{-.2em}{$#2$}\middle\backslash\raisebox{.2em}{$#1$}\right.}
\title{Cubulating a free-product-by-cyclic group}

\author{Fran\c{c}ois Dahmani}
\address{Fran\c{c}ois Dahmani, Institut Fourier, Univ. Grenoble Alpes, F-38000 Grenoble, France}

\email{francois.dahmani@univ-grenoble-alpes.fr}

\author{Suraj Krishna M S }
\address{Suraj Krishna, Faculty of Mathematics, Technion -- Israel Institute of Technology, Haifa 32000, Israel}

\email{surajms@campus.technion.ac.il}

\begin{document}

\maketitle

\begin{abstract}
   Let $G = H_1 * ... * H_k * F_r$ be a {finitely generated} torsion-free group and $\phi$ an automorphism of $G$ that preserves this free factor system. We show that when $\phi$ is fully irreducible and atoroidal relative to this free factor system, the mapping torus $\Gamma = G \rtimes_{\phi} \mathbb{Z}$ acts relatively geometrically on a hyperbolic CAT(0) cube complex. This is a generalisation of a result of Hagen and Wise for hyperbolic free-by-cyclic groups.
\end{abstract}

\tableofcontents

\section{Introduction}
Consider a finitely generated free group $F$ and an automorphism $\phi : F \to F$. Hagen and Wise showed in \cite{hagen_wise_irreducible} that if $\phi$ is atoroidal and fully irreducible, then the mapping torus $F \rtimes_{\phi} \bbZ$ acts properly and cocompactly on a hyperbolic CAT(0) cube complex. It often happens that automorphisms of free groups are neither atoroidal nor fully irreducible, suggesting various directions of generalisation. In \cite{hagen_wise_general}, Hagen and Wise relaxed the requirement of full irreducibility and, using the sophisticated machinery of relative train track maps, showed that the mapping torus still acts geometrically on a CAT(0) cube complex. They asked (see the discussion around Problem B in \cite{hagen_wise_general}) if a systematic answer to which free-by-cyclic groups admit cubulations is possible, especially in the presence of polynomially growing subgroups.

In this article, by investigating \emph{relative cubulations} instead of usual cubulations, we provide an answer in great generality as to when such groups are relatively cubulated. Let $\phi$ be an automorphism of $F$ and let $F = H_1 * H_2* \ldots * H_k * F_r$ be a free 
{decomposition}
% factor system 
that is preserved by $\phi$
{(up to taking conjugates of the factors)}. 
Such a free 
{decomposition}
% factor system
always exists for any $\phi$. In particular, when $\phi$ is not fully irreducible, there exists a free 
{decomposition}
% factor system 
preserved by $\phi$, relative to which $\phi$ is fully irreducible. Let us also assume that $\phi$ is atoroidal relative to the free 
{decomposition}.
% factor system.
We will clarify the meanings of both these terms {more precisely, using the notion of free factor systems,}  in \cref{sec: flow_space}. 
In our setting, we allow, for instance, elements to have polynomial growth under $\phi$ as long as they are elliptic in the free product. Under a mild complexity condition, the first named author and Li showed in \cite{dahmani_li} that the mapping torus $F \rtimes_{\phi} \bbZ$ is hyperbolic relative to the suspensions of the free factors $H_i$. A particular case of the main result of this paper shows that such a mapping torus acts relatively geometrically on a hyperbolic CAT(0) cube complex. 

\subsection{Main result}

Let $G$ be a {finitely generated} group, $\phi$ an automorphism of $G$ and $H$ a subgroup of $G$ whose conjugacy class is preserved by a power of $\phi$. Let $n$ be the minimal positive power of $\phi$ such that $\phi^n(H) = g^{-1}Hg$. Then we say that the suspension of $H$ by $\phi$ in the semidirect product $G\rtimes_{\phi} \langle t \rangle$ is the group $H\rtimes_{ad_g \circ \phi^n} \langle t^{n} g^{-1}\rangle $, { where $ad_g: G \to G$ denotes the inner automorphism that sends any element $h \mapsto ghg^{-1}$. }

To state our main result for automorphisms of free products of groups, we need the notions of full irreducibility and atoroidality \emph{relative to a free factor system}. These notions are analogous to the corresponding ones for automorphisms of free groups, with the condition that the free 
{decomposition}
% factor system 
is preserved, but within a free factor, there are no restrictions. 
We also need a technical notion of \emph{no twinned subgroups}.\footnote{This requirement can in fact be removed, see \cite[Lemma 3.3]{DMM23} or \cref{rem: no twins}.} 
We refer the reader to \cref{sec: flow_space} for the definitions.  

We also recall 
% from \cite{relative_cubulation} 
the notion of relative cubulation introduced by Einstein--Groves in \cite{relative_cubulation}: 
A relatively hyperbolic group $(\Gamma,\scrP)$ is \emph{relatively cubulated} if it acts cocompactly on a CAT(0) cube complex with cell stabilisers either trivial or conjugate to a finite index subgroup in $\scrP$.

\begin{theorem}\label{thm: main_intro}
    {Let $G$ be a finitely generated torsion-free group and let $G \cong H_1 * \dots * H_k*F_r$ be a free decomposition such that each $H_i$ is nontrivial.} Let $\phi$ be an automorphism that preserves 
    % this 
    {the associated}
    free factor system. Assume that 
    % $r \geq 2$ or 
    $k+r \geq 3$, that $\phi$ is fully irreducible relative to the free factor system and atoroidal relative to the free factor system, and that there exist no twinned subgroups. 
    Then the mapping torus $\Gamma = G \rtimes_\phi \bbZ$ admits a relative cubulation for the peripheral structure of the suspensions of the free factors $H_i$.
\end{theorem}

We recover the result of \cite{hagen_wise_irreducible} when $G = F_r$ above\footnote{it is stated there in terms of irreducible atoroidal automorphisms, which are fully irreducible by \cite[Corollary B.4]{irred_iwip}.}, as well as some cases of \cite{hagen_wise_general}, through a telescopic argument of Groves and Manning \cite[Theorem D]{groves_manning_omnibus}.  
In  general, the relatively geometric action on CAT(0) cube complexes that we obtain otherwise still has interesting consequences.

Note that any group acting on a CAT(0) cube complex admits an action on an $\ell^2$-space built using characteristic functions of hyperplane-halfspaces \cite{niblo_reeves}. We thus have: 

\begin{cor}
    Let $G$ and $\phi$ be as in \cref{thm: main_intro}. Then the mapping torus $\Gamma$ acts on a Hilbert space with unbounded orbits, with no global fixed point for $G$. 
\end{cor}

In addition, when $G$ is residually finite, using a generalisation by Einstein--Groves \cite[Theorem 1.6]{einstein_groves_separability} of a well-known result of Haglund--Wise \cite{haglund_wise_special}, we have:

\begin{cor}
    Let $G$ and $\phi$ be as in \cref{thm: main_intro}. Further, let $G$ be residually finite. Then every full relatively quasiconvex subgroup of the mapping torus $\Gamma$ is separable.
\end{cor}

{Another consequence of \cref{thm: main_intro} can be seen in our recent work with Mutanguha \cite{DMM23}, where we showed that all hyperbolic hyperbolic-by-cyclic groups are virtually special.}

\subsection{Method}
The procedure that we adopt
to cubulate follows the scheme laid out in \cite{hagen_wise_irreducible}. The goal is to obtain a collection of codimension-1 subgroups of $\Gamma$ and then apply Sageev's dual cube complex construction \cite{sageev_cubulation}. The codimension-1 subgroups we build will be stabilised by full relatively quasiconvex subgroups in the relatively hyperbolic group $\Gamma$. We then apply the boundary criterion of \cite{relative_cubulation}.

In order to do this, an important tool is Francaviglia--Martino's absolute train tracks for free products \cite{francaviglia_martino}. Given $G$ and an automorphism $\phi$ satisfying the hypotheses of the main result, there exists a $G$-tree $T$ 
and a train track map $f:T\to T$ representing $\phi$. Taking mapping cylinders for $f$,  one then obtains a \emph{flow space} on which $\Gamma$ acts. The flow space has the structure of a tree of spaces, where the underlying graph is a line and vertex and edge spaces are copies of the tree $T$. The map $f$ ``flows'' a point on any tree to its image in the next tree. We describe the flow space and various properties we need in order to define walls in the flow space in \cref{sec: flow_space}.

Before explaining how we build walls in our setup, let us motivate our construction in the surface case. Let $M_f$ be the mapping torus of a closed hyperbolic surface $S_g$ under a pseudo-Anosov map $f$. Cooper, Long and Reid showed in \cite{cooper_long_reid} that in this case there exists an immersed, quasi-Fuchsian surface in $M_f$ (and hence a quasiconvex wall in the universal cover). First, take a simple closed curve $C$ in $S_g$ which is disjoint from its $f$-image. Such a simple closed curve exists up to taking a finite cover of $M_f$. The required immersed surface in $M_f$ is then obtained by cutting $S_g$ along $C$ and $f(C)$ and gluing $C_{\pm}$ to $f(C_{\mp})$ (the cut-and-cross-join technique).

Hagen--Wise mimicked this construction to the setup of hyperbolic free-by-cyclic groups in the fully irreducible case. A surface with a pseudo-Anosov map is now replaced by a graph with a train track map. The analogue of cutting along a simple closed curve is cutting along a point in the graph. However, the situation here is more complicated as a train track map is only a homotopy equivalence and not a homeomorphism. A point often has multiple pre-images and the cut-and-cross-join operation is performed along all points with the same image.

We use the same operation, but now our train track representative is defined not on a finite graph but on the $G$-tree $T$. The lack of local finiteness of $T$ tree gives rise to additional difficulties, but we were able to manage them because of the behaviour of train track maps in this setting, and considerations of angles at vertices under relative hyperbolicity.

While cocompact cubulations require walls to be (relatively) quasiconvex, relative cubulations require walls to also be full. This forced us to introduce \emph{saturations} of our walls in order to ensure fullness. The construction of walls and their saturations can be found in \cref{sec: walls}. 

Finally, in order to use the boundary criterion, we need sufficiently many wall saturations to not only cut bi-infinite geodesics in the flow space, but also to cut pairs of \emph{principal flow lines} that are stabilised by maximal parabolic subgroups and pairs consisting of a geodesic ray and a principal flow line. We show this in \cref{sec: wall_saturations_cut}, ensuring the separation of every pair of points in the Bowditch boundary. We verify the latter in \cref{sec: separation_bowditch}, where we also give a proof of \cref{thm: main_intro}. 

\subsection{Questions}
We end this introduction with three questions arising from this work.
\begin{question}
    Let $G = A*B$ be a torsion-free group and $\phi$ be an automorphism that is fully irreducible and atoroidal relative to the above free {decomposition}. Does the mapping torus of $G$ admit a relative cubulation?
\end{question}

The above makes a case for a combination theorem of relatively cubulated groups, which is as yet a largely unexplored area of research.

\begin{question}  Let $G$ be a free product and  $\phi$ be an automorphism that is fully irreducible but not necessarily atoroidal, relative to the given free {decomposition}. Let $\scrP$ be the peripheral structure of suspensions of maximal subgroups of $G$ on which iterations of $\phi$ make lengths of conjugacy classes grow at most polynomially.  Does the mapping torus of $G$ admit a relative cubulation?
\end{question} 

Motivated by \cite{dahmani_li}, such a construction would apply to free groups automorphisms, refining the cartography of possible cubulations of free-by-cyclic groups.

In \cite{dahmani_ms}, we showed that the mapping torus of a torsion-free hyperbolic group is hyperbolic relative to the suspensions of the maximal polynomially growing subgroups. This leads to a natural question: 

\begin{question}
    Let $G$ be a torsion-free hyperbolic group and $\phi$ be an automorphism of $G$. Does the mapping torus of $G$ admit a relative cubulation?
\end{question}

{The answer is yes when $\phi$ is atoroidal \cite{DMM23}.}
% that the above question has a positive answer, with the mapping torus of $G$ admitting a (proper and cocompact) cubulation.}

\subsection{Acknowledgements}
Suraj Krishna was partially supported during this work by CEFIPRA grant number 5801-1, ``Interactions between dynamical systems, geometry and number theory" at TIFR and by ISF grant number 1226/19 at the Technion. Fran\c{c}ois Dahmani is supported by ANR-22-CE40-0004 GoFR.  Both acknowledge support by    LabEx CARMIN, ANR-10-LABX-59-01.    Different institutions must be thanked for hosting different stages of this work: the IHP (UAR 839 CNRS-Sorbonne Universit\'e)  hosting the trimester program ``Groups acting on Fractals'' in 2022, the Institut Fourier, the Math department of the Technion, for mutual visits, and the IHES.

We would like to thank Mark Hagen, Thomas Ng and Dani Wise for helpful conversations{, and the anonymous referee for numerous comments which made the paper more pleasant}.

\section{The flow space} \label{sec: flow_space}

\subsection{Free \texorpdfstring{$G$}{G}-trees relative to \texorpdfstring{$\calH$}{H}, and train track maps}
{Let us fix $G$ to be a finitely generated group for the rest of the paper.} 

{A free factor system for $G$ is a tuple $(H_1,\dots, H_k)$ of subgroups such that there exists a free subgroup $F<G$ for which $G= H_1*H_2*\dots*H_k*F$.  Another free factor system $(J_1,\dots, J_\ell)$ of $G$ is strictly larger if each $H_i$ is conjugate into some $J_r$, and one inclusion is strict.}

A $G$-tree is a metric tree endowed with an isometric action of $G$.

A \emph{free $G$-tree relative to $\calH= \{H_1,\dots, H_k\}$}
is a $G$-tree which is minimal for its $G$-action, its edge stabilisers are trivial, and its {nontrivial} elliptic subgroups are exactly the conjugates of the $\{H_1, \dots, H_k\} $ in $G$. We may as well require that there is no vertex of valence $2$ that has trivial stabiliser. 

A vertex is \emph{singular} if its stabiliser is nontrivial. 

Observe that, because $G$ is finitely generated, any such $G$-tree has finite quotient by $G$.  There is a whole space of such free $G$-trees relative to $\{H_1, \dots, H_k\}$, as studied in \cite{GL_outerspace}. 

It is convenient to have a notion of \emph{angle} in these non-locally finite trees. Let $T$ be a free $G$-tree relative to $\calH$. Let us choose a word metric for each $H_j${, which is finitely generated as $G$ is}. Let $v_j \in T$ be the unique vertex fixed by $H_j$, and finally choose a finite set of orbit representatives for the $H_j$-action of edges issued from $v_j$. Denote the angle between two edges $e,e'$ issued from the vertex fixed by $H_j$,  to be the word length of the element {$h h'^{-1}$} 
given by $h$ and $h'$ sending respectively $e$ and $e'$ into our set of representatives. If $e,e'$ are distinct edges in the finite set of orbit representatives, the angle between them is one. 
It is clear that the $Stab(v)$-action on edges adjacent to $v$ preserves angles. We hence complete the definition by $G$-equivariance: it defines the angle between two edges adjacent to singular vertices stabilised by conjugates of the $H_j$.  In other (locally finite) vertices, we may convene that angles are $0$ (if edges are equal) or $1$ (if they are different).   Observe that only finitely many edges make a given angle with a given edge, and that angles (around a given vertex) satisfy the triangular inequality.

Let $\phi$ be an automorphism of $G$ that preserves the conjugacy class of each $H_i$.  Consider a free $G$-tree $T$ relative to $\{H_1,\dots, H_k\}$.  One says that a continuous map \emph{$f:T\to T$ realises the automorphism $\phi$} if it is equivariant in the sense that for all $p\in T$, and all $g\in G$, $f(gp)= \phi(g) f(p)$.  Such maps realising $\phi$ exist in our context  \cite[Lemma 4.2]{francaviglia_martino}.
Such a map is also, by equivariance, a quasiisometry of $T$. 

A \emph{turn} is a pair of edges $e_1$ and $e_2$ in $T$ starting at a common vertex $v$. The pair is a \emph{proper turn} if $e_1$ and $e_2$ are distinct. We say that $e_1$ and $e_2$ make a \emph{legal turn} if the two paths $f(e_1)$ and $f(e_2)$ starting at $f(v)$ share no proper common subpath. In other words, by definition, $f$ sends legal turns to proper turns.  We say that $f$ is a \emph{train track map} if it sends edges to reduced paths without non-legal turns, and if moreover, for all such $e_1, e_2$ in $T$ making a legal turn, their images  $f(e_1) $  and $f(e_2)$ are two paths starting at $f(v)$ by two edges that themselves make a legal turn. In other words, by definition, $f$ is a train track map if it  sends legal turns to legal turns. 

It is far from obvious that train track maps representing automorphisms exist. A theorem of Francaviglia and Martino \cite{francaviglia_martino} ensures that if $\phi$ is fully irreducible relative to $\{H_1, \dots, H_k\}$, then there exists a free $G$-tree relative to $\{H_1, \dots, H_k\}$, that we denote by $T$, and there exists $f:T\to T$ continuous, with constant speed on edges, that realises $\phi$ and is a train track map.

Recall that, {if $(H_1, \cdots, H_k)$ is a free factor system of $G$, and if $\phi$ is an automorphism permuting it (i.e. preserving the set of conjugacy classes $\{[H_1], \cdots, [H_k]\}$), we say that } $\phi$ is \emph{fully irreducible relative to $\{H_1, \cdots, H_k\}$} \cite[Definition 8.2]{francaviglia_martino} if {no positive power of $\phi$ preserves any larger free factor system.} 
%given any free product decomposition $G = H'_1 * \cdots * H'_{k'} * F_{m'}$ such that for each $r$ there exists an $l$ such that $H_r$ is a subgroup of a conjugate of $H'_l$, with the containment being proper for at least one of them, then $\phi^i$ (for each $i \in \mathbb{Z}$) does not preserve the conjugacy class of some $H'_l$.

We recall for later use an equivalent formulation of irreducibility (see Definition 8.1 and Lemma 8.3 of \cite{francaviglia_martino}). Let $f : T \to T$ be a map realising the automorphism $\phi$. We say $f$ is \emph{irreducible} if for every proper subgraph $W$ of $T$ that is $f$-invariant and $G$-invariant, the quotient of $W$ by $G$ is a forest such that each component subtree contains at most one non-free vertex. The map $f$ is \emph{fully irreducible} if $f^i$ is irreducible for all $i>0$. 
We say $\phi$ is (fully) irreducible relative to $\{H_1, \cdots, H_k\}$ if every $f$ realising $\phi$ is (fully) irreducible.

\subsection{The flow space of an automorphism}

From now on $T$ and $f$ are thus chosen, so that $f$ is a train track map on $T$ representing 
{the relatively fully irreducible}
$\phi$. 

Define, for all $i\in \mathbb{Z}$ a tree $T_i$ equivariantly isomorphic to $T$. Let us denote by $(p\mapsto p_i)$  the identification $T\to T_i$. Define the action of $G$ on $T_i$ as $g.p_i = (\phi^i(g) p)_i$. Observe that this makes $T_i$ a free $G$-tree relative to $\{H_1, \dots, H_k\}$. Define $f_i:T_i \to T_{i+1}$ by $f_i(p_i) = (f(p))_{i+1}$. Observe, by the property of train tracks, how $f_i$ sends a turn: if it is legal when seen in $T$, its image is a legal turn when seen in $T$. We will make the abuse of language that any composition of the form {$f_{i+k}\circ f_{i+k-1}\circ \cdots \circ f_{i+1} \circ f_i$}   (from $T_i$ to $T_{i+k+1}$) is called a (positive) iteration of $f$.  

Start with the disjoint union  of all the $T_i$, $i\in \mathbb{Z}$. For each $i$, and each (unoriented) edge $e_i$ in $T_i$, we choose an orientation, and glue a rectangle $R_{e_i} = [0,1] \times [0, 1]$ such that $\{0\} \times [0,1]$
is glued to $e_i$, and $\{1\} \times [0,1]$  
is glued on the path $f_i(e_i)$ in $T_{i+1}$ (respecting orientation). See \cref{fig:2cells} for an illustration.  Finally, for each vertex $v_i$, we glue together the sides $[0,1] \times \{0\}$ 
of the rectangles $R_{e_i}$ for all $e_i$ starting at this vertex, and the sides $[0,1] \times \{1\}$ 
of the rectangles $R_{e_i}$ for all $e_i$ terminating  at this vertex. The resulting space, with a natural structure of a cell complex, is denoted by $\wt X$.  We call this space the \emph{flow space of $\phi$} (with respect to $\calH$, $T$ and $f$). Thus, the flow space is a \emph{tree of spaces}, where the underlying tree is a bi-infinite combinatorial line and vertex spaces are the trees $T_i$.

\begin{figure}
    \centering
\begin{tikzpicture}[scale=.5,line cap=round,line join=round,>=triangle 45,x=1cm,y=1cm]
% \clip(-14.387105658369531,-3.9683314733824613) rectangle (-1.5687854416463496,13.806405893807119);
\draw [line width=3pt,color=red] (-6,7)-- (-6,5);
\draw [color=red, line width = 2pt] (-6,6.5)-- (-6,5.5);
\draw [line width=3pt,color=red] (-2,9)-- (-3,7.5);
\draw [line width=3pt,color=red] (-3,7.5)-- (-3,5.6);
\draw [line width=3pt,color=red] (-3,5.6)-- (-2,4);
\draw [line width=2pt] (-6,7)-- (-2,9);
\draw [line width=2pt] (-6,5)-- (-2,4);
\draw [line width=2pt] (-2,9)-- (3,11);
\draw [line width=2pt] (-3,7.5)-- (3,8);
\draw [line width=2pt] (-3,5.6)-- (3,5);
\draw [line width=2pt] (-2,4)-- (3,2);
\draw [line width=3pt,color=red] (3,2)-- (2,3);
\draw [line width=3pt,color=red] (2,3)-- (2,4);
\draw [line width=3pt,color=red] (2,4)-- (3,5);
\draw [line width=3pt,color=red] (3,5)-- (4,6);
\draw [line width=3pt,color=red] (4,6)-- (4.028,7.15);
\draw [line width=3pt,color=red] (4.028,7.15)-- (3,8);
\draw [line width=3pt,color=red] (3,8)-- (2,9);
\draw [line width=3pt,color=red] (2,9)-- (2,10);
\draw [line width=3pt,color=red] (2,10)-- (3,11);

\draw [line width=0.8pt] (-6,7)-- (-5.3,7.876);
\draw [line width=0.8pt] (-6,7)-- (-5.696,7.92);
\draw [line width=0.8pt] (-6,7)-- (-6.026,7.986);
\draw [line width=0.8pt,dotted] (-6.0167245022983495,7.63424458716049)-- (-6.752,7.238);
\draw [line width=0.8pt] (-6,7)-- (-7.06,7.304);
\draw [line width=0.8pt] (-6,5)-- (-5.322,4.114);
\draw [line width=0.8pt] (-6,5)-- (-5.74,4.07);
\draw [line width=0.8pt] (-6,5)-- (-6.114,4.026);
\draw [line width=0.8pt] (-6,5)-- (-6.906,4.95);
\draw [line width=0.8pt] (-2,4)-- (-0.944,4.18);
\draw [line width=0.8pt] (-2,4)-- (-1.384,4.488);
\draw [line width=0.8pt] (-2,4)-- (-0.988,3.806);
\draw [line width=0.8pt] (-2.96,5.66)-- (-2.022,5.918);
\draw [line width=0.8pt] (-2.96,5.66)-- (-2.352,6.072);
\draw [line width=0.8pt] (-2.96,5.66)-- (-2.638,6.27);
\draw [line width=0.8pt] (-2,9)-- (-2.66,9.504);
\draw [line width=0.8pt] (-2,9)-- (-2.374,9.702);
\draw [line width=0.8pt] (-2,9)-- (-2.066,9.834);
\draw [line width=0.8pt] (-2,9)-- (-1.274,9.636);
\draw [line width=0.8pt] (3,11)-- (2.18,11.242);
\draw [line width=0.8pt] (3,11)-- (2.4,11.506);
\draw [line width=0.8pt] (3,11)-- (3.6963425995492196,11.550223891810681);
\draw [line width=0.8pt] (3,11)-- (3.8963425995492194,11.030223891810682);
\draw [line width=0.8pt] (4.028,7.15)-- (4.0763425995492195,8.030223891810678);
\draw [line width=0.8pt] (4.028,7.15)-- (4.396342599549219,7.990223891810678);
\draw [line width=0.8pt] (4.028,7.15)-- (4.65634259954922,6.430223891810677);
\draw [line width=0.8pt] (4.028,7.15)-- (4.896342599549219,6.710223891810677);
\draw [line width=0.8pt] (2,9)-- (2.65634259954922,9.490223891810679);
\draw [line width=0.8pt] (2,9)-- (2.8563425995492198,9.010223891810679);
\draw [line width=0.8pt] (3,5)-- (3.9163425995492194,5.010223891810675);
\draw [line width=0.8pt] (3,5)-- (3.7963425995492193,4.4902238918106745);
\draw [line width=0.8pt] (3,5)-- (3.4763425995492194,4.350223891810674);
\draw [line width=0.8pt] (3,5)-- (2.7763425995492197,4.410223891810674);
\draw [line width=0.8pt] (2,3)-- (2.29634259954922,3.510223891810673);
\draw [line width=0.8pt] (2,3)-- (2.57634259954922,3.350223891810673);
\draw [line width=0.8pt] (2,3)-- (2.65634259954922,2.6502238918106724);
\draw [line width=0.8pt] (3,2)-- (3.2963425995492197,2.4502238918106722);
\draw [line width=0.8pt] (3,2)-- (3.5963425995492195,2.110223891810672);
\draw [line width=0.8pt] (3,2)-- (3.0963425995492195,1.270223891810671);
\draw [line width=0.8pt] (3,2)-- (2.67634259954922,1.410223891810671);
\draw [line width=0.8pt,dotted] (-6.163657400450777,4.4502238918106745)-- (-6.463657400450777,4.890223891810675);
\draw [line width=0.8pt,dotted] (-2.066,9.834)-- (-1.274,9.636);
\draw [line width=0.8pt,dotted] (-1.384,4.488)-- (-0.944,4.18);
\draw [line width=0.8pt,dotted] (-0.944,4.18)-- (-0.988,3.806);
\draw [line width=0.8pt,dotted] (2.6912053228772304,11.260416844373536)-- (3.3236152407482327,11.255708666005802);
\draw [line width=0.8pt,dotted] (4.216108932156037,7.579094053345781)-- (4.481055023177773,6.920548197229784);
\draw [line width=0.8pt,dotted] (2.7763425995492197,4.410223891810674)-- (3.4763425995492194,4.350223891810674);
\draw [line width=0.8pt,dotted] (2.57634259954922,3.350223891810673)-- (2.65634259954922,2.6502238918106724);
\draw [line width=0.8pt,dotted] (3.31531500599035,2.0582806714342285)-- (3.0409227328098853,1.6900184048779718);
% \begin{scriptsize}
\draw [fill=blue] (-6,7) circle (2.5pt);
\draw [fill=blue] (-6,5) circle (2.5pt);
\draw (-6.5,6) node {$e_0$};
\draw (-6,0) node {$T_0$};
\draw (-2,0) node {$T_1$};
\draw (3,0) node {$T_2$};

\draw [left] (-4,6.2) node {$R_{e_0}$};
\draw (-2.5,6.8) node {${f_1}$};
\draw (0.5,6.8) node {$R_{f_1}$};

\draw [fill=blue] (-2,9) circle (2.5pt);
\draw [fill=blue] (-2.96,7.44) circle (2.5pt);
\draw [fill=blue] (-2.96,5.66) circle (2.5pt);

\draw [fill=blue] (-2,4) circle (2.5pt);
\draw [fill=blue] (3,11) circle (2.5pt);
\draw [fill=blue] (3,8) circle (2.5pt);
\draw [fill=blue] (3,5) circle (2.5pt);
\draw [fill=blue] (3,2) circle (2.5pt);
\draw [fill=blue] (2,3) circle (2.5pt);
\draw [fill=blue] (2,4) circle (2.5pt);
\draw [fill=blue] (4,6) circle (2.5pt);
\draw [fill=blue] (4.028,7.15) circle (2.5pt);
\draw [fill=blue] (2,9) circle (2.5pt);
\draw [fill=blue] (2,10) circle (2.5pt);
\draw [fill=blue] (-5.3,7.876) circle (1pt);
\draw [fill=blue] (-5.696,7.92) circle (1pt);
\draw [fill=blue] (-6.026,7.986) circle (1pt);
\draw [fill=blue] (-7.06,7.304) circle (1pt);
\draw [fill=blue] (-5.322,4.114) circle (1pt);
\draw [fill=blue] (-5.74,4.07) circle (1pt);
\draw [fill=blue] (-6.114,4.026) circle (1pt);
\draw [fill=blue] (-6.906,4.95) circle (1pt);
\draw [fill=blue] (-0.944,4.18) circle (1pt);
\draw [fill=blue] (-1.384,4.488) circle (1pt);
\draw [fill=blue] (-0.988,3.806) circle (1pt);
\draw [fill=blue] (-2.022,5.918) circle (1pt);
\draw [fill=blue] (-2.352,6.072) circle (1pt);
\draw [fill=blue] (-2.638,6.27) circle (1pt);
\draw [fill=blue] (-2.66,9.504) circle (1pt);
\draw [fill=blue] (-2.374,9.702) circle (1pt);
\draw [fill=blue] (-2.066,9.834) circle (1pt);
\draw [fill=blue] (-1.274,9.636) circle (1pt);
\draw [fill=blue] (2.18,11.242) circle (1pt);
\draw [fill=blue] (2.4,11.506) circle (1pt);
\draw [fill=blue] (3.6963425995492196,11.550223891810681) circle (1pt);
\draw [fill=blue] (3.8963425995492194,11.030223891810682) circle (1pt);
\draw [fill=blue] (4.0763425995492195,8.030223891810678) circle (1pt);
\draw [fill=blue] (4.396342599549219,7.990223891810678) circle (1pt);
\draw [fill=blue] (4.65634259954922,6.430223891810677) circle (1pt);
\draw [fill=blue] (4.896342599549219,6.710223891810677) circle (1pt);
\draw [fill=blue] (2.65634259954922,9.490223891810679) circle (1pt);
\draw [fill=blue] (2.8563425995492198,9.010223891810679) circle (1pt);
\draw [fill=blue] (3.9163425995492194,5.010223891810675) circle (1pt);
\draw [fill=blue] (3.7963425995492193,4.4902238918106745) circle (1pt);
\draw [fill=blue] (3.4763425995492194,4.350223891810674) circle (1pt);
\draw [fill=blue] (2.7763425995492197,4.410223891810674) circle (1pt);
\draw [fill=blue] (2.29634259954922,3.510223891810673) circle (1pt);
\draw [fill=blue] (2.57634259954922,3.350223891810673) circle (1pt);
\draw [fill=blue] (2.65634259954922,2.6502238918106724) circle (1pt);
\draw [fill=blue] (3.2963425995492197,2.4502238918106722) circle (1pt);
\draw [fill=blue] (3.5963425995492195,2.110223891810672) circle (1pt);
\draw [fill=blue] (3.0963425995492195,1.270223891810671) circle (1pt);
\draw [fill=blue] (2.67634259954922,1.410223891810671) circle (1pt);
\draw [fill=blue] (-6.163657400450777,4.4502238918106745) circle (1pt);
\draw [fill=blue] (-6.463657400450777,4.890223891810675) circle (1pt);

\draw [fill=blue] (2.6912053228772304,11.260416844373536) circle (1pt);
\draw [fill=blue] (3.3236152407482327,11.255708666005802) circle (1pt);
\draw [fill=blue] (4.216108932156037,7.579094053345781) circle (1pt);
\draw [fill=blue] (4.481055023177773,6.920548197229784) circle (1pt);
\draw [fill=blue] (3.31531500599035,2.0582806714342285) circle (1pt);
\draw [fill=blue] (3.0409227328098853,1.6900184048779718) circle (1pt);
% \end{scriptsize}
\end{tikzpicture}
    \caption{2-cells in the flow space. The 2-cell $R_{e_0}$ is bounded by the vertical edge $e_0$ on the left, two horizontal edges, and three vertical edges (including $f_1$) in $T_1$.}
    \label{fig:2cells}
\end{figure}
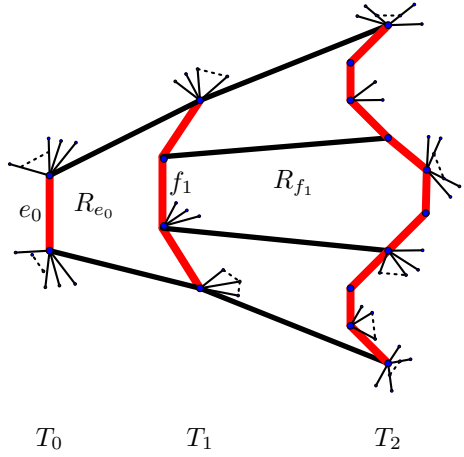

In $\wt X$, we will call any edge in some tree $T_i$ as a vertical edge, and the image of a side $[0,1] \times \{1\}$  or $[0,1] \times \{0\}$ of the rectangle $R(e_i)$ as a horizontal edge. We call a path horizontal if it intersects rectangles in horizontal segments $[0,1] \times \{h\}$.

If $f :T\to T$ is a train track map realising $\phi$, then for all $L\geq 1$, the map  $f^L$ is a train track map realising the automorphism $\phi^L$. Keeping the same $T$ and $f$ we may thus produce the space $\wt X_L$ by using the map $f^L$ realising $\phi^L$.  We will prefer to index the trees by $L\mathbb{Z}$ though.  

Observe that $\wt X_L$ need not be isomorphic to  $\wt X$. However we have a natural map $\varrho_L : \wt X_L \to \wt X$: it is obviously defined on the trees $T_{Li}$, and on a rectangle $R(e_{Li})$, one sends a horizontal edge onto the unique reduced concatenation of horizontal edges of $\wt X$ with the same endpoints.

\subsection{Action on the flow space}

Consider the semi-direct product $G\rtimes_{\phi} \mathbb{Z}$, and write $t$ for the generator of $\mathbb{Z}$ that induces the automorphism $\phi$ by conjugation on $G$: $G\rtimes_{\phi} \mathbb{Z}=\langle G, t\, |\, t^{-1}gt = \phi (g), \forall g\in G\rangle$.  Then $G\rtimes_{\phi} \mathbb{Z}$ acts co-finitely on $\wt X$, by defining the action of $G$ on each $T_i$ as above, and defining the action of $t$ to be the shift of indices: $t$ maps $T_i$ on $T_{i+1}$ isometrically, through the identification with $T$.   

The group $G$ preserves each of the trees $T_i$, and in each of them it induces the action of $G$ on $T$ precomposed by $\phi^i$. Thus, the orbits of $G$ on each of the $T_i$ are the same, after identification to $T$. In particular,

\begin{lemma}\label{lem: group_action_on_flowspace}
If $x, y \in T$ are such that there is $g \in G$ for which $gx=y$, and if $i_1, i_2 \in \mathbb{Z}$, then there exists $\gamma \in G\rtimes_{\phi} \mathbb{Z}$ for which $\gamma x_{i_1} = y_{i_2}$. \qed
\end{lemma}

\subsection{Forward flow, backward flow, and principal flow lines}

For each $i$, from each point  $x_i$ of $T_i$, there is a unique horizontal ray starting at $x_i$, containing all its images by positive iterates of $f$. We denote it by $\sigma(x_i)$, and $\sigma_k(x_i)$ is its initial subsegment of length $k$.  Its endpoint is denoted (slightly abusively) $f^k(x_i)$ as already convened.   
The segment $\sigma_1(x_i)$ of this ray whose endpoints are $x_i$ and $f_i(x_i) \in T_{i+1}$ is called the midsegment at $x_i$.  
We call the ray $\sigma(x_i)$ the \emph{forward flow (ray)}\footnote{forward path in \cite{hagen_wise_irreducible}.} from $x_i$. It is sometimes useful to use the forward flow of length $L$: the initial subsegment of length $L$ of the forward flow ray (see \cref{fig:fwd_ladder} for an illustration).  The \emph{forward ladder} of a forward flow (ray or segment) $\sigma$, denoted by $N(\sigma)$, is the smallest subcomplex of the cell complex $\wt X$ containing $\sigma$. 

\begin{figure}
    \centering
    \begin{tikzpicture}[scale = .7,line cap=round,line join=round,>=triangle 45,x=1cm,y=1cm]

\draw [line width=2pt,color=red] (-6.06,3.82)-- (-6,1.82);
\draw [line width=2pt] (-6,1.82)-- (-1.84,0.94);
\draw [line width=2pt,color=red] (-1.84,0.94)-- (-1.6,2.48);
\draw [line width=2pt,color=red] (-1.6,2.48)-- (-1.52,4.24);
\draw [line width=2pt,color=red] (-1.52,4.24)-- (-1.72,6.12);
\draw [line width=2pt] (-1.72,6.12)-- (-6.06,3.82);

\draw [line width=2pt] (-1.6,2.48)-- (3.5,3.7);
\draw [line width=2pt,color=red] (3.5,3.7)-- (3.5,-0.56);
\draw [line width=2pt] (3.5,-0.56)-- (-1.84,0.94);
\draw [line width=2pt,color=red] (-1.84,0.94)-- (-1.6,2.48);

\draw [line width=2pt,color=red] (3.5,2.28)-- (3.5,0.84);
\draw [line width=2pt] (3.5,0.84)-- (9.24,-1.36);
\draw [line width=2pt,color=red] (9.24,-1.36)-- (9.32,0.02);
\draw [line width=2pt,color=red] (9.32,0.02)-- (9.32,1.62);
\draw [line width=2pt,color=red] (9.32,1.62)-- (9.22,3.38);
\draw [line width=2pt,color=red] (9.22,3.38)-- (8.78,5.02);
\draw [line width=2pt] (8.78,5.02)-- (3.5,2.28);

\draw [line width=2pt,color=orange] (-6.013089879108802,2.2563293036267362)-- (-1.7193965420714639,1.7138721883747738);
\draw [line width=2pt,color=orange] (-1.7193965420714639,1.7138721883747738)-- (3.5,1.362);
\draw [line width=2pt,color=orange] (3.5,1.362)-- (9.32,0.9506);

\draw [line width=0.8pt,color=gray] (-3.902506522639143,1.989682184426714)-- (-4.2108,2.3784);
\draw [line width=0.8pt,color=gray] (0.6534,1.8944)-- (0.9889294339080661,1.5312869833001);
\draw [line width=0.8pt,color=gray] (0.9889294339080661,1.5312869833001)-- (0.6292,1.241);
\draw [line width=0.8pt,color=gray] (6.1468,1.5556)-- (6.603069343446125,1.1426524522519357);
\draw [line width=0.8pt,color=gray] (6.603069343446125,1.1426524522519357)-- (6.1468,0.8538);
\draw [line width=0.8pt,color=gray] (-3.9025065226391407,1.9896821844267139)-- (-4.235,1.7734);

\draw [fill=blue] (-6.06,3.82) circle (2pt);
\draw [fill=blue] (-6,1.82) circle (2pt);
\draw [fill=blue] (-1.84,0.94) circle (2pt);
\draw [fill=blue] (-1.6,2.48) circle (2pt);
\draw [fill=blue] (-1.52,4.24) circle (2pt);
\draw [fill=blue] (-1.72,6.12) circle (2pt);
\draw [fill=blue] (3.5,3.7) circle (2pt);
\draw [fill=blue] (3.5,-0.56) circle (2pt);
\draw [fill=blue] (3.5,2.28) circle (2pt);
\draw [fill=blue] (3.5,0.84) circle (2pt);
\draw [fill=blue] (9.24,-1.36) circle (2pt);
\draw [fill=blue] (9.32,0.02) circle (2pt);
\draw [fill=blue] (9.32,1.62) circle (2pt);
\draw [fill=blue] (9.22,3.38) circle (2pt);
\draw [fill=blue] (8.78,5.02) circle (2pt);
\draw [fill=black] (-6.013089879108802,2.2563293036267362) circle (2.5pt);
\draw[left] (-6,2.25) node {$x$};
\draw [fill=black] (-1.7193965420714639,1.7138721883747738) circle (2.5pt);
\draw[below right] (-1.85,1.75) node {$f(x)$};
\draw [fill=black] (3.5,1.362) circle (2.5pt);
\draw[right] (3.4,1.8) node {$f^2(x)$};
\draw [fill=black] (9.32,0.9506) circle (2.5pt);
\draw[below right] (9.4,1.5) node {$f^3(x)$};

\end{tikzpicture}
    \caption{The orange segment between $x$ and $f^3(x)$ is the forward flow of length $3$ from $x$. Its forward ladder is the union of the $2$-cells in the picture.}
    \label{fig:fwd_ladder}
\end{figure}
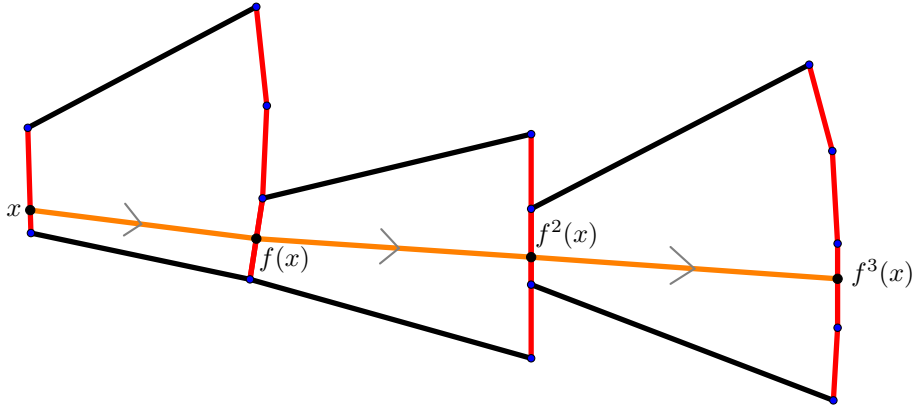

We emphasise that the forward flow is different from the action of $t$. 

In the backward direction, we note the following. 

\begin{lemma}\label{lem: vertical_locally_finite} For any $x_i \in T_i$ with $x_i$ not a singular vertex, its preimage  $f_{i-1}^{-1} (x_i)$ in $T_{i-1}$ is finite. In particular, every vertical edge of $\wt X$ is contained in finitely many $2$-cells.  \end{lemma} 

\begin{proof}
It suffices to prove that each edge of $T$ is contained in finitely many images of edges of $T$ under $f$.
Assume on the contrary that there are infinitely many such edges $e_k$, whose images meet $x_i$.  Since there are finitely many $G$-orbits of edges in $T$, we may assume that all edges $e_k$ are  images of $e_0$ by some elements $g_k\in G$, all different. The map $f$ being a quasiisometry of $T$, all edges $e_k$ are at bounded distances, so the $g_k$ have bounded displacement. This forces that for large $k$, the  $g_k$  and their images by $\phi$ have a contribution in some of the free factors $H_i$ that is larger than the maximal angle of the path $f(e_0)$. However $f(e_k) = f(g_k(e_0)) = \phi(g_k) f(e_0)$. This forces $f(e_k) $ and $f(e_0)$ to be disjoint, a contradiction.\end{proof}

In the case of the preimage of an infinite valence vertex of $T_i$, we have an even clearer picture. {For $j \in \mathbb{N}$, let us denote by $f^{-j}(v)$ the set $(f^j)^{-1}(v)$.}

\begin{lemma}\label{lem: preimage_singular_vertices}
Let $v \in T_i$ be a singular vertex of $T_i$. Then for each $j \in \mathbb{N}$, the {set $f^{-j}(v)$ is finite and} contains a unique singular vertex.
\end{lemma}

\begin{proof} We may assume that $v\in T$. The set $f^{-j}(v)$ lies in $T$. We will first show that there is a unique singular vertex in $f^{-j}(v)$ and then show that the set is finite. 

Assume that it contains two different singular vertices $w_1, w_2$, so that $f^j(w_1) = f^j(w_2) = v$. By equivariance of $f$, the stabilisers $G_{w_1}$ and $G_{w_2}$ of $w_1$ and $w_2$,  are sent by $\phi^j$ inside  the stabiliser of $G_{v}$. But they are also sent onto stabilisers of vertices,  since $\phi$ preserves the conjugacy classes of the $H_k$. Thus $\phi^j(G_{w_1}) = \phi^j(G_{w_2})$, and therefore $G_{w_1} = G_{w_2}$. {In particular, the tree-geodesic between the vertices $w_1$ and $w_2$ is pointwise fixed by $G_{w_1}$, and since $w_1\neq w_2$,  this contradicts the triviality of stabilisers of edges of the tree}.

Assume now that the preimage $f^{-j}(v)$ has infinitely many non-singular points. Denote by $(e_i)_{i\in \mathbb{N}}$ the collection of edges containing a preimage. Since $f$ is a quasiisometry, they are all at a bounded distance from each other. Up to extracting a subsequence, we may assume they are in the same $G$-orbit: write $e_i= g_ie_0$. Being at bounded distance from each other, the displacement of $g_i$ remains bounded. Consider the segment $s_i$ between $e_0$ and $e_i= g_ie_0$, containing both. Taking a subsequence if necessary, we may assume that there is a vertex $w$ for which the $s_i$ have the same prefix until $w$ and then start having angles going to infinity at this vertex.   Consider $f(e_i ) = \phi(g_i) f(e_0)$. The displacement of $\phi(g_i)$ also remains bounded, and the angle of $f(s_i)$ at the image $f(w)$ tends to infinity.  After $w$,  all these images of the $f(s_i)$  are thus disjoint. It follows that all the images $e_i$ are either disjoint, or share possibly  only $f(w)$. In other words,  the $f(e_i \setminus \{w\} )$ are disjoint.   Since at  $w$, angles are arbitrarily large, it is a singular point  contradicting the initial assumption that the $e_i$ each contain a non-singular preimage of a point.
\end{proof}

We may now define the backward flow.
For a point $x_i \in T_i$ the \emph{backward flow}\footnote{level in \cite{hagen_wise_irreducible}.} $\tau_L(x_i)$ from $x_i$ of length $L$ is the union of length $L$ forward flows from each point in $f^{-L}(x_i)$ (\cref{fig:my_bwd_flow}).
Note that $\tau_L(x_i)$ is a rooted tree, rooted at $x_i$. 
{When $x \notin \wt{X}^0$, by \cref{lem: vertical_locally_finite}, $\tau_L(x)$ is a finite rooted tree. In particular, \cite[Proposition 2.5]{hagen_wise_irreducible} gives the following observation.}

\begin{figure}
    \centering
    \begin{tikzpicture}[line cap=round,line join=round,>=triangle 45,x=1cm,y=1cm,scale=.8]
\draw [line width=2pt] (-2,1)-- (1.9803373599999943,-0.3363277600000015);
\draw [line width=2pt,color=red] (1.9803373599999943,-0.3363277600000015)-- (2.6212828999999997,0.85370634);
\draw [line width=2pt,color=red] (2.6212828999999997,0.85370634)-- (2.6412828999999993,2.23370634);
\draw [line width=2pt,color=red] (2.6412828999999993,2.23370634)-- (2.0157685799999943,3.4548127799999984);
\draw [line width=2pt] (2.0157685799999943,3.4548127799999984)-- (-2,3);
\draw [line width=2pt] (-6.01,-0.68)-- (-3.476070520000008,-1.1158146000000015);
\draw [line width=2pt,color=red] (-3.476070520000008,-1.1158146000000015)-- (-2,1);
\draw [line width=2pt,color=red] (-2,3)-- (-3.476070520000008,4.234299619999998);
\draw [line width=2pt] (-3.476070520000008,4.234299619999998)-- (-6,1);
\draw [line width=2pt,color=red] (-1.63,-1.36)-- (-2,1);
\draw [line width=2pt,color=red] (-2,1)-- (-2,3);
\draw [line width=2pt,color=red] (-2,3)-- (-1.350197320000007,6.041291839999999);
\draw [line width=2pt] (-1.350197320000007,6.041291839999999)-- (-6,5);
\draw [line width=2pt,color=red] (-5.956255920000009,6.820778679999998)-- (-6.027118360000009,2.6398947199999983);
\draw [line width=2pt] (-6.027118360000009,2.6398947199999983)-- (-9.60567158000001,3.4548127799999984);
\draw [line width=2pt,color=red] (-9.60567158000001,3.4548127799999984)-- (-9.60567158000001,5.049217679999998);
\draw [line width=2pt] (-9.60567158000001,5.049217679999998)-- (-5.956255920000009,6.820778679999998);
\draw [line width=2pt,color=red] (-7.1609174000000095,2.0021327599999985)-- (-6,1);
\draw [line width=2pt,color=red] (-6,1)-- (-6.01,-0.68);
\draw [line width=2pt,color=red] (-6.01,-0.68)-- (-6.49,-1.86);
\draw [line width=2pt] (-6.49,-1.86)-- (-9.46394670000001,-0.9386585000000015);
\draw [line width=2pt,color=red] (-9.46394670000001,-0.9386585000000015)-- (-9.46394670000001,0.9037649399999985);
\draw [line width=2pt] (-9.46394670000001,0.9037649399999985)-- (-7.1609174000000095,2.0021327599999985);
\draw [line width=2pt,dashed ] (-4.908071449165063,2.3992562490425025)-- (-2.750208762726572,-0.0753569234787027);
\draw [line width=2pt] (-5.99,3.64)-- (-4.908071449165065,2.399256249042503);
\draw [line width=2pt] (-2.750208762726572,-0.07535692347870278)-- (-1.63,-1.36);
\draw [line width=2pt,color=orange] (-9.46394670000001,-0.4071902000000015)-- (-6.0059523809523805,0);
\draw [line width=2pt,color=orange] (-6.0059523809523805,0)-- (-2,2);
\draw [line width=2pt,color=orange] (-2,2)-- (2.63177682097438,1.5777868872322538);
\draw [line width=2pt,dashed,color=orange] (-2,2)-- (-4.14328992325197,3.3792886224683025);
\draw [line width=2pt,color=orange] (-4.14328992325197,3.3792886224683025)-- (-5.997680531321089,4.376726612056287);
\draw [line width=2pt,color=orange] (-5.997680531321089,4.376726612056287)-- (-9.60567158000001,4.2697308399999985);
\draw [line width=0.8pt,color=gray] (-8.223854000000012,4.624043039999998)-- (-7.728417117239134,4.325401236040133);
\draw [line width=0.8pt,color=gray] (-7.728417117239134,4.325401236040133)-- (-8.330147660000012,3.9862810799999986);
\draw [line width=0.8pt,color=gray] (-4.539007120000011,4.234299619999998)-- (-4.333697127687685,3.4817046821533673);
\draw [line width=0.8pt,color=gray] (-4.333697127687685,3.4817046821533673)-- (-4.96418176000001,3.4548127799999984);
\draw [line width=0.8pt,color=gray] (-0.1809670600000086,2.2147200799999984)-- (0.27026915530220547,1.7930519012617063);
\draw [line width=0.8pt,color=gray] (0.27026915530220547,1.7930519012617063)-- (-0.25182950000000864,1.2935083599999986);
\draw [line width=0.8pt,color=gray] (-4.50357590000001,1.1872146999999984)-- (-3.931682991895029,1.0355936325754385);
\draw [line width=0.8pt,color=gray] (-3.931682991895029,1.0355936325754385)-- (-4.25555736000001,0.47859029999999847);
\draw [line width=0.8pt,color=gray] (-7.940404240000012,0.1951405399999985)-- (-7.373288447568269,-0.1610083172681089);
\draw [line width=0.8pt,color=gray] (-7.373288447568269,-0.1610083172681089)-- (-7.940404240000012,-0.5843463000000015);
\draw [line width=0.8pt,color=gray] (-7.940404240000012,-0.5843463000000015)-- (-7.373288447568269,-0.1610083172681089);
\draw [fill=blue] (-2,3) circle (1.5pt);
\draw [fill=blue] (-2,1) circle (1.5pt);
\draw [fill=blue] (1.9803373599999943,-0.3363277600000015) circle (1.5pt);
\draw [fill=blue] (2.6212828999999997,0.85370634) circle (1.5pt);
\draw [fill=blue] (2.6412828999999993,2.23370634) circle (1.5pt);
\draw [fill=blue] (2.0157685799999943,3.4548127799999984) circle (1.5pt);
\draw [fill=blue] (-6,1) circle (1.5pt);
\draw [fill=blue] (-6.01,-0.68) circle (1.5pt);
\draw [fill=blue] (-3.476070520000008,-1.1158146000000015) circle (1.5pt);
\draw [fill=blue] (-3.476070520000008,4.234299619999998) circle (1.5pt);
\draw [fill=blue] (-6,5) circle (1.5pt);
\draw [fill=blue] (-5.99,3.64) circle (1.5pt);
\draw [fill=blue] (-1.63,-1.36) circle (1.5pt);
\draw [fill=blue] (-1.350197320000007,6.041291839999999) circle (1.5pt);
\draw [fill=blue] (-5.956255920000009,6.820778679999998) circle (1.5pt);
\draw [fill=blue] (-6.027118360000009,2.6398947199999983) circle (1.5pt);
\draw [fill=blue] (-9.60567158000001,3.4548127799999984) circle (1.5pt);
\draw [fill=blue] (-9.60567158000001,5.049217679999998) circle (1.5pt);
\draw [fill=blue] (-7.1609174000000095,2.0021327599999985) circle (1.5pt);
\draw [fill=blue] (-6.49,-1.86) circle (1.5pt);
\draw [fill=blue] (-9.46394670000001,-0.9386585000000015) circle (1.5pt);
\draw [fill=blue] (-9.46394670000001,0.9037649399999985) circle (1.5pt);
\draw [fill=black] (-9.46394670000001,-0.4071902000000015) circle (2.5pt);
% \draw[color=black] (-9.258678929090927,0.035305579999997116) node {$C_1$};
\draw [fill=black] (-6.0059523809523805,0) circle (2.5pt);
% \draw[color=black] (-5.804133474545471,0.4171237618181781) node {$D_1$};
\draw [fill=black] (-2,2) circle (2.5pt);
% \draw[color=black] (-1.8041334745454694,2.45348739818181) node {$E_1$};
\draw [fill=black] (2.63177682097438,1.5777868872322538) circle (2.5pt);
\draw[right] (2.7,1.7) node {$x$};
\draw [fill=black] (-5.997680531321089,4.376726612056287) circle (2.5pt);
% \draw[color=black] (-5.804133474545471,4.817123761818168) node {$H_1$};
\draw [fill=black] (-9.60567158000001,4.2697308399999985) circle (2.5pt);
% \draw[color=black] (-9.404133474545473,4.70803285272726) node {$I_1$};
\end{tikzpicture}
    \caption{The backward flow of length $3$ from $x$ (in orange).}
    \label{fig:my_bwd_flow}
\end{figure}
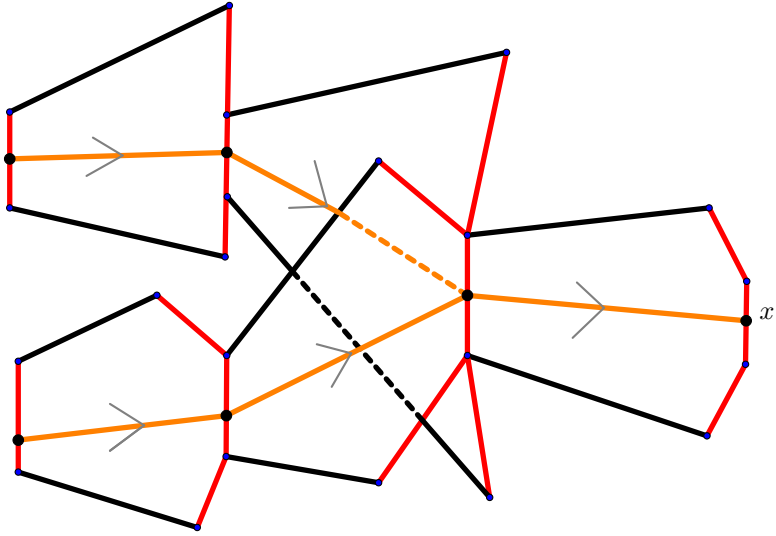

\begin{prop}\label{prop: local_separation_backward_band}
Let $x \notin \wt{X}^0$. Then for any $L \geq 0$, there exists a topological embedding $\tau_L(x) \times [-1,1] \to \wt{X}$ such that $\tau_L(x) \times \{0\}$ maps isomorphically onto $\tau_L(x)$. \qed
\end{prop}

Finally, by \cref{lem: preimage_singular_vertices}, there is a unique singular vertex $v_j$ in each $T_j, j<i$ that is in the backward flow of a singular vertex $v_i$ of $T_i$. Hence, we can construct the \emph{principal flow line} of $v_i$  to be the direct limit of the forward flow rays of $v_j$, for $j \to -\infty$. It is well defined, bi-infinite, as all $v_j$ as above, and all images of $v_i$ by positive iterates of $f$ have the same principal flow line.

\subsection{Geometry of the flow space}

In order to have relative hyperbolicity for the group $\Gamma := G \rtimes_{\phi} \mathbb{Z}$, following \cite{dahmani_li}, we will need the following 
three 
% {two}
additional properties, which we will assume to hold from now on:

\begin{itemize}
\item Any fundamental domain of $T$ contains at least two edges{, i.e., $k+r \geq 3$ or $r\geq 2$}.
    \item The automorphism $\phi$ is \emph{atoroidal relative to $\{H_1, \cdots, H_k\}$}: given any element $g \in G$ such that $g$ is not contained in any conjugate of any $H_i$, then for all $n \in \mathbb{N}$, $[\phi^n(g)] \neq [g]$.
    \item The automorphism $\phi$ has \emph{no twinned subgroups}: given two subgroups $H \neq K$ such that $[H],[K] \in \calH$, then given any $g \in G$ and any $m \in \mathbb{N}$, $\phi^m(H) \neq gHg^{-1}$ whenever $\phi^m(K) = gKg^{-1}$.
\end{itemize}
{\begin{remark}\label{rem: no twins}
    % In \cite{dahmani_li}, an additional property of \emph{no twinned subgroups} is assumed for $\phi$. The property says that given two subgroups $H \neq K$ such that $[H],[K] \in \calH$, then given any $g \in G$ and any $m \in \mathbb{N}$, $\phi^m(H) \neq gHg^{-1}$ whenever $\phi^m(K) = gKg^{-1}$. 
    We in fact do not need to assume the no twinning property, as by \cite[Lemma 3.3]{DMM23}, it automatically holds whenever $\phi$ is relatively fully irreducible and $k+r \geq 3$. We clarify that the case $r \geq 2$ is redundant if $k = 0$ as there are no atoroidal maps of $F_2$. 
\end{remark}}

Let us recall that the \emph{cone-off} of a graph $Z$ over a family of subgraphs $\mathcal{L}$, which we denote here by $\widehat{Z}$, is the graph obtained by adding to $Z$ a vertex $v_L$ for each $L \in \mathcal{L}$ and an edge between each vertex of $L$ and $v_L$. Usually, each such edge is assigned to have length $\frac{1}{2}$. 

{It can help to picture $T$ as equivariantly quasiisometric (by a collapse map) to a cone-off of the Cayley graph of $G$ over the left cosets of the free factors, and to picture $\wt X$ as equivariantly q.i. to a cone-off of the Cayley graph of $\Gamma$ over the left cosets of the same free factors of $G < \Gamma$. Finally, one can see the cone-off of $\wt X$ over principal flow lines as equivariantly q.i. to the cone-off of the Cayley graph of $\Gamma$ over the cosets of the suspensions of the free factors of $G$.}
% Further, $\widehat{Z}$ is said to be \emph{fine} if for each $n \in \mathbb{N}$, every edge of $\widehat{Z}$ belongs to only finitely many simple cycles of length $n$.

\begin{theorem}\label{thm: hyperbolicity_graph_and_cone-off}
The $1$-skeleton $\wt{X}^1$ is $\delta$-hyperbolic for some $\delta >0$. Moreover, the cone-off of $\wt X^{1}$ over the principal flow lines is also hyperbolic.
\end{theorem}

This is essentially proved in \cite{dahmani_li}. Let us cover how to obtain it. We want to use Theorem 4.5 from \cite{mj_reeves}.  However, a little care is in order. The assumption of this theorem is that one has a tree (denoted by $T$ in \cite{mj_reeves}, that we temporarily denote by $T_{MR}$ here) of relatively hyperbolic spaces $S_v, v\in T_{MR}^{(0)}$, with some properties, and the conclusion is that the whole space denoted by $X$ in \cite{mj_reeves}, and that we temporarily denote by $X_{MR}$, is itself hyperbolic relative to maximal cone subtrees. Here the tree $T_{MR}$ is just a bi-infinite line $\mathbb{Z}$ indexing our trees $T_i$, and the spaces $S_i, i\in \mathbb{Z}$ are indeed  the trees $T_i$, all isometric to $T$, on which $G$ acts through $\phi^i$, and the total space $X_{MR}$ is $\wt{X}$.  The different assumptions of the theorem were checked in {Section 2.3 of} \cite{dahmani_li}, and the conclusion is that  $\wt X$ is hyperbolic relative to a collection of quasiconvex lines (the principal flow lines of singular vertices).
It follows that $\wt X$ is hyperbolic itself.  It also implies that its cone-off over those lines is hyperbolic as well. 

Recall that a graph is said to be \emph{fine} if for each $n \in \mathbb{N}$, every edge of the graph belongs to only finitely many simple cycles of length $n$.
Recall also that a finitely generated group $\Gamma$ is hyperbolic relative to a finite collection of finitely generated subgroups $\mathscr{P}$ if the cone-off of a Cayley graph of $\Gamma$ over the left cosets of elements of $\mathscr{P}$ is hyperbolic and fine \cite{bowditch_rel_hyp}.

The main result of \cite{dahmani_li} (Theorem 0.2) was actually a related statement, that the group $G\rtimes_\phi \langle t\rangle$ is itself relatively hyperbolic with respect to the collection $\mathscr{P}$ consisting of mapping tori of the subgroups $H_1, \cdots, H_k$. {Since the stabiliser of a singular vertex in $T_0$ is a conjugate of some $H_i$, each principal flow line is in fact cocompactly stabilised by the suspension of a conjugate of some $H_i$ (this can, for instance, be deduced from \cref{lem: periodic flow line}(2)).}

\subsection{Quasiconvexity and divergence of forward flow rays}

We first {observe, following \cite{hagen_wise_irreducible},} that forward flow rays are uniformly quasiconvex (\cref{prop: quasiconvexity_fwd_ladder}).

From $x_i \in T_i$ to $f^n(x_i)$ any path intersects each tree $T_{i+k}$ and therefore has at least $n$ horizontal edges. However, the forward-path has exactly $n$ horizontal edges and no vertical contribution.  If $D$ is an upper bound to the diameter of $2$-cells, we see that the intersection of the forward flow ray with  $\bigcup_i T_i$  is  a $D$-bilipschitz embedding of $\mathbb{N}$ in the $1$-skeleton of $\wt X$. Since the latter is hyperbolic, we have quasiconvexity. We thus have:

\begin{prop}[Proposition 2.3 in \cite{hagen_wise_irreducible}]\label{prop: quasiconvexity_fwd_ladder}
There exists $\lambda \geq 0$ such that the $1$-skeleton $N(\sigma)^1$ of any forward ladder is $\lambda$-quasiconvex in $\wt{X}^1$. \qed
\end{prop}

Second, two flow rays starting at different points in the same edge of $T_i$ diverge from each other, as stated below.

\begin{lemma}\label{lem: edge_flows_diverge} 
Given $\epsilon >0$, there exists $N>0$ such that for any two points $x,y$ contained in any vertical edge $e_i \subset T_i$ with $d(x,y) \geq \epsilon$, the distance in $T_{i+N}$ between $f^N(x)$ and $f^N(y)$ is at least $e^{100(\delta + \lambda)}$ (and thus the forward rays from $x$ and $y$ diverge).
\end{lemma}

\begin{proof}
    The forward images of a single edge are all legal paths, and  by \cite[Lemma 1.11]{dahmani_li} $f$ applies a uniform stretching factor $>1$ on all legal paths.
\end{proof} 

Finally, two flow rays starting at different points in the star of a vertex either uniformly diverge from each other or fellow-travel forever. More precisely we have the following.

\begin{prop}\label{prop: flowrays_fellow_travel_or_diverge} For all $\epsilon >0$, there exists $N>0$ such that,  for all singular vertices $v\in T_i$, for all edges $e_1\neq e_2$ in $T_i$ starting at $v$, for all $x_1, x_2 $ respectively in $e_1, e_2$, at distance at least $\epsilon$ from $v$, either the distance in $T_{i+N}$ between $f^N(x_1)$ and $f^N(x_2)$ 
% in $T_{i+N}$ 
is at least $e^{100(\delta + \lambda)}$, % in $T_{i+N}$, (repetition) 
or the forward rays fellow-travel until infinity.
\end{prop}

\begin{proof} We fix $\epsilon$ and $v$. For each $e$ edge issued at $v$, and $x \in e$ at distance at least $\epsilon$ from $v$, we know by \cref{lem: edge_flows_diverge} that there is $N_{\epsilon}$ such that, for all $n\geq N_{\epsilon}$,  $f^{n} (x)$ is at distance (in $T_{i+n}$) at least $e^{1000(\delta +\lambda)}$ from the principal flow line of $v$.  

By \cite[Lemma 2.6]{dahmani_li} we know that $f$ induces a quasiisometry on angles: there exists $\theta_1>1$ such that, if $e_1, e_2$ issued at $v$ make an angle $\theta$, then the paths $f(e_1), f(e_2)$ make an angle at least $\theta/\theta_1 - \theta_1$ at their common initial point $f(v)$  (while it is slightly stronger than the stated lemma, the claim is still true with same proof, however, the stated lemma is also sufficient to be used in a similar way). Call $\rho(\theta) = \theta/\theta_1 - \theta_1$.  Calculus ensures that it is possible to find $\theta_0$ such that $\rho^{N_0} (\theta_0) >0$. If $e_1, e_2$ issued at $v$ make an angle greater than $\theta_0$, the paths $f^{N_0} (e_1)$ and $f^{N_0} (e_2)$ make a positive angle at their initial common point $f^{N_0} (v)$. In particular, they do not overlap. It follows that, measured in $T_{i+N_0}$, the distance $d( f^{N_0} (x_1), f^{N_0}(x_2))$ is equal to $d( f^{N_0} (x_1), f^{N_0}(v)) + d(f^{N_0}(v),  f^{N_0}(x_2))$. It is then greater than $2e^{1000(\delta +\lambda)}$.  
Since the two rays are quasiconvex in the hyperbolic space $\wt X$ and have started to diverge after $N_0$ edges, they will diverge onward after that point. 

It remains to treat the case of two edges making an angle less than $\theta_0$. There are finitely many $Stab(v)$-orbits of such pairs of edges. Partition them into classes of those whose forward ray fellow travel, and those whose forward rays exponentially diverge. Since there are finitely many, one can choose $N$ that is suitable for all of those in the second class, and larger than $N_0$.  
\end{proof}

Observe that this does not prevent some points in different edges from having fellow-travelling forward rays, with the same endpoint at infinity. 
However, if the origins of the flow rays are translates of each other by an elliptic element $g$ in $T_i$, they diverge, provided that $g$ is not a torsion element: 

\begin{lemma}\label{lem: divergence_flowrays_elliptic}
For any point $x_i$ in $T_i$, and for each $g \in G$ elliptic in $T_i$ such that $g^nx_i \neq x_i$ for any $n \in \mathbb{N}$, the flow rays from $gx_i$ and $x_i$ do not fellow-travel until infinity.
\end{lemma}

\begin{proof}
Note that the midpoint of the segment of $T_i$ between $x_i$ and $gx_i$ is {fixed by $g$, and is hence} a singular vertex $v_i$. If the flow rays from $x_i$ and $gx_i$ fellow-travel until infinity, then by equivariance, so do the flow rays from $g^nx_i$ and $g^{n'}x_i$, for all $n,n' \in \mathbb{Z}$. But for large enough powers, the segments from $v_i$ to $g^nx_i$ and $x_i$ make arbitrarily large angles. By arguing as in \cref{prop: flowrays_fellow_travel_or_diverge}, we can conclude that the flow rays from $x_i$ and $g^nx_i$ cannot fellow-travel for arbitrarily large powers, a contradiction.
\end{proof}

Henceforth, we will assume that the group $G$ is torsion-free, so the above result always holds for elliptic elements.

Let us also record a basic observation about principal flow lines:

\begin{lemma}\label{lem: principal flow lines diverge}
Let $v_1 \neq v_2$ be two singular vertices in $\wt{X}$. Then either the principal flow lines of $v_1$ and $v_2$ coincide or uniformly diverge from each other in both the forward and backward directions{, i.e., for each $R >0$ there exists $B_R  > 0$ such that for any pair of singular vertices $v_1 \neq v_2$, either the principal flow lines $\Lambda_{v_i}$ through $v_i$ coincide, or $N_R(\Lambda_{v_1}) \cap \Lambda_{v_2}$ has diameter at most $B_R$}. 
\end{lemma}
\begin{proof}
    {Observe that by \cref{lem: preimage_singular_vertices}, for each singular vertex, there exists a unique principal flow line going through this vertex. 
Since the automorphism $\phi$ has no twinned subgroups for the (finite) collection $\mathcal{H}$,
 if two principal flow lines are different, they must uniformly diverge as required.}
\end{proof}

\subsection{Periodic points and forward flow rays}
In the case of certain special points called periodic points, we have no fellow-travelling of flow rays between any translates, which we prove below. We will also show that periodic points {with flow rays (lines, in fact) diverging from every principal flow line} are dense in {edges of $T$ (\cref{lem: periodic_divergence_principal_lines}).}

A point $x \in T$ is a \emph{periodic point} if there exist $g = g_x \in G$ and $n = n_x>0$ such that $f^n(x) = gx$. If $x$ is periodic, then we will call each $x_i \in T_i$ also as a periodic point. We say that the periodic point $x \in T$ has \emph{period $n$} if $f^n(x) = gx$ as above and for all $0 <k < n$, $f^k(x) \neq hx$ for any $h \in G$.

\begin{prop}\label{prop: periodic_flowrays}
For any periodic point $x_i$ in $T_i$, and for each $g \in G$ with $gx_i \neq x_i$, the flow rays from $gx_i$ and $x_i$ diverge.
\end{prop}

\begin{proof}
Observe that by \cref{lem: divergence_flowrays_elliptic}, the statement has to be proved only for loxodromic elements $g$. Since $x_i$ is periodic, there is a minimal positive $k_i$ and $g_i\in G$ such that $g_i x_{i+k_i} = (f^{k_i}(x_i))$ (which is in $T_{i+k_i}$). Thus, the element $g_i t^{k_i}$ sends $x_i$ to $ f^{k_i}(x_i)$ in the flow space. This flow space is hyperbolic, and this element is loxodromic, since for any $x\in T_0$, $d_{\wt{X}^1}( (g_i t^{k_i})^r x, x ) \geq rk_i$ (the distance separating $T_0$ from $T_{rk_i}$).  It follows that the forward flow from $x_i$ remains at bounded distance from  $(g_i t^{k_i})^n x_i)_{n\in \mathbb{N}}$, hence from the axis of $(g_i t^{k_i})$.   Now, from $gx_i$, the forward flow remains close to the axis of  $g(g_i t^{k_i}) g^{-1}$. 

Since the element $g$ is loxodromic on $T_i$, $g$ and $g_i t^{k_i}$  generate a non-elementary subgroup of isometries of $\wt X$, because their axes eventually diverge. It follows that the limit points of $g_i t^{k_i}$ and of $g$ in the boundary of $\wt{X}^1$ cannot be the same. One concludes that $g$ does not fix the limit points of $g_i t^{k_i}$, and $g(g_i t^{k_i}) g^{-1}$ and $(g_i t^{k_i})$ have thus divergent axes. This is the desired conclusion.
\end{proof}

\begin{lemma}\label{lem: irreducibility}
For all edges $e, e'$ in $T$, there exist $n \geq 0$ and $g\in G$ such that $f^n(e)$ contains $ge'$.
\end{lemma}

\begin{proof}
Assume that for some $e,e'$, for every $n$, the path $f^n(e)$ does not contains any $G$-translate of $e'$. Note that this means that no $G$-translate of $f^n(e)$ contains any $G$-translate of $e'$. Let $W$ be the union of all $G$-translates of all paths $f^n(e)$. Then $W$ is a proper subgraph which is both $f$-invariant and $G$-invariant. By the irreducibility of $f$, the quotient of $W$ by $G$ is a (bounded) forest with at most one non-free vertex in each component. Thus, the quotient of each $f^n(e)$ is a segment with at most one non-free vertex. This is not possible as the lengths of $f^n(e)$ are arbitrarily long. 
\end{proof}

\begin{lemma}\label{lem: periodic_dense}
For each $\epsilon >0$ and each closed subinterval $d$ of any edge $e$ in $T$ of length $\epsilon$, there exists a periodic point $x$ in $d$.
\end{lemma}

\begin{proof}
Lemma 1.11 of \cite{dahmani_li} ensures that there exists a growth factor $\lambda > 1$ such that every subsegment of every edge of $T$ expands under $f$ by $\lambda$. Thus, given $d \subset e$, there exists $n$ such that $f^n(d)$ contains an edge. The result then follows from \cref{lem: irreducibility}, by applying Brouwer's fixed point theorem.
\end{proof}

Let us now observe the following useful facts about periodic points. Given a periodic point $x$ of period $n$, the preimage $f^{-n}(x)$ contains a $G$-translate of $x$ (\cref{lem: periodic flow line}(1)). One can then define a \emph{periodic flow line} through $x$ as the direct limit of forward flow rays from the $G$-translates of $x$ at the various $f^{-nk}(x)$. We will show below in \cref{lem: periodic flow line}(2) that a periodic flow line is really \emph{periodic}, i.e., there is an infinite subgroup of $\Gamma$ that stabilises the line and acts cocompactly on the line.

\begin{lemma}\label{lem: periodic flow line}
Let $x \in T$ be a periodic point of period $n$. Then \begin{enumerate}
    \item there exists $g' \in G$ such that $g'x \in f^{-n}(x)$, and
    \item the periodic flow line through $x$ is periodic. {In particular, principal flow lines are periodic and stabilised by suspensions of the relevant free factors of $G$.}
\end{enumerate}
\end{lemma}

\begin{proof}
\begin{enumerate}
\item Let $g \in G$ be such that $f^n(x) = gx$. Observe that, by equivariance of $f$, for $g' = \phi^{-n}(g^{-1})$, $g'x \in f^{-n}(x)$. Indeed, $f^n(\phi^{-n}(g^{-1})x) = \phi^n(\phi^{-n}(g^{-1}))f^n(x) = g^{-1}f^n(x) = x$.
\item Consider the flow segment $\sigma_n(x)$ from $x \in T_0$ to $f_{n-1}\circ\cdots\circ f_0(x) \in T_n$. Recall that the element $t$ acts by shifting indices, and therefore, $f_{n-1}\circ f_{n-2}\circ\cdots\circ f_0(x) = t^n f^n(x)$. This implies that $t^ng(\sigma_n(x))$ is the flow segment from $t^ngx \in T_n$ to $t^ngt^nf^n(x) \in T_{2n}$. Using the facts that $gt^n = t^n\phi^n(g)$ and $gx = f^n(x)$, we conclude that $t^ng(\sigma_n(x))$ is the flow segment from $t^nf^n(x)$ to $t^{2n}\phi^n(g)f^n(x)$. Since $\phi^n(g)f^n(x) = f^n(gx) = f^{2n}(x)$, we have that the concatenation $\sigma_n(x)\cdot t^ng(\sigma_n(x)) = \sigma_{2n}(x)$. It is also easy to check that the flow segment of length $n$ from $t^{-n}g'(x)$ is equal to $(t^ng)^{-1}(\sigma_n(x))$. Thus the periodic flow line through $x$ is the union over all $k \in \mathbb{Z}$ of $(t^ng)^k(\sigma_n(x))$. \qedhere
\end{enumerate}
\end{proof}

\begin{lemma} \label{lem;from_one_to_both_directions} If a periodic flow line is asymptotic to another periodic flow line in one direction, they are asymptotic in both directions. 
\end{lemma}

\begin{proof} We will prove the contrapositive. Assume $\Lambda', \Lambda''$ are periodic flow lines that are non-asymptotic in the forward  direction {(the other case is similar).} Since they are periodic, there are elements $\gamma', \gamma''$ in the group $\Gamma$ {(\cref{lem: periodic flow line}(2))} that respectively fix the endpoints of $\Lambda'$ and $\Lambda''$ in the boundary $\partial \wt X^1$ of the hyperbolic space $\wt X^1$.  In particular,    $\gamma', \gamma''$ cannot be elements of the same parabolic subgroup of $\Gamma$. Assume that the backward directions of $\Lambda', \Lambda''$ converge to the same point of $\partial \wt X^1$. 
It would follow that the commutator  
$[\gamma',\gamma'']$ is in $G$ and has small displacement on infinitely many $T_{-i}, i \in \bbN$,  realised near the ray $\Lambda'$.   Since $\gamma'$ and $\gamma''$ have different axes, they do not commute.  However, by hyperbolicity of the automorphism,   $[\gamma',\gamma'']$  must be elliptic, hence it is  in a unique conjugate of a free factor of the free product $G$. It hence fixes non-free vertices in each $T_{-i}$, at bounded distance from $\Lambda'$, and all in the same principal flow line. It then follows that both $\Lambda'$ and $\Lambda''$ are asymptotic to the same principal flow line, and therefore that the elements $\gamma', \gamma''$ preserving them fix a point fixed by a parabolic group, thus ensuring that they are in the same parabolic subgroup, contradicting what we previously had.  \end{proof}

Given two flow lines that intersect the tree $T_0$ in $x, y$, we say that the two flow lines are separated by the segment $[x,y]$ in $T_0$. The next statement says that a periodic flow line either diverges from every principal flow line, or it is asymptotic to a principal flow line which is Hausdorff-close to it in terms of both distance and angle.

\begin{lemma}\label{lem;only_finitely_many_traps}
  {There exist constants $\delta_0, \theta_0$ such that the following holds: 
  Let $\Lambda$ be a periodic flow line in $\wt X$ and $x$ be its intersection with $T_0$.  
  If a principal flow line $\Lambda'$, with intersection $y$ at $T_0$, is asymptotic to $\Lambda$, then $d_{T_0}(x,y)\leq \delta_0$.
 Further, there exists a principal flow line $\Lambda''$ asymptotic to $\Lambda, \Lambda'$, such that if $z \in \Lambda''\cap T_0$, then {for every vertex $v$ in the interior of the segment $[x,z]_{T_0}$} the angle {subtended by $[x,z]_{T_0}$ at $v$} 
 % of $[x,z]_{T_0}$ in $T_0$ 
 is less than $\theta_0$.}
\end{lemma}

\begin{proof}
Consider a geodesic $[x,y]_{\wt X^{1}}$ in $\wt X^1$. By asymptoticity of the bi-infinite lines $\Lambda$ and $\Lambda'$ in the hyperbolic space $\wt X^1$,  its length is at most $\delta$.  Denote by $e_1, \dots, e_r$ its consecutive vertical edges, let $T_{k_i}$ be the tree containing $e_i$, and write $e_i= (v_i^{k_i}, w_i^{k_i})$. Observe that $r\leq \delta$ and $|k_i| \leq \delta$ for all $i$.
     
Since $e_1$ is the first vertical edge, $x\in f^{-k_1}(v_1^{k_1} )$, and similarly, $y\in f^{-k_r} (w_r^{k_r})$.
Take, for all $i<r$, $x_i \in f^{-k_i} (w_i^{k_i})$. Such a point $x_i$ is in $T_0$ and is in $f^{-k_{i+1}} (v_{i+1}^{k_{i+1}})$. Therefore, for all $i<r$, $d_{T_0} (x_i, x_{i+1}) \leq K^{|k_{i+1}|}$, where $K$ is the Lipschitz constant of $f$. Since for all $i$, $|k_{i+1}|\leq \delta$, and $r\leq \delta$, it follows that $d_{T_0}(x,y) \leq \delta K^\delta$.

 {For the second part of the statement, we first observe: By \cref{lem: divergence_flowrays_elliptic}, for any edge $e$ incident to a vertex $w$, and any $g \in \mathrm{stab}(w)$, $e$ and $ge$ form a legal turn and so their iterates under $f$ subtend a positive angle at all $f^n(w)$. Since there are finitely many orbits of edges incident to $w$, there are only finitely many edges incident to $w$ that make an illegal turn with $e$. Since there are finitely many orbits of vertices, there exists a maximal angle, which we call $\theta_0-1$, such that any illegal turn between any pair of edges at any vertex subtends an angle of at most $\theta_0-1$.}
 
{Assume now that there exists $v\in [x,y]_{T_0}$ such that $Ang_v [x,y]_{T_0} \geq \theta_0$. We choose $v$ to be closest to $x$ with this property. Note that, after changing $\theta_0$ if necessary, $v$ is a singular vertex.} Let $\Lambda_v$  be the associated principal flow line.

For all $n \in \bbN$, let us denote by $x_{-n}$ the point $f^{-n} (x) \cap \Lambda$ and by $y_{-n}$ the point $f^{-n}(y) \cap \Lambda'$. Let us also denote: $x_n = f^n(x), y_n = f^n(y), v_n = f^n(v) $ and $v_{-n}$ the point $f^{-n} (v) \cap \Lambda_v$. 

Since $\theta_0 -1$ is the maximal angle at an illegal turn,
for all $n\in \bbN$,  $Ang_{v_n} [x_n, y_n]_{T_n} \geq 1$.
In particular, 
$v_n$ is between $x_n, y_n$ in $T_n$ for all $n \in \bbN$, and so the principal flow line $\Lambda_v$ is asymptotic to $\Lambda$ in one and hence both directions. 
\end{proof} 

\begin{lemma}\label{lem: periodic_divergence_principal_lines}
   For each $\epsilon > 0$ and each closed subinterval $d$ of any edge $e$ in $T$ of length $\epsilon$, there exists a periodic point $x$ in $d$ such that the periodic flow line through $x$ diverges from every principal flow line in both the forward and backward directions.
\end{lemma}

\begin{proof}
There are finitely many orbits of vertices, hence finitely many orbits of periodic lines containing a vertex. Let $n_0$ be the maximum of their period. If a subinterval $d$ of $e$ is given, one may take a sufficiently small subinterval $d'$ such that its images by $f^k, k=1, \dots n_0$  do not meet any vertex. This, together with the periodicity of periodic points guarantees that any periodic point in this subinterval will have a periodic flow line that misses all vertices. 

There are only finitely many vertices in the tree $T_0$ that are at distance $\leq \delta K^\delta+1$ from $d$ with a path that makes no angle greater than $\theta_1$. Let $N$ be  this number.   This produces $N$ principal flow lines. 

Consider $2N+1$ periodic points in the subinterval $d'$. Since the interval between any two of them is a legal path, their forward flow rays diverge from each other. Therefore, at least one of these  periodic points, say $x$,  belongs to a periodic flow line 
that  diverges from the $N$ principal flow lines in both directions. By \cref{lem;only_finitely_many_traps}, we know that it diverges from all principal flow lines. 
\end{proof}

\section{Relative cubulation}\label{sec: rel_cubul}
In this section, we will recall the notion of relative cubulation and the boundary criterion for relative cubulation. We will construct walls in \cref{sec: walls} for the flow space $\wt X$ and show in \cref{sec: wall_saturations_cut} and \cref{sec: separation_bowditch} that the stabilisers in $\Gamma = G \rtimes_{\phi} \mathbb{Z}$ of ``wall saturations'' satisfy the hypotheses required for the boundary criterion to hold.

\begin{defn}[\cite{relative_cubulation}]\label{defn: rel_cubulation}
Let $\Gamma$ be hyperbolic relative to a collection of subgroups $\mathscr{P}$. Then $(\Gamma, \mathscr{P})$ is \emph{relatively cubulated} if there exists a $\cat$ cube complex $C$ such that
\begin{enumerate}
\item there is a cubical, cocompact action of $\Gamma$ on $C$,
\item each element of $\mathscr{P}$ is elliptic, and
\item  the stabiliser of any cube of $C$ is either finite, or conjugate to a finite index subgroup of an element of $\mathscr{P}$.
\end{enumerate}
We call such an action of $\Gamma$ on $C$ a \emph{relatively geometric} action.
\end{defn}

Let us recall a few definitions before stating the boundary criterion for relative cubulation. A subgroup $H \leq \Gamma$ is \emph{relatively quasiconvex} (with respect to $\scrP$) if there exists a constant $K$ such that given $h_1,h_2 \in H$, any geodesic path $\gamma$ between $h_1$ and $h_2$ in the cone-off $\widehat{\Gamma}$ (with respect to $\scrP$) is such that every non-cone vertex of $\gamma$ is at $\Gamma$-distance at most $K$ from a vertex of $H$ (see \cite{osin_rel_hyp,hruska_relative_quasiconvex}).

A subgroup $H$ of $(\Gamma, \scrP)$ is \emph{full} if the intersection of any conjugate of $H$ with any element $P \in \scrP$ is either finite or is of finite index in $P$.

A subgroup $H$ of $\Gamma$ is a \emph{codimension-$1$ subgroup} if for some $r \geq 0$, $\Gamma \setminus \calN_r(H)$ contains at least two orbits of components that are not contained in any finite neighbourhood of $H$.

Let us also recall a definition (of the many equivalent ones) of the \emph{Bowditch boundary} of a relatively hyperbolic group $(\Gamma, \scrP)$: it is a compact, metrizable space $\partial_B(\Gamma, \scrP)$ such that \begin{enumerate}
    \item $\Gamma$ acts properly discontinuously on the space of distinct triples of $\partial_B(\Gamma, \scrP)$, 
    \item each point of $\partial_B(\Gamma, \scrP)$ is either a conical point or a bounded parabolic point, and
    \item the stabiliser of a bounded parabolic point is a conjugate of an element of $\scrP$.
\end{enumerate} 

As a set, $\partial_B(\Gamma, \scrP)$ is the union of the Gromov boundary of the cone-off $\widehat{\Gamma}$ relative to $\scrP$ and the set of cone vertices of $\widehat{\Gamma}$. We refer to \cite[Section 9]{bowditch_rel_hyp} for more details{, in particular for the definitions of conical and bounded parabolic points that will not be needed here}.

The following theorem is implicit in Bergeron and Wise's \cite{bergeron_wise} and explicitly stated and proved in Einstein and Groves' \cite{relative_cubulation}.

\begin{theorem}[Boundary criterion for relative cubulation]\label{thm: bdry_criterion}
Let $(\Gamma, \mathscr{P})$ be relatively hyperbolic with one-ended parabolics. Suppose that for each pair of distinct points $u,v \in \partial_B (\Gamma, \mathscr{P})$, there exists a full relatively quasiconvex codimension-$1$ subgroup $H$ of $\Gamma$ such that $u,v$ lie in {$H$-}distinct components of $\partial_B(\Gamma, \mathscr{P}) \setminus \Lambda H$. Then there exists a finite collection of full relatively quasiconvex codimension-$1$ subgroups such that the action of $\Gamma$ on the dual cube complex is relatively geometric.
\end{theorem}
We note that since our parabolics are suspensions of the infinite free factors of $G$, they are always one-ended.

\section{Walls in the flow space} \label{sec: walls}
\subsection{Constructing immersed walls} \label{sec: immersed_walls}
Our construction of walls starts with  that of Hagen--Wise in \cite{hagen_wise_irreducible}, although they work in a locally finite setup. However, the stabilisers of walls thus obtained would not be full subgroups in general, so we later ``saturate'' these walls to obtain full codimension-1 subgroups.

Let us explain the construction of Hagen--Wise. Fix a fundamental domain $D \subset T$ for the $G$-action on $T${, i.e., a smallest subtree that contains exactly one edge from each orbit of edges}. In order to construct immersed walls in $\wt X$, we need the following data: a choice of a \emph{tunnel length} $L \in \mathbb{N}$ and a choice of a subinterval of each open edge of $D$, called a \emph{primary bust}. 

Choose the tunnel length $L \in \mathbb{N}$. The immersed wall will be first constructed in $\wt X_L$ and then pushed to $\wt X$ via $\varrho_L : \wt X_L \to \wt X$.

Choose a closed subinterval in the interior of each edge of $D$. This choice extends equivariantly to a choice of a subinterval $d_k$ in every edge $e_k$ of $T$. 
For each $i \in \mathbb{Z}$, the copy $d_{k,Li}$ in $T_{Li}$ of each $d_k \subset e_k$ is a \emph{primary bust}.  
Note that $f^{-L}(d_k)$ is a union of finitely many subintervals $\{d_{kj}\} \subset T$, by \cref{lem: vertical_locally_finite}. 
For each $i \in \mathbb{Z}$, the copy $d_{kj,Li}$ in $T_{Li}$ of each $\{d_{kj}\}$ is a \emph{secondary bust}.
We will always choose primary busts satisfying the following: 

\begin{lemma} \label{lem: choice_of_busts}
Let $L \in \mathbb{N}$ and $\epsilon >0$. Let $\{x_k\}$ be a collection of points in $D$, with exactly one point in any edge of $D$.  The collection of primary busts in $\wt X_L$ can be chosen so that: 
\begin{enumerate}
    \item every primary bust is disjoint from every secondary bust in the collection;
    \item the primary busts in $D$ lie in the $\epsilon$-neighbourhood of $\{x_k\}$;
    \item the flow rays from the endpoints of the primary busts do not meet any vertex of $\wt X_L$;
    \item if the flow ray from $x_k$ does not meet a vertex, then one of the endpoints of the primary bust $d_k$ in $D$ can be chosen to be $x_k$;
    \item the $f^L$-image of any primary bust does not contain two points in the same $G$-orbit;
    \item if $f^L(x_k) \neq f^L(x_j)$ whenever $k \neq j$, then $f^L(d_k) \cap f^L(d_j) = \emptyset$.
\end{enumerate}
\end{lemma}
{The lemma is a reproduction of \cite[Lemma 3.5]{hagen_wise_irreducible}, whose proof can be replicated in our case. Since there are several properties to check, the proof is long, but not difficult, based on Brouwer's fixed point theorem.  We thus refer the reader to \cite{hagen_wise_irreducible} for the proof. In fact, (3) can be obtained either by the proof of Hagen and Wise \cite{hagen_wise_irreducible}, or by applying  \cref{lem: periodic_divergence_principal_lines}.} 
% will be ensured by choosing bust endpoints to be arbitrarily close to periodic points satisfying \cref{lem: periodic_divergence_principal_lines}.

Denote by $T_{Li + \frac{1}{2}} = T'_{Li} \subset \wt X_L$ the parallel copy of $T$ at distance $\frac{1}{2}$ from $T_{Li}$ and distance $L - \frac{1}{2}$ from $T_{(i+1)L}$. We will denote by $d'_{k,Li}$ and $d'_{kj,Li}$ respectively, the copies of $d_{k,Li}$ and $d_{kj,Li}$ in $T'_{Li}$. We refer the reader to \cref{fig:wall} for an illustration of what follows.

\begin{enumerate}
    \item Consider the subspace $\mathbf{C}_{Li}$ of $T'_{Li}$ obtained by removing each open primary bust and each open secondary bust in $T'_{Li}$. A component $\mathbf{N}_{Li}$ of $\mathbf{C}_{Li}$ is a \emph{nucleus}. Note that each nucleus is either a subinterval of some edge or contained in the star of a vertex.
    \item The endpoints of the nuclei are endpoints of either primary busts or of secondary busts. To each endpoint of each secondary bust $d'_{kj,L(i-1)}$ in a nucleus, glue a forward flow segment of length $L-\frac{1}{2}$. Note that this flow segment ends at the primary bust $d_{k,Li}$. 
    The union of all flow segments which end at a common endpoint (an endpoint of some $d_{k,Li}$) is a \emph{level} $\mathbf{L}$. Such an endpoint is a \emph{forward endpoint}. 
    \item Join the endpoints $d^{\pm}_{k,Li}$ of each primary bust $d_{k,Li}$ to the endpoints $d'^{\mp}_{k,Li}$ (note the change in order) of its parallel copy $d'_{k,Li}$ by segments called \emph{slopes}. A slope is denoted by $\mathbf{S}$.
\end{enumerate}

{The graph $\overline{W}_L$ is constructed as a quotient of the disjoint union of the nuclei, levels and slopes, with the gluing performed in a natural way: an endpoint of a primary bust of a nucleus is glued to an endpoint of the corresponding slope, an endpoint of a secondary bust in a nucleus is glued to the initial point of the relevant flow segment of a level, while a forward endpoint is glued to the relevant endpoint of a slope.} 
There exists a non-combinatorial immersion of the graph $\overline{W}_L$ in $\wt X_L$, with loss of injectivity at midpoints of slopes. Let $\overline{W}$ be the graph obtained from $\overline{W}_L$ by `folding' its levels according to $\varrho_L$ so that the following diagram commutes:

\begin{center}
\begin{tikzcd}
\overline{W}_L \arrow[r] \arrow[d]
& \wt X_L \arrow[d, "\varrho_L"] \\
\overline{W} \arrow[r]
& \wt X
\end{tikzcd}    
\end{center}
A component $W_u$ of $\overline{W}$ is an \emph{immersed wall}.\footnote{The subscript $u$ in $W_u$ denotes ``unsaturated'', as we will soon saturate $W_u$ in order to obtain a full codimension-1 subgroup.}

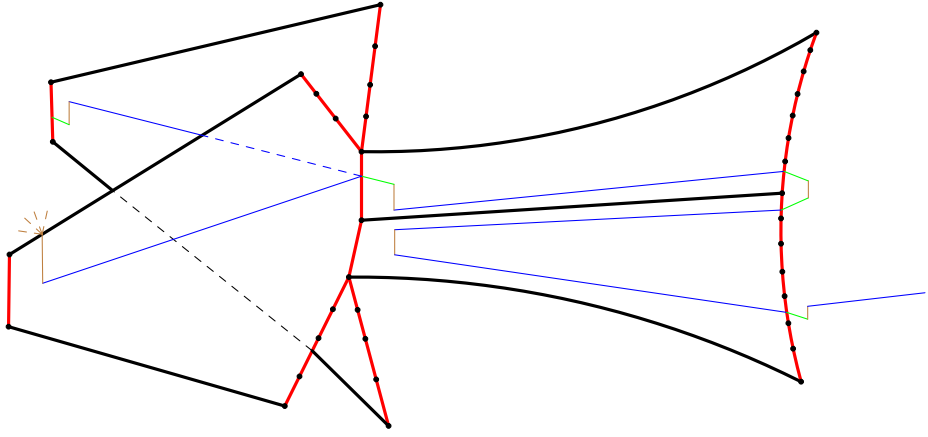
\begin{figure}
    \centering
    \begin{tikzpicture}[scale = .6,line cap=round,line join=round,>=triangle 45,x=1cm,y=1cm]

\draw [line width=1.2pt, color=red] (-0.52,2.72)-- (-0.52,1.2);
\draw [line width=1.2pt, color=red] (-0.52,1.2)-- (-0.8,-0.06);
\draw [line width=1.2pt] (-0.52,1.2)-- (8.8,1.8);
\draw [shift={(-0.19777777777777658,22.053333333333324)},line width=1.2pt]  plot[domain=4.695723856670749:5.241356869229838,variable=\t]({1*19.336018332073213*cos(\t r)+0*19.336018332073213*sin(\t r)},{0*19.336018332073213*cos(\t r)+1*19.336018332073213*sin(\t r)});
\draw [shift={(-0.6231307309405295,-22.094124967251766)},line width=1.2pt]  plot[domain=1.1077142286854866:1.578823215572557,variable=\t]({1*22.034834826038637*cos(\t r)+0*22.034834826038637*sin(\t r)},{0*22.034834826038637*cos(\t r)+1*22.034834826038637*sin(\t r)});
\draw [shift={(21.45194178082192,0.9601472602739729)},line width=1.2pt, color=red]  plot[domain=2.7872249919041923:3.4081614624309373,variable=\t]({1*12.679786411834097*cos(\t r)+0*12.679786411834097*sin(\t r)},{0*12.679786411834097*cos(\t r)+1*12.679786411834097*sin(\t r)});
\draw [line width=1.2pt, color=red] (-0.52,2.72)-- (-1.86,4.44);
\draw [line width=1.2pt, color=red] (-0.52,2.72)-- (-0.1,5.98);
\draw [line width=1.2pt, color=red] (-0.8,-0.06)-- (-2.22,-2.92);
\draw [line width=1.2pt, color=red] (-0.8,-0.06)-- (0.08,-3.36);
\draw [line width=1.2pt] (-1.86,4.44)-- (-8.32,0.44);
\draw [line width=1.2pt, color=red] (-8.32,0.44)-- (-8.34,-1.16);
\draw [line width=1.2pt] (-8.34,-1.16)-- (-2.22,-2.92);
\draw [line width=1.2pt] (0.08,-3.36)-- (-1.6117156008892224,-1.6948638158754759);
\draw [  dashed] (-1.6117156008892224,-1.6948638158754759)-- (-6.032867545676891,1.8561810862681796);
\draw [line width=1.2pt] (-6.032867545676891,1.8561810862681796)-- (-7.36,2.94);
\draw [line width=1.2pt, color=red] (-7.36,2.94)-- (-7.4,4.26);
\draw [line width=1.2pt] (-7.4,4.26)-- (-0.1,5.98);
\draw [ color=blue] (0.2148309541697958,0.9897520661157008)-- (8.780717180349555,1.4260341286407408);
\draw [ color=blue] (0.19980465815176424,1.4255146506386152)-- (8.841394994453214,2.283436041507512);
\draw [ color=green] (8.780717180349555,1.4260341286407408)-- (9.380871525169045,1.695987978963182);
\draw [ color=green] (8.841394994453214,2.283436041507512)-- (9.380871525169045,2.07164537941397);
\draw [ color=brown] (9.380871525169045,2.07164537941397)-- (9.380871525169045,1.695987978963182);
\draw [ color=brown] (0.19980465815176424,1.4255146506386152)-- (0.19980465815176335,1.9965138993238134);
\draw [ color=brown] (0.2148309541697958,0.9897520661157008)-- (0.2148309541697949,0.43377911344853415);
\draw [ color=blue] (0.2148309541697949,0.43377911344853415)-- (8.901063945924673,-0.8433074863769174);
\draw [ color=green] (8.901063945924673,-0.8433074863769174)-- (9.365845229150997,-0.9937190082644612);
\draw [ color=brown] (9.365845229150997,-0.9937190082644612)-- (9.365845229150997,-0.7082193839218622);
\draw [  color=blue] (9.365845229150997,-0.7082193839218622)-- (11.965394440270453,-0.4076934635612316);
\draw [ color=green] (0.19980465815176335,1.9965138993238134)-- (-0.52,2.1768294515401916);
\draw [ color=blue] (-0.52,2.1768294515401916)-- (-7.58382,-0.19733);
\draw [  dashed,color=blue] (-0.52,2.1768294515401916)-- (-4.051438033936354,3.083072424807212);
\draw [ color=blue] (-4.051438033936354,3.083072424807212)-- (-6.997791134485342,3.8297220135236603);
\draw [ color=brown] (-6.997791134485342,3.8297220135236603)-- (-6.997791134485342,3.318827948910588);
\draw [ color=green] (-6.997791134485342,3.318827948910588)-- (-7.376485611287644,3.4840251724922227);
\draw [ color=brown] (-7.58382,-0.19733)-- (-7.595587973547305,0.8885523383608028);
\draw [  dashed,color=brown] (-7.595587973547305,0.8885523383608028)-- (-8.094710743801643,1.3503831705484575);
\draw [  dashed,color=brown] (-7.595587973547305,0.8885523383608028)-- (-7.809211119459045,1.4856198347107412);
\draw [  dashed,color=brown] (-7.595587973547305,0.8885523383608028)-- (-7.4485800150262875,1.560751314800899);
\draw [  dashed,color=brown] (-7.595587973547305,0.8885523383608028)-- (-8.290052592036053,0.9897520661157008);

\draw [fill=black] (-0.52,2.72) circle (1.5pt);
\draw [fill=black] (-0.52,1.2) circle (1.5pt);
\draw [fill=black] (-0.8,-0.06) circle (1.5pt);
\draw [fill=black] (8.8,1.8) circle (1.5pt);
\draw [fill=black] (9.22,-2.38) circle (1.5pt);
\draw [fill=black] (9.56,5.36) circle (1.5pt);
\draw [fill=black] (9.423604362828659,4.9721395843737906) circle (1.5pt);
\draw [fill=black] (9.276581930151853,4.500994045496966) circle (1.5pt);
\draw [fill=black] (9.14182335724371,3.9995498080147316) circle (1.5pt);
\draw [fill=black] (9.033516373713434,3.5213371903289667) circle (1.5pt);
\draw [fill=black] (8.941100806379133,3.023116373454311) circle (1.5pt);
\draw [fill=black] (8.866507102648233,2.504108856673259) circle (1.5pt);
\draw [fill=black] (8.775256123400037,1.2405473184889635) circle (1.5pt);
\draw [fill=black] (8.77525285392669,0.6798950516038189) circle (1.5pt);
\draw [fill=black] (8.80422682583122,0.05887600974101881) circle (1.5pt);
\draw [fill=black] (8.854539263970649,-0.482916590760307) circle (1.5pt);
\draw [fill=black] (8.937946667498558,-1.083601685524325) circle (1.5pt);
\draw [fill=black] (9.04323604851519,-1.647726490265102) circle (1.5pt);
\draw [fill=black] (-1.86,4.44) circle (1.5pt);
\draw [fill=black] (-0.1,5.98) circle (1.5pt);
\draw [fill=black] (-2.22,-2.92) circle (1.5pt);
\draw [fill=black] (0.08,-3.36) circle (1.5pt);
\draw [fill=black] (-8.32,0.44) circle (1.5pt);
\draw [fill=black] (-8.34,-1.16) circle (1.5pt);
\draw [fill=black] (-7.36,2.94) circle (1.5pt);
\draw [fill=black] (-7.4,4.26) circle (1.5pt);
\draw [fill=black] (-1.5204055532183423,4.0041026503996635) circle (1.5pt);
\draw [fill=black] (-1.092417332772402,3.4547446360959193) circle (1.5pt);
\draw [fill=black] (-0.41919067012217703,3.5024724176231024) circle (1.5pt);
\draw [fill=black] (-0.32884635320251765,4.20371640133284) circle (1.5pt);
\draw [fill=black] (-0.21855164753794898,5.059813402443539) circle (1.5pt);
\draw [fill=black] (-0.6061410788381743,-0.7869709543568466) circle (1.5pt);
\draw [fill=black] (-0.43485477178423243,-1.4292946058091285) circle (1.5pt);
\draw [fill=black] (-0.19485477178423238,-2.3292946058091286) circle (1.5pt);
\draw [fill=black] (-1.1540251078854453,-0.7730364848960376) circle (1.5pt);
\draw [fill=black] (-1.4722298940761083,-1.4139278148293448) circle (1.5pt);
\draw [fill=black] (-1.893828952530404,-2.263063946645743) circle (1.5pt);

\end{tikzpicture}
    \caption{A part of a wall in $\wt X_L$. Brown paths are nuclei, blue paths are levels and green paths are slopes.}
    \label{fig:wall}
\end{figure}

\begin{remark} \label{rmk: local_homeo_wall}
Note that the immersion $W_u \to \wt X$ extends to a local homeomorphism $W_u \times [-1,1] \to \wt X$, with $W_u$ identified with $W_u \times \{0\}$. Indeed, this is because the embedding of each nucleus, level and slope of $W_u$ in $\wt X$ extends to a local homeomorphism of the above type.  
\end{remark}

We will denote by $N(W_u)^1 \subset \wt{X}^1$ the $1$-skeleton of the smallest subcomplex of $\wt X$ that contains $W_u$. When the tunnel length $L$ is sufficiently large, it turns out that $N(W_u)^1 \hookrightarrow \wt{X}^1$ is a quasiconvex embedding (\cref{prop: approximations_quasiconvex}) and that
$W_u$ separates $\wt X$ into exactly two components (\cref{prop: long_tunnels_walls}). Thus the local homeomorphism of \cref{rmk: local_homeo_wall} is in fact a homeomorphism in that case.

\subsection{Wall saturations} \label{sec: saturations}
Let $W_u$ be an immersed wall as above, which is a component of the graph $\overline{W}$. Fix a \emph{saturation length} $M \in L\mathbb{N}$. We will define below the $M$-saturation $W$ of $W_u$ as a union of components of $\overline{W}$ and principal flow lines. We first start with a definition:

Let $v \in T_{Li}$ be a singular vertex. 
% Fix $M \in L\mathbb{N}$. 
The \emph{$M$-saturation of $v$} is the set of singular vertices $\{f^{Mk}(v) \in T_{Mk + Li}\}_{k \in \mathbb{Z}}$. 
(Recall that for each $k \in \mathbb{Z}$, there is a unique singular vertex $f^{Mk}(v)$, by \cref{lem: preimage_singular_vertices}.)
By abuse of notation, for a vertex $v$ in a ``fractional'' copy $T_{\frac{1}{2}+ Li}$, we will denote by the $M$-saturation of $v$ the intersection points of the principal flow line of $v$ in the trees $T_{\frac{1}{2}+ Mk + Li}$.

Let $v \in T_{\frac{1}{2}+ Li}$ be a singular vertex in a nucleus of $\overline{W}$. Denote by $W(v) (\subset \overline{W})$, the smallest union of components of $\overline{W}$ such that each vertex in the $M$-saturation of $v$ is contained in the nucleus of some component of $W(v)$.

Let $W_u$ be an immersed wall. Let $W'_u \subset \overline{W}$ be the smallest subgraph of $\overline{W}$ satisfying the following two conditions: \begin{itemize}
    \item $W_u \subset W'_u$, and
    \item for each singular vertex $v \in W'_u$, either the principal flow line of $v$ intersects a component infinitely many times, or $W(v) \subset W'_u$ (but not both).
\end{itemize}

The \emph{$M$-saturation $W$ of $W_u$} is defined as the union of $W'_u$ with the principal flow lines of all singular vertices of $W'_u$. We refer the reader to \cref{fig:saturation} for an illustration.

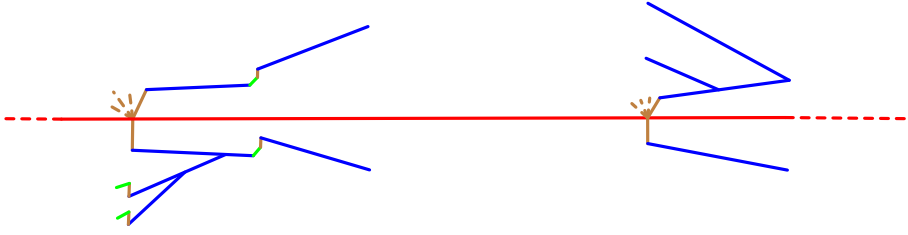
\begin{figure}
    \centering
    \begin{tikzpicture}[line cap=round,line join=round,>=triangle 45,x=.3cm,y=.3cm]
% \clip(-5.566506183189663,-16.498135978303893) rectangle (37.49995534868198,19.012805986572584);
\draw [line width=1.2pt,color=brown] (2.16,1.02)-- (2.14,-0.38);
\draw [line width=1.2pt,dashed,color=brown] (2.16,1.02)-- (1.995541602072372,2.3294838855185493);
\draw [line width=1.2pt,color=brown] (2.16,1.02)-- (2.7668829377735795,2.3080577373046274);
\draw [line width=1.2pt,dashed,color=brown] (2.16,1.02)-- (0.9456603395901725,1.7081255873148051);
\draw [line width=1.2pt,dashed,color=brown] (2.16,1.02)-- (1.3099048592268538,2.200926996235016);
\draw [line width=1.2pt,color=blue] (2.14,-0.38)-- (7.502341371589709,-0.6227562529966207);
\draw [line width=1.2pt,color=blue] (1.989806249202907,-2.426181191440244)-- (6.276784081174258,-0.5672745753044475);
\draw [line width=1.2pt,color=blue] (1.9667742019836376,-3.669911741280784)-- (4.47214215854439,-1.3497980548322286);
\draw [line width=1.2pt,color=brown] (1.989806249202907,-2.426181191440244)-- (2.0128382964221765,-1.850380010958512);
\draw [line width=1.2pt,color=brown] (1.9667742019836376,-3.669911741280784)-- (1.989806249202907,-3.1171426080183213);
\draw [line width=1.2pt,color=green] (2.0128382964221765,-1.850380010958512)-- (1.4370371159404411,-2.0346363887126664);
\draw [line width=1.2pt,color=green] (1.989806249202907,-3.1171426080183213)-- (1.48310121037898,-3.393527174649553);
\draw [line width=1.2pt,color=blue] (2.7668829377735795,2.3080577373046274)-- (7.348036353576131,2.5081271235405667);
\draw [line width=1.2pt,color=green] (7.348036353576131,2.5081271235405667)-- (7.687085219378824,2.8552485813861814);
\draw [line width=1.2pt,color=green] (7.502341371589709,-0.6227562529966207)-- (7.823429688807945,-0.24423989539686747);
\draw [line width=1.2pt,color=brown] (7.687085219378824,2.8552485813861814)-- (7.695157811421746,3.2185152233176386);
\draw [line width=1.2pt,color=brown] (7.823429688807945,-0.24423989539686747)-- (7.837730778040094,0.17049169233547226);
\draw [line width=1.2pt,color=blue] (7.695157811421746,3.2185152233176386)-- (12.584547970066836,5.105298249260805);
\draw [line width=1.2pt,color=blue] (7.837730778040094,0.17049169233547226)-- (12.653644111724645,-1.2515467832575113);
\draw [line width=1.2pt,color=red] (-1,1)-- (31.07659529460902,1.0700271734055748);
\draw [line width=1.2pt,dashed,color=red] (-1,1)-- (-3.52967571461338,1.0195807142084583);
\draw [line width=1.2pt,dashed,color=red] (31.07659529460902,1.0700271734055748)-- (36.62570580629191,1.0195807142084583);
\draw [line width=1.2pt,color=brown] (24.976854682075437,1.0567106854903283)-- (24.97893094162056,-0.08383160547001854);
\draw [line width=1.2pt,color=brown] (24.976854682075437,1.0567106854903283)-- (25.515266792779123,1.944327069566578);
\draw [line width=1.2pt,dashed,color=brown] (24.976854682075437,1.0567106854903283)-- (24.672591830310903,1.8326472552635618);
\draw [line width=1.2pt,dashed,color=brown] (24.976854682075437,1.0567106854903283)-- (24.27663612505475,1.6905093097869957);
\draw [line width=1.2pt,dashed,color=brown] (24.976854682075437,1.0567106854903283)-- (25.068547535567056,1.9341743591753946);
\draw [line width=1.2pt,color=blue] (25.515266792779123,1.944327069566578)-- (31.25026790672866,2.723581408250297);
\draw [line width=1.2pt,color=blue] (24.97893094162056,-0.08383160547001854)-- (31.168959447252302,-1.260533106091253);
\draw [line width=1.2pt,color=blue] (28.126335915169022,2.2991111230999497)-- (24.908208067572687,3.699282921966595);
\draw [line width=1.2pt,color=blue] (31.25026790672866,2.723581408250297)-- (24.989516527049044,6.13853670625734);
\end{tikzpicture}
    \caption{A part of a wall saturation. The red line is a principal flow line, with two unsaturated components (on the left and right) attached.}
    \label{fig:saturation}
\end{figure}

\begin{lemma}\label{lem: saturations are connected}
Let $W_u$ be an immersed wall. For any $M \in L\mathbb{N}$, the $M$-saturation $W$ of $W_u$ is a connected graph. 
\end{lemma}

\begin{proof}We will prove the lemma by showing that a certain connected subgraph containing $W_u$ in $W$ is in fact equal to $W$. Consider the subgraph $W'$ of $W$ built inductively as the ascending union of subgraphs $W^n$ as follows: $W^1 = W_u$. $W^2$ is the union of each principal flow line intersecting $W_u$ with the union of all $W(v)$, where $v$ is a singular vertex in $W^1$. $W^n$ is the union of each principal flow line intersecting $W^{n-1}$ along with the union of all $W(v)$, where $v$ is a singular vertex in $W^{n-1}$ with the property that $W(v) \subset W$. Note that $W'$ is a connected subgraph of $W$, by construction. Further, the subgraph $W'_u$ (introduced in the definition of the $M$-saturation $W$ of $W_u$) is contained in $W'$. This gives the reverse containment $W' \supset W$.
\end{proof}

\subsection{Approximations of walls}

In order to show that families of immersed walls are uniformly quasiconvex, Hagen--Wise use a technical construction called approximations. Approximations are not walls in $\wt X$ (they are walls in $\wt X_L$, as it turns out), but have the advantage that, unlike immersed walls, they do not have backward flows that can have long fellow-travelling subpaths. 

We will first state the definition of approximations given in \cite{hagen_wise_irreducible} and then extend their definition to saturations. Let $W_u \to \wt X$ be an immersed wall of tunnel length $L$. 
The \emph{approximation} $A : W_u \to \wt{X}$ 
is defined as below. We refer the reader to \cref{fig:approx} for an illustration.
\begin{enumerate}
    \item For each $x' \in W_u$ such that $x'$ lies in a nucleus $\mathbf{N}_{Li} \subset T_{Li + \frac{1}{2}}$, let $x \in T_{Li}$ denote the parallel copy of $x'$ behind it at distance $\frac{1}{2}$. Then $A(x') := f^L(x) \in T_{Li+L} \subset \wt X$.
    \item  For each $x$ in a level of $W_u$, $A(x)$ is the unique forward endpoint of the level.
    \item Let $\mathbf{S}$ be a slope in $W_u$ associated to the primary bust $d$. Recall that $\mathbf{S}$ is a segment from an endpoint $d^+$ (say) of $d$ to the other endpoint $d'^-$ of $d'$, where $d'$ is the parallel copy of $d$ at distance $\frac{1}{2}$. $A$ maps $\mathbf{S}$ homeomorphically to the concatenation of $d$ and the forward flow segment of length $L$ from $d^-$. 
\end{enumerate}

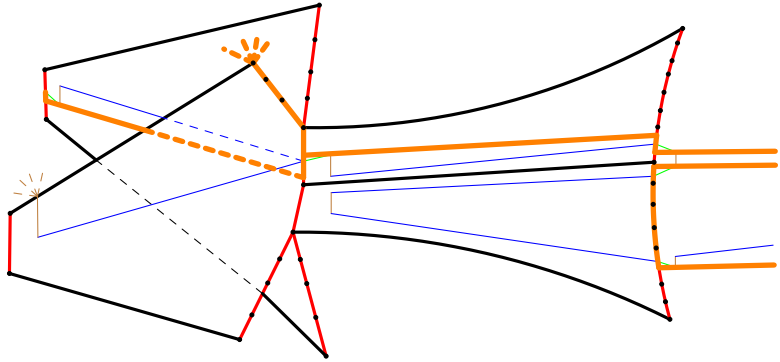
\begin{figure}
    \centering  
    \begin{tikzpicture}[line cap=round,line join=round,>=triangle 45,x=1cm,y=1cm, scale=.5]
\draw [line width=1.2pt,color=red] (-0.52,2.72)-- (-0.52,1.2);
\draw [line width=1.2pt,color=red] (-0.52,1.2)-- (-0.8,-0.06);
\draw [line width=1.2pt] (-0.52,1.2)-- (8.8,1.8);
\draw [shift={(-0.19777777777777658,22.053333333333324)},line width=1.2pt]  plot[domain=4.695723856670749:5.241356869229838,variable=\t]({1*19.336018332073213*cos(\t r)+0*19.336018332073213*sin(\t r)},{0*19.336018332073213*cos(\t r)+1*19.336018332073213*sin(\t r)});
\draw [shift={(-0.6231307309405295,-22.094124967251766)},line width=1.2pt]  plot[domain=1.1077142286854866:1.578823215572557,variable=\t]({1*22.034834826038637*cos(\t r)+0*22.034834826038637*sin(\t r)},{0*22.034834826038637*cos(\t r)+1*22.034834826038637*sin(\t r)});
\draw [shift={(21.45194178082192,0.9601472602739729)},line width=1.2pt,color=red]  plot[domain=2.7872249919041923:3.4081614624309373,variable=\t]({1*12.679786411834097*cos(\t r)+0*12.679786411834097*sin(\t r)},{0*12.679786411834097*cos(\t r)+1*12.679786411834097*sin(\t r)});
\draw [line width=1.2pt,color=red] (-0.52,2.72)-- (-1.86,4.44);
\draw [line width=1.2pt,color=red] (-0.52,2.72)-- (-0.1,5.98);
\draw [line width=1.2pt,color=red] (-0.8,-0.06)-- (-2.22,-2.92);
\draw [line width=1.2pt,color=red] (-0.8,-0.06)-- (0.08,-3.36);
\draw [line width=1.2pt] (-1.86,4.44)-- (-8.32,0.44);
\draw [line width=1.2pt,color=red] (-8.32,0.44)-- (-8.34,-1.16);
\draw [line width=1.2pt] (-8.34,-1.16)-- (-2.22,-2.92);
\draw [line width=1.2pt] (0.08,-3.36)-- (-1.6117156008892224,-1.6948638158754759);
\draw [ ,dashed] (-1.6117156008892224,-1.6948638158754759)-- (-6.032867545676891,1.8561810862681796);
\draw [line width=1.2pt] (-6.032867545676891,1.8561810862681796)-- (-7.36,2.94);
\draw [line width=1.2pt,color=red] (-7.36,2.94)-- (-7.4,4.26);
\draw [line width=1.2pt] (-7.4,4.26)-- (-0.1,5.98);
\draw [ ,color=blue] (0.2148309541697958,0.9897520661157008)-- (8.780717180349555,1.4260341286407408);
\draw [ ,color=blue] (0.19980465815176424,1.4255146506386152)-- (8.841394994453214,2.283436041507512);
\draw [ ,color=green] (8.780717180349555,1.4260341286407408)-- (9.380871525169045,1.695987978963182);
\draw [ ,color=green] (8.841394994453214,2.283436041507512)-- (9.380871525169045,2.07164537941397);
\draw [ ,color=brown] (9.380871525169045,2.07164537941397)-- (9.380871525169045,1.695987978963182);
\draw [ ,color=brown] (0.19980465815176424,1.4255146506386152)-- (0.19980465815176335,1.9965138993238134);
\draw [ ,color=brown] (0.2148309541697958,0.9897520661157008)-- (0.2148309541697949,0.43377911344853415);
\draw [ ,color=blue] (0.2148309541697949,0.43377911344853415)-- (8.90106394592467,-0.843307486376901);
\draw [ ,color=green] (8.90106394592467,-0.843307486376901)-- (9.365845229150997,-0.9937190082644612);
\draw [ ,color=brown] (9.365845229150997,-0.9937190082644612)-- (9.365845229150997,-0.7082193839218622);
\draw [ ,color=blue] (9.365845229150997,-0.7082193839218622)-- (11.965394440270453,-0.4076934635612316);
\draw [ ,color=green] (0.19980465815176335,1.9965138993238134)-- (-0.52,1.8272294515401915);
\draw [ ,color=blue] (-0.52,1.8272294515401915)-- (-7.58382,-0.19733);
\draw [ ,dashed,color=blue] (-0.52,1.8272294515401915)-- (-4.226934700602961,2.974405758140582);
\draw [ ,color=blue] (-4.226934700602961,2.974405758140582)-- (-6.997791134485342,3.8297220135236603);
\draw [ ,color=brown] (-6.997791134485342,3.8297220135236603)-- (-6.997791134485342,3.318827948910588);
\draw [ ,color=green] (-6.997791134485342,3.318827948910588)-- (-7.38195227795431,3.6644251724922228);
\draw [ ,color=brown] (-7.58382,-0.19733)-- (-7.595587973547305,0.8885523383608023);
\draw [ ,dashed,color=brown] (-7.595587973547305,0.8885523383608023)-- (-8.094710743801643,1.3503831705484575);
\draw [ ,dashed,color=brown] (-7.595587973547305,0.8885523383608023)-- (-7.809211119459045,1.4856198347107412);
\draw [ ,dashed,color=brown] (-7.595587973547305,0.8885523383608023)-- (-7.4485800150262875,1.560751314800899);
\draw [ ,dashed,color=brown] (-7.595587973547305,0.8885523383608023)-- (-8.290052592036053,0.9897520661157008);
\draw [line width= 2pt,color=orange] (-7.381952277954311,3.6644251724922188)-- (-7.374912627222182,3.4321166983319853);
\draw [line width= 2pt,color=orange] (-7.374912627222182,3.4321166983319853)-- (-4.744635953713634,2.653847706678865);
\draw [line width= 2pt,dashed,color=orange] (-4.744635953713634,2.653847706678865)-- (-0.52,1.3890952470064308);
\draw [line width= 2pt,color=orange] (-0.52,1.3890952470064308)-- (-0.52,1.9842568546045714);
\draw [line width= 2pt,color=orange] (-0.52,1.9842568546045714)-- (8.867973957529818,2.5160191290606395);
\draw [line width= 2pt,color=orange] (8.867973957529818,2.5160191290606395)-- (8.819976691753645,2.060347895745961);
\draw [line width= 2pt,color=orange] (8.819976691753645,2.060347895745961)-- (12.021386000516358,2.091724420723366);
\draw [line width= 2pt,color=orange] (8.7929668734321,1.686327502230151)-- (12.025662215392098,1.7211538974670193);
\draw [shift={(21.45194178082194,0.960147260273971)},line width= 2pt,color=orange]  plot[domain=3.0842906000810824:3.2975250669274314,variable=\t]({1*12.67978641183412*cos(\t r)+0*12.67978641183412*sin(\t r)},{0*12.67978641183412*cos(\t r)+1*12.67978641183412*sin(\t r)});
\draw [line width= 2pt,color=orange] (8.925997249042533,-1.0090396370312906)-- (11.99216,-0.93554);
\draw [line width= 2pt,color=orange] (-0.52,1.9842568546045714)-- (-0.52,2.72);
\draw [line width= 2pt,color=orange] (-0.52,2.72)-- (-1.86,4.44);
% \draw [line width= 2pt,color=orange] (-0.52,2.72)-- (-0.3627424738527393,3.940617941047785);
\draw [line width= 2pt,dashed,color=orange] (-1.86,4.44)-- (-2.6395269286122374,4.811626796933604);
\draw [line width= 2pt,dashed,color=orange] (-1.86,4.44)-- (-2.1397729969389663,5.047621709112649);
\draw [line width= 2pt,dashed,color=orange] (-1.86,4.44)-- (-1.7510754945264222,5.144796084715785);
\draw [line width= 2pt,dashed,color=orange] (-1.86,4.44)-- (-1.3762600457714689,5.103149923743012);
\draw [fill=black] (-0.52,2.72) circle (1.5pt);
\draw [fill=black] (-0.52,1.2) circle (1.5pt);
\draw [fill=black] (-0.8,-0.06) circle (1.5pt);
\draw [fill=black] (8.8,1.8) circle (1.5pt);
\draw [fill=black] (9.22,-2.38) circle (1.5pt);
\draw [fill=black] (9.56,5.36) circle (1.5pt);
\draw [fill=black] (9.423604362828659,4.9721395843737906) circle (1.5pt);
\draw [fill=black] (9.276581930151853,4.500994045496966) circle (1.5pt);
\draw [fill=black] (9.178576598844137,4.144717962825276) circle (1.5pt);
\draw [fill=black] (9.063061398260693,3.660633571352447) circle (1.5pt);
\draw [fill=black] (8.967040274244066,3.1746941260645096) circle (1.5pt);
\draw [fill=black] (8.897176034486424,2.736334329355827) circle (1.5pt);
\draw [fill=black] (8.775256123400037,1.2405473184889635) circle (1.5pt);
\draw [fill=black] (8.775252853926688,0.679895051603864) circle (1.5pt);
\draw [fill=black] (8.804226825831217,0.05887600974106366) circle (1.5pt);
\draw [fill=black] (8.854539263970647,-0.48291659076030147) circle (1.5pt);
\draw [fill=black] (9.002127593325264,-1.4438307235144856) circle (1.5pt);
\draw [fill=black] (9.101884737458525,-1.9126692106135892) circle (1.5pt);
\draw [fill=black] (-1.86,4.44) circle (1.5pt);
\draw [fill=black] (-0.1,5.98) circle (1.5pt);
\draw [fill=black] (-2.22,-2.92) circle (1.5pt);
\draw [fill=black] (0.08,-3.36) circle (1.5pt);
\draw [fill=black] (-8.32,0.44) circle (1.5pt);
\draw [fill=black] (-8.34,-1.16) circle (1.5pt);
\draw [fill=black] (-7.36,2.94) circle (1.5pt);
\draw [fill=black] (-7.4,4.26) circle (1.5pt);
\draw [fill=black] (-1.5204055532183423,4.0041026503996635) circle (1.5pt);
\draw [fill=black] (-1.092417332772402,3.4547446360959193) circle (1.5pt);
\draw [fill=black] (-0.41919067012217703,3.5024724176231024) circle (1.5pt);
\draw [fill=black] (-0.32884635320251765,4.20371640133284) circle (1.5pt);
\draw [fill=black] (-0.21855164753794898,5.059813402443539) circle (1.5pt);
\draw [fill=black] (-0.6061410788381743,-0.7869709543568466) circle (1.5pt);
\draw [fill=black] (-0.43485477178423243,-1.4292946058091285) circle (1.5pt);
\draw [fill=black] (-0.19485477178423238,-2.3292946058091286) circle (1.5pt);
\draw [fill=black] (-1.1540251078854453,-0.7730364848960376) circle (1.5pt);
\draw [fill=black] (-1.4722298940761083,-1.4139278148293448) circle (1.5pt);
\draw [fill=black] (-1.893828952530404,-2.263063946645743) circle (1.5pt);
\end{tikzpicture}
    \caption{The approximation of the part of the wall from \cref{fig:wall} is indicated in thick orange.}
    \label{fig:approx}
\end{figure}

Let $M \in L \mathbb{N}$ and let $W \to \wt{X}$ be an $M$-saturation of $W_u$. The \emph{approximation of the $M$-saturation} $ A: W \to \wt X$ is defined the same way on nuclei, slopes and levels. On each point $x$ of a principal flow line, by abuse of notation, $A(x) := f^{L-\frac{1}{2}}(x)$. 

We will denote by $N(A(W_u))^1$ (respectively $N(A(W))^1$) the $1$-skeleton of the smallest subcomplex of $\wt X$ containing $A(W_u)$ (respectively $A(W)$).

\subsection{Large tunnel length and quasiconvexity}

In this subsection, we recall results from section 4 of \cite{hagen_wise_irreducible} towards quasiconvexity of immersed walls. We note that their methods do not require local finiteness of $\wt{X}^1$ and thus work in our case as well, both for immersed walls and their saturations.

The main lemma that we will repeatedly use is the following. Recall that $\mathcal{N}_r(Y)$ denotes the $r$-neighbourhood of $Y$.

\begin{lemma}[Lemma 4.3 of \cite{hagen_wise_irreducible}] \label{lem: quasiconvexity small overlaps}
Let $Z$ be a $\delta$-hyperbolic space and let $P = \alpha_0\beta_1\alpha_1\cdots\beta_k\alpha_k$ be a path such that each $\alpha_i$ is a $(\lambda_1,\lambda_2)$-quasigeodesic and each $\beta_i$ is a $(\mu_1,\mu_2)$-quasigeodesic. Suppose that for each $R \geq 0$, there exists a $B_R \geq 0$ such that for all $i$, each intersection below has diameter $\leq B_R$:
$$ \calN_{3\delta +R}(\beta_i) \cap \beta_{i+1}, \; \calN_{3\delta +R}(\beta_i) \cap \alpha_i, \; \calN_{3\delta +R}(\beta_i) \cap \alpha_{i-1}. $$
Then there exists $L_0$ such that if each $\beta_i$ is of length at least $L_0$, then the path $P$ is a {$({4\lambda_1\mu_1}, \frac{\mu_2}{2})$}-quasigeodesic.
\end{lemma}

Let $\mathcal{W} := \{W_u \to \wt X\}$ be a family of immersed walls in $\wt X$. In order to use \cref{lem: quasiconvexity small overlaps}, we need the following property to be satisfied:
\begin{defn}[Definition 3.15 in \cite{hagen_wise_irreducible}]\label{defn: ladder_overlap}
The family of immersed walls $\mathcal{W}$ has the \emph{ladder overlap property} if there exists $B \geq 0$ such that for all $W_u \in \mathcal{W}$ and all distinct slopes $\mathbf{S}_1, \mathbf{S}_2 \subset W_u$, 
$$ \mathrm{diam}(\mathcal{N}_{3\delta + 2\lambda}(N(A(\mathbf{S}_1))^1) \cap \mathcal{N}_{3\delta + 2\lambda}(N(A(\mathbf{S}_2)))^1) \leq B,$$
where $\lambda$ is the quasiconvexity constant for forward ladders (\cref{prop: quasiconvexity_fwd_ladder}).
\end{defn}

In practice, any family of immersed walls is constructed in the following way: We will first choose a finite set of periodic points in the fundamental domain $D$ of $T$ such that the flow rays from these points diverge and do not meet any vertex. The various immersed walls of the family will then be constructed by choosing tunnel lengths and by choosing primary busts in small neighbourhoods of these periodic points{, with the size of the neighbourhoods depending on the chosen tunnel lengths}. The ladder overlap property will then hold because the finitely many flow rays from the periodic points have bounded overlaps.

\begin{prop}\label{prop: approximations_quasiconvex}
Let $\mathcal{W}$ be a family of immersed walls satisfying the ladder overlap property. Then there exist $L_0, \kappa_1, \kappa_2$ such that for all $W_u \in \mathcal{W}$ with tunnel length at least $L_0$, the inclusion $N(A(W_u))^1 \hookrightarrow \wt{X}^1$ is a $(\kappa_1,\kappa_2)$-quasiisometric embedding.
\end{prop}

We refer to Proposition 4.1 of \cite{hagen_wise_irreducible} for a proof. They show that when tunnels are long enough, the hypotheses of \cref{lem: quasiconvexity small overlaps} are satisfied. {The same proof works in our case as \cref{lem: quasiconvexity small overlaps} works for any hyperbolic space, not necessarily proper. The main point is that geodesic paths in $N(A(W_u))^1$ can be written as alternating paths $P = \alpha_0\beta_1\alpha_1\cdots\beta_k\alpha_k$ satisfying the conditions of \cref{lem: quasiconvexity small overlaps}, where each $\beta_i$ is a geodesic fellow traveling with a forward flow segment of length $L \geq L_0$.}

We note that the constants depend only on the following data: $L_0$ depends on the quasiconvexity constant $\mu$ of nuclei of walls (there is a uniform quasiconvexity constant $\mu$ as nuclei are either subintervals of edges or stars of vertices) and the constant $B$ of the ladder overlap property. The constants $\kappa_1, \kappa_2$ depend on $L_0$ and $\mu${, but not on the tunnel length $L \geq L_0$ of individual immersed walls in $\mathcal{W}$}. 

\begin{prop}\label{prop: approximations_trees}
Let $\mathcal{W}$ be a family of immersed walls with the ladder overlap property. Then there exists $L_1 \geq L_0$ such that for each $W_u \in \mathcal{W}$ with tunnel length at least $L_1$, $A(W_u)$ is a tree.
\end{prop}

This is a consequence of \cref{prop: approximations_quasiconvex}: {Indeed, any closed path in $A(W_u)$ has length at most $\kappa_1\kappa_2$. Thus, if $L_1 \geq L_0, \kappa_1\kappa_2 + 1$, then no closed path can contain a forward flow segment from a slope approximation. This implies that the closed path has to be contained in the intersection of a tree $T_i$ in the flow space with the approximation of $W_u$, which is impossible.}
See Proposition 4.4 of \cite{hagen_wise_irreducible} for a precise proof.

% STORY TELLING WHY IT WORKS....

Recall that a \emph{wall} is a subspace $Y \subset \wt X$ such that $\wt X \setminus Y$ has exactly two components, none of which is contained in any finite neighbourhood of $Y$ (i.e., components are deep). 
We now state the main result of this subsection:

\begin{prop}\label{prop: long_tunnels_walls}
Let $\mathcal{W}$ be a family of immersed walls with the ladder overlap property and let $L_1$ be the constant from \cref{prop: approximations_trees}. Then for each $W_u \in \mathcal{W}$ with tunnel length at least $L_1$, $W_u$ is a wall.
\end{prop}

{Again, a detailed proof is given in Proposition 4.6 of \cite{hagen_wise_irreducible}. Since the flow space $\wt{X}$ is simply connected, it is enough to show that $W_u$ locally separates (a small neighbourhood) into two components. Using \cref{rmk: local_homeo_wall}, the only place where this can fail is when two distinct slopes intersect in their interior. But when the tunnel length $L \geq L_1$, such a scenario would contradict the fact that $A(W_u)$ is a tree.}

% The fact that the approximation of $W_u$ is a tree ensures that there is no self-intersection at the interior of two slopes in $W_u$, ensuring that $W_u$ separates $\wt{X}$. 
% See Proposition 4.6 of \cite{hagen_wise_irreducible} for a proof.

{\begin{prop}[Approximations of saturations are quasiconvex]\label{prop: saturation approximations quasiconvex}
Let $\mathcal{W}$ be a family of immersed walls satisfying the ladder overlap property. Then there exists $L_2 \geq L_1$ such that for all $W_u \in \mathcal{W}$ with tunnel length $L \geq L_2$, there exist $M_0, \Theta_1$ and $\Theta_2$ such that for any $M$-saturation $W$ of $W_u$ with $M \geq M_0L$, the inclusion $N(A(W))^1 \hookrightarrow \wt{X}^1$ is a $(\Theta_1,\Theta_2)$-quasiisometric embedding.
\end{prop}}

{\begin{proof}
    We will prove the statement by again using \cref{lem: quasiconvexity small overlaps}. Let $P$ be a path in $N(A(W))^1$ such that $P$ is a concatenation $\alpha_0\beta_1\alpha_1\cdots\beta_k\alpha_k$, with each $\alpha_i$ a geodesic in the approximation of a component subgraph of $\overline{W}$ and each $\beta_i$ is a geodesic segment in a principal flow line. By \cref{prop: approximations_quasiconvex}, each $\alpha_i$ is a $(\kappa_1,\kappa_2)$-quasigeodesic, while \cref{prop: quasiconvexity_fwd_ladder} ensures that each $\beta_i$ is a $(\lambda_1,\lambda_2)$-quasigeodesic. We choose tunnel length $L$ large enough so that, in length smaller than $L$, $\alpha_i$ and $\beta_j$ diverge from each other:  
    recall that primary bust endpoints are chosen so that their flow rays are not asymptotic to any principal flow line (\cref{lem: periodic_divergence_principal_lines}). By choosing $L$ larger than the minimal divergence distance between the flow ray of any primary bust endpoint and a principal flow line, we obtain the required bounds. 
\end{proof}}

We also observe that since any immersed wall $W_u$ (respectively its $M$-saturation $W$) of tunnel length $L$ is at Hausdorff distance $L$ from its approximation $A(W_u)$ (respectively $A(W)$), we have:

\begin{lemma}\label{lem: walls_and_approximations_fellow_travel}
For any wall $W_u$ (with $M$-saturation $W$), the limit sets in $\partial \wt X^1$ of $N(W_u)^1$ and $N(A(W_u))^1$ (respectively, $N(W)^1$ and $N(A(W))^1$) coincide. \qed
\end{lemma}

{\begin{lemma} \label{lem: busts_same_component} Let $d, d'$ be primary busts such that for some $m,m' \in \bbZ$, $d \subset T_{mL}$ and $d' \subset T_{m'L}$, and one of the endpoints of each of $d,d'$ is contained in $W_u$. Then the interiors of $d$ and $d'$ lie in the same complementary component of $W_u$. \end{lemma}
\begin{proof} 
For the purposes of this argument, we endow $W_u$ with a graph structure where the vertices are vertices of nuclei and the endpoints of various primary and secondary busts. Edges consist of forward flow segments in levels, slopes, and segments in nuclei between vertices. 
Let us assume that $d \subset T_0$ and let $P$ be a (geodesic) path in $W_u$ between the endpoints of $d$ and $d'$. The edge of $P$ with endpoint $d^+$ (say), which is either a slope starting at $d^+$ or a forward flow segment ending at $d^+$, leads to a nucleus $\mathbf{N}$ in a ``fractional'' tree $T_{\frac{1}{2}}$ or $T_{-L + \frac{1}{2}}$. Except for the  central vertex, every vertex of this nucleus is the endpoint of either a primary bust or of a secondary bust. In the latter case, note that each such secondary bust (in this fractional copy of $T$) meets the same complementary component $C$ as $d$ (see \cref{fig:components}), and thus flowing such a secondary bust to a primary bust in $T_L$ or $T_0$ (along an edge of $W_u$ without crossing it) will not lead to a change of components. In the former case, the primary bust  in the fractional copy, with an endpoint in $\mathbf{N}$, is not in $C$, but travelling back along the slope will lead to its parallel copy (in either $T_0$ or $T_{-L}$), and this primary bust is contained in $C$. One can now proceed by induction on the distance between $d$ and $d'$ to obtain the desired result.\end{proof}}
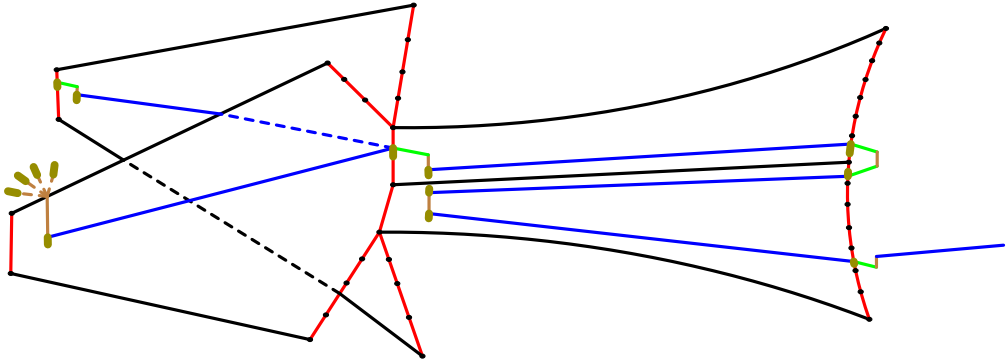
\begin{figure}
    \centering
    \begin{tikzpicture}[xscale = .65, yscale = 0.5, line cap=round,line join=round,>=triangle 45,x=1cm,y=1cm]
\draw [line width=1.2pt,color=red] (-0.52,2.72)-- (-0.52,1.2);
\draw [line width=1.2pt,color=red] (-0.52,1.2)-- (-0.8,-0.06);
\draw [line width=1.2pt] (-0.52,1.2)-- (8.8,1.8);
\draw [shift={(-0.19777777777777658,22.053333333333324)},line width=1.2pt]  plot[domain=4.695723856670749:5.241356869229838,variable=\t]({1*19.336018332073213*cos(\t r)+0*19.336018332073213*sin(\t r)},{0*19.336018332073213*cos(\t r)+1*19.336018332073213*sin(\t r)});
\draw [shift={(-0.6231307309405295,-22.094124967251766)},line width=1.2pt]  plot[domain=1.1077142286854866:1.578823215572557,variable=\t]({1*22.034834826038637*cos(\t r)+0*22.034834826038637*sin(\t r)},{0*22.034834826038637*cos(\t r)+1*22.034834826038637*sin(\t r)});
\draw [shift={(21.45194178082192,0.9601472602739729)},line width=1.2pt,color=red]  plot[domain=2.7872249919041923:3.4081614624309373,variable=\t]({1*12.679786411834097*cos(\t r)+0*12.679786411834097*sin(\t r)},{0*12.679786411834097*cos(\t r)+1*12.679786411834097*sin(\t r)});
\draw [line width=1.2pt,color=red] (-0.52,2.72)-- (-1.86,4.44);
\draw [line width=1.2pt,color=red] (-0.52,2.72)-- (-0.1,5.98);
\draw [line width=1.2pt,color=red] (-0.8,-0.06)-- (-2.22,-2.92);
\draw [line width=1.2pt,color=red] (-0.8,-0.06)-- (0.08,-3.36);
\draw [line width=1.2pt] (-1.86,4.44)-- (-8.32,0.44);
\draw [line width=1.2pt,color=red] (-8.32,0.44)-- (-8.34,-1.16);
\draw [line width=1.2pt] (-8.34,-1.16)-- (-2.22,-2.92);
\draw [line width=1.2pt] (0.08,-3.36)-- (-1.6117156008892224,-1.6948638158754759);
\draw [line width=1.2pt,dashed] (-1.6117156008892224,-1.6948638158754759)-- (-6.032867545676891,1.8561810862681796);
\draw [line width=1.2pt] (-6.032867545676891,1.8561810862681796)-- (-7.36,2.94);
\draw [line width=1.2pt,color=red] (-7.36,2.94)-- (-7.4,4.26);
\draw [line width=1.2pt] (-7.4,4.26)-- (-0.1,5.98);
\draw [line width=1.2pt,color=blue] (0.2148309541697958,0.9897520661157008)-- (8.780717180349555,1.4260341286407408);
\draw [line width=1.2pt,color=blue] (0.20037791829867577,1.5988337066301155)-- (8.841394994453214,2.283436041507512);
\draw [line width=1.2pt,color=green] (8.780717180349555,1.4260341286407408)-- (9.380871525169045,1.695987978963182);
\draw [line width=1.2pt,color=green] (8.841394994453214,2.283436041507512)-- (9.380871525169045,2.07164537941397);
\draw [line width=1.2pt,color=brown] (9.380871525169045,2.07164537941397)-- (9.380871525169045,1.695987978963182);
\draw [line width=1.2pt,color=brown] (0.20037791829867577,1.5988337066301155)-- (0.19980465815176335,1.9965138993238134);
\draw [line width=1.2pt,color=brown] (0.2148309541697958,0.9897520661157008)-- (0.2148309541697949,0.43377911344853415);
\draw [line width=1.2pt,color=blue] (0.2148309541697949,0.43377911344853415)-- (8.90106394592467,-0.843307486376901);
\draw [line width=1.2pt,color=green] (8.90106394592467,-0.843307486376901)-- (9.365845229150997,-0.9937190082644612);
\draw [line width=1.2pt,color=brown] (9.365845229150997,-0.9937190082644612)-- (9.365845229150997,-0.7082193839218622);
\draw [line width=1.2pt,color=blue] (9.365845229150997,-0.7082193839218622)-- (11.965394440270453,-0.4076934635612316);
\draw [line width=1.2pt,color=green] (0.19980465815176335,1.9965138993238134)-- (-0.52,2.1768294515401916);
\draw [line width=1.2pt,color=blue] (-0.52,2.1768294515401916)-- (-7.58382,-0.19733);
\draw [line width=1.2pt,dashed,color=blue] (-0.52,2.1768294515401916)-- (-4.051438033936354,3.0830724248072117);
\draw [line width=1.2pt,color=blue] (-4.051438033936354,3.0830724248072117)-- (-6.988842116417383,3.5979101899766732);
\draw [line width=1.2pt,color=brown] (-6.988842116417383,3.5979101899766732)-- (-6.9796646306753365,3.808992362043771);
\draw [line width=1.2pt,color=green] (-6.9796646306753365,3.808992362043771)-- (-7.389919506151647,3.9273437030043175);
\draw [line width=1.2pt,color=brown] (-7.58382,-0.19733)-- (-7.595587973547305,0.8885523383608023);
\draw [line width=1.2pt,dashed,color=brown] (-7.595587973547305,0.8885523383608023)-- (-8.094710743801643,1.3503831705484575);
\draw [line width=1.2pt,dashed,color=brown] (-7.595587973547305,0.8885523383608023)-- (-7.809211119459045,1.4856198347107412);
\draw [line width=1.2pt,dashed,color=brown] (-7.595587973547305,0.8885523383608023)-- (-7.4485800150262875,1.560751314800899);
\draw [line width=1.2pt,dashed,color=brown] (-7.595587973547305,0.8885523383608023)-- (-8.290052592036053,0.9897520661157008);
\draw [line width=3pt,color=olive] (-0.52,2.176829451540191)-- (-0.52,1.9512692783774392);
\draw [line width=3pt,color=olive] (0.2003779182986758,1.5988337066301155)-- (0.20471660821812673,1.4510489497810455);
\draw [line width=3pt,color=olive] (8.841394994453214,2.283436041507512)-- (8.817793606615815,2.034988347103665);
\draw [line width=3pt,color=olive] (8.780717180349555,1.4260341286407408)-- (8.78564413497464,1.5448586391065997);
\draw [line width=3pt,color=olive] (0.2148309541697958,0.9897520661157008)-- (0.21967106719130067,1.0936011635240008);
\draw [line width=3pt,color=olive] (0.21483095416979492,0.43377911344853426)-- (0.20851155836871038,0.3068557915313822);
\draw [line width=3pt,color=olive] (8.90106394592467,-0.843307486376901)-- (8.909972322846464,-0.9042522583227015);
\draw [line width=3pt,color=olive] (-7.58382,-0.19733)-- (-7.581153710309237,-0.389128312212138);
\draw [line width=3pt,color=olive] (-6.988842116417371,3.5979101899766723)-- (-6.992657834914038,3.4494557594562334);
\draw [line width=3pt,color=olive] (-7.389919506151647,3.92734370300432)-- (-7.386104345641894,3.8014434061824933);
\draw [line width=3pt,color=olive] (-8.199298750449753,0.9765271100607313)-- (-8.397175941270673,1.02861389047313);
\draw [line width=3pt,color=olive] (-8.035772190798902,1.2958482092055346)-- (-8.195624588146693,1.4451533535960233);
\draw [line width=3pt,color=olive] (-7.8002686177868865,1.4606259275116256)-- (-7.873142423148325,1.6735782204698681);
\draw [line width=3pt,color=olive] (-7.468077636493243,1.471597766074962)-- (-7.4431662031505015,1.7273252479695962);

\draw [fill=black] (-0.52,2.72) circle (1.5pt);
\draw [fill=black] (-0.52,1.2) circle (1.5pt);
\draw [fill=black] (-0.8,-0.06) circle (1.5pt);
\draw [fill=black] (8.8,1.8) circle (1.5pt);
\draw [fill=black] (9.22,-2.38) circle (1.5pt);
\draw [fill=black] (9.56,5.36) circle (1.5pt);
\draw [fill=black] (9.423604362828659,4.9721395843737906) circle (1.5pt);
\draw [fill=black] (9.276581930151853,4.500994045496966) circle (1.5pt);
\draw [fill=black] (9.14182335724371,3.9995498080147316) circle (1.5pt);
\draw [fill=black] (9.033516373713434,3.5213371903289667) circle (1.5pt);
\draw [fill=black] (8.941100806379133,3.023116373454311) circle (1.5pt);
\draw [fill=black] (8.866507102648233,2.504108856673259) circle (1.5pt);
\draw [fill=black] (8.775256123400037,1.2405473184889635) circle (1.5pt);
\draw [fill=black] (8.775252853926688,0.6798950516038695) circle (1.5pt);
\draw [fill=black] (8.804226825831217,0.058876009741069324) circle (1.5pt);
\draw [fill=black] (8.854539263970647,-0.48291659076030147) circle (1.5pt);
\draw [fill=black] (8.937946667498558,-1.083601685524325) circle (1.5pt);
\draw [fill=black] (9.04323604851519,-1.647726490265102) circle (1.5pt);
\draw [fill=black] (-1.86,4.44) circle (1.5pt);
\draw [fill=black] (-0.1,5.98) circle (1.5pt);
\draw [fill=black] (-2.22,-2.92) circle (1.5pt);
\draw [fill=black] (0.08,-3.36) circle (1.5pt);
\draw [fill=black] (-8.32,0.44) circle (1.5pt);
\draw [fill=black] (-8.34,-1.16) circle (1.5pt);
\draw [fill=black] (-7.36,2.94) circle (1.5pt);
\draw [fill=black] (-7.4,4.26) circle (1.5pt);
\draw [fill=black] (-1.5204055532183423,4.0041026503996635) circle (1.5pt);
\draw [fill=black] (-1.092417332772402,3.4547446360959193) circle (1.5pt);
\draw [fill=black] (-0.41919067012217703,3.5024724176231024) circle (1.5pt);
\draw [fill=black] (-0.32884635320251765,4.20371640133284) circle (1.5pt);
\draw [fill=black] (-0.21855164753794898,5.059813402443539) circle (1.5pt);
\draw [fill=black] (-0.6061410788381743,-0.7869709543568466) circle (1.5pt);
\draw [fill=black] (-0.43485477178423243,-1.4292946058091285) circle (1.5pt);
\draw [fill=black] (-0.19485477178423238,-2.3292946058091286) circle (1.5pt);
\draw [fill=black] (-1.1540251078854453,-0.7730364848960376) circle (1.5pt);
\draw [fill=black] (-1.4722298940761083,-1.4139278148293448) circle (1.5pt);
\draw [fill=black] (-1.893828952530404,-2.263063946645743) circle (1.5pt);
% \draw[color=olive] (-0.7686937124490356,2.3084562489184393) node {$d$};
\end{tikzpicture}
    \caption{The busts in olive are all in a single complementary component.}
    \label{fig:components}
\end{figure}

\begin{prop}\label{prop: walls_cocompact_stabilisers}
Let $W_u \in \mathcal{W}$ be an immersed wall with tunnel length at least $L_2$ (from \cref{prop: long_tunnels_walls}). Then $\mathrm{Stab}(W_u) < \Gamma$ is a codimension-1 subgroup that stabilises each complementary component of $W_u$ and acts cocompactly on $W_u$.
\end{prop}
\begin{proof}
{First, complementary components are not flipped: as a consequence of \cref{lem: busts_same_component}, if $\gamma \in \mathrm{Stab}(W_u)$ and $C$ is a complementary component of $W_u$, we have that $\gamma C = C$.}

{Let us show cocompactness. Recall (from \cref{sec: immersed_walls}) that $W_u$ is a component of the graph $\overline{W}$. For a point $x \in \wt{X}$ such that $x \in \overline{W}$ (the image of $\overline{W}$, to be precise), denote by $W(x)$ the component of $\overline{W}$ that contains $x$. Observe that for two points $x,y \in \overline{W}$, $W(x) = W(y)$ if and only if $y \in W(x)$ (and $x \in W(y)$). }

{Let $L$ be the tunnel length of $W_u$. Denote by $\Gamma_L$ the subgroup of $\Gamma$ generated by $G$ and $t^L$. We first note that, by the equivariant choice of primary busts, and the fact that no primary bust intersects any secondary bust (\cref{lem: choice_of_busts}(1)), the action of $\Gamma_L$ on the flow space restricts to an action on $\overline{W}$.}

{Let $x \in W_u$. Let us denote by $H < \Gamma$ the stabiliser of $W_u$. For an element $\gamma \in \Gamma_L$, we observe that $\gamma \in H$ if and only if $\gamma x \in W_u$. Indeed, $\gamma x \in W_u$ if and only if $W_u = W(x) = W(\gamma x)$, with the latter being equal to $\gamma W(x)$.
Let us denote by $X$ the compact quotient of $\wt X$ by the action of $\Gamma$. Observe that $X$ is the mapping torus of the compact graph $\leftQ{T}{G}$ by the map induced by $f$. There are finitely many images in $X$ of busts and nuclei of $W_u$ and hence the image of $W_u$ in $X$ is compact. This implies that $H$ acts cocompactly on $W_u$.}

{Finally, we show that $\mathrm{Stab}(W_u)$ has   codimension-1. There is a continuous equivariant collapse map from a Cayley graph of $\Gamma$ to $\wt X^1$ which crushes free factors of $G$ to points. Since $W_u$ separates $\wt X$ into deep components, so does its preimage in $\Gamma$, giving us the desired result.} 
\end{proof}

{Note that if $W_u$ is a wall, then any $M$-saturation of $W_u$ separates $\wt X^1$ into several deep components and therefore separates the boundary $\partial \wt X^1$. In fact:}

{\begin{prop} \label{prop: saturated_walls_stabilisers}
Let $W_u \in \mathcal{W}$ be a wall with tunnel length at least $L_2$ (from \cref{prop: saturation approximations quasiconvex}). Then there exists $L_3 \geq L_2$ such that for all $L > L_3$ and for all $M$ large enough, the stabiliser of the $M$-saturation $W$ of $W_u$ is a relatively quasiconvex full subgroup. Further, $W$ separates the flow space into at least two deep components.
\end{prop}}

\begin{proof}
{ {  To prove separation with deep components, since $W_u$ itself is of co-dimension 1 by \cref{prop: walls_cocompact_stabilisers}, it suffices to find, for each deep component of    $\wt X \setminus W_u$,   a ray starting at  a  point of $W_u$, entering the given component,  and never encountering $W$ again. To do that,  it suffices to find a ray avoiding all the principal flow lines issued from $W_u$, and diverging from each of them within $M/10$ of distance from $W_u$. Since $\wt X$ is non-elementary hyperbolic, and principal flow lines are disjoint and diverge from each other in bounded time, it is possible to find such a ray, provided $M$ is sufficiently large compared to $\delta$. } }
    % : in each side choose a ray that leaves all pfl in less than M/2
%It is codimension-1 by the same argument as for $W_u$ in \cref{prop: walls_cocompact_stabilisers}.}

Since $W$ is quasiconvex, an application of Theorem 1.2 of \cite{einstein_groves_ng} implies that the stabiliser of $W$ is relatively quasiconvex in $\Gamma$. Let us show that $\mathrm{Stab}(W)$  is also full. To this end, consider a principal flow line $\Lambda$ whose associated parabolic group (a suspension of a vertex stabiliser in $T_0$ by an element translating  along $\Lambda$)  intersects  $\mathrm{Stab}(W)$.  We want to show that  $\mathrm{Stab}(W)$ contains a finite index subgroup of the parabolic group. For this it is sufficient to show that $\Lambda$ intersects $W$, since in that case, the saturation property of $W$ ensures that $W$ is stabilised by such a subgroup. 

Thus, let  $\mathrm{Stab}(W)$ intersect $\mathrm{Stab}(\Lambda)$. We distinguish whether the intersection contains an element translating along $\Lambda$ or an element fixing a vertex in $\Lambda$.

First assume that $\mathrm{Stab}(W)$ contains an element  of $\Gamma$ which acts loxodromically on  $\Lambda$. This implies that a finite neighbourhood of $W$ contains $\Lambda$. Since every pair of principal flow lines diverge, either a vertex of $\Lambda$  is contained in $W$ (and we are done, by the property of saturations),  or there exists a periodic union of tunnels and slopes and nuclei of $W$ that fellow travels with $\Lambda$. Assume the latter. If there exists a singular vertex (of some nucleus) in this periodic union, then the no twinning property ensures a contradiction. 
 Let $\gamma$ be translating along $\Lambda$, and such that $\gamma^k$ stabilises $W$ (for some large $k$).  Let $x\in W$ be near $\Lambda$, and consider a shortest path $p$ in $W$ from $x$ to $\gamma^k x \in W$.  Since $W$ is assumed not to cross $\Lambda$, the path $p$ does not contain an arc of $\Lambda$. Thus $p$ consists of nuclei, slopes, and levels,  
 and is a quasigeodesic, that must remain close to $\Lambda$. Note that $p$ does not contain any piece of a principal flow line by the no twinning assumption. Since $p$ moves from a tree $T_r$ to a tree $T_{r+k}$ for some large $k$, there must be a level in $p$. As a subpath of $p$,  it travels close to $\Lambda$. But a level is a piece of a periodic flow line, and by our choice of primary bust endpoints (\cref{lem: periodic_divergence_principal_lines}) and large tunnel length, levels 
 diverge  from any principal flow line $\Lambda$. Hence, $p$ does not stay close to $\Lambda$ as promised.

{Assume now that the stabiliser of $W$ contains an elliptic element $g${, fixing some vertex $v$ of $T_0$}. 
% Note that $g$ fixes some vertex $v$ of $T_0$ and is therefore of infinite order. 
Let $x \in T_0$ be in $W$ so that for each $n \in \bbZ$, $g^n x \in W$. Since the $T_0$-distance between $g^mx$ and $g^n x$ is bounded for all $m,n$ (it equals twice the $T_0$-distance between $x$ and $v$), there exists $K \geq 0$ such that any geodesic path $\alpha_{m,n}$ in the quasiconvex space $W$ between $g^mx$ and $g^nx$ does not meet $T_i$ for $i >K$. We flow each such geodesic to $T_K$ to obtain a path in $T_K$ between $f^K(g^mx)$ and $f^K(g^nx)$. Observe that there is a uniform bound on the length of each such path. Thus, up to taking a subsequence, there is an infinite valence vertex of $T_K$ contained in each of the flowed paths between $f^K(g^mx)$ and $f^K(g^nx)$. By equivariance of $f$, this vertex $w$ is fixed by $\phi^K(g^n)$. By uniqueness of vertex stabilisers under the $G$-action on $T_K$, we have $w = f^K(v)$. We claim that this implies that infinitely many of the paths $\alpha_{m,n}$ intersect the principal flow line $\Lambda_v$ of $v$. Indeed, by \cref{lem: preimage_singular_vertices} the intersection of the backward flow of $w$ with the trees $T_i$ intersects the infinite collection $\alpha_{m,n}$ in a finite set. Let $w' \in T_i$ be one such point that lies in infinitely many $\alpha_{m,n}$. If $w'$ is in $\Lambda_v$, we are done. If not, then by equivariance, we again have that $w'$ is stabilised by $\phi^i(g^k)$ for infinitely many $k$, a contradiction.} \end{proof}

\subsection{Many effective walls}

Recall that $D \subset T$ denotes the fixed fundamental domain. 

\begin{defn}\label{defn: effective_walls}
Using the terminology of \cite{hagen_wise_irreducible}, we say $\wt X$ has \emph{many effective walls} if the following two conditions are satisfied:
\begin{enumerate}
    \item For each $y \in D \subset T_0$ such that the flow ray from $y$ does not meet a vertex, there exists a set $\mathcal{W}$ of immersed walls with the ladder overlap property such that for every $\epsilon >0$, there exists $W_u \in \mathcal{W}$ of arbitrarily large tunnel length $L$ and a primary bust in the $\epsilon$-neighbourhood of $y$. We will call such a set of walls a \emph{regular effective set of walls}.
    \item For each periodic point $a \in D$, there exists $k = k(a)$ and a set of immersed walls $\mathcal{W}_a$ with the ladder overlap property such that 
    % there exists $W_u \in \mathcal{W}_a$ of arbitrarily large tunnel length and such that 
    for each primary bust $d' \subset T' = T_{\frac{1}{2}}$ that is joined to $a' = a_{\frac{1}{2}} \in T'$ by a path in $W_u$ disjoint from the slopes of $W_u$, we have that $d_{\wt{X}^1}(f^n(a),f^n(d)) \geq 3\delta + 2\lambda$, for all $n \geq k$. We will call such a set of walls a \emph{periodic effective set of walls}.
\end{enumerate}
\end{defn}

Let us comment on why one needs the property of many effective walls. The full details are available in Section 5 of \cite{hagen_wise_irreducible}. For a hyperbolic group to admit a proper cocompact cubulation, the boundary criterion of \cite{bergeron_wise} stipulates that every pair of distinct points in the boundary of the group (equivalently, the limit set of every bi-infinite geodesic of $\wt X^1$ in the setup of \cite{hagen_wise_irreducible}) should be separated by the limit set of a quasiconvex codimension-1 subgroup (equivalently, a quasiconvex wall in $\wt X^1$). Having many effective walls assures that this can be done in $\wt X^1$. Indeed, there are two types of bi-infinite geodesics in $\wt X^1$, ``horizontal'' geodesics and ``non-horizontal'' geodesics, that we define below. If $\wt X^1$ has many effective walls, then each horizontal geodesic is cut by a wall from a periodic effective set, while each non-horizontal geodesic is cut by a wall from a regular effective set.

\begin{defn}[Geodesic classification] \label{defn: geodesic_types}
    Let $N>0$. A bi-infinite geodesic $\gamma$ in $\wt X^1$ is \emph{$N$-horizontal}\footnote{$N$-ladderlike in \cite{hagen_wise_irreducible}} if there exists a forward flow segment $\sigma$ of length $N$ such that a geodesic of the forward ladder $N(\sigma)$ joining the endpoints of $\sigma$ fellow-travels with a subpath of $\gamma$ at distance at most $2\delta + \lambda$. Here, $\delta$ is the hyperbolicity constant of $\wt X^1$ and $\lambda$ is the quasiconvexity constant of forward ladders (\cref{prop: quasiconvexity_fwd_ladder}). Otherwise, $\gamma$ is \emph{$N$-non-horizontal}.\footnote{$N$-deviating in \cite{hagen_wise_irreducible}}
\end{defn}

\begin{theorem}\label{thm: effective_walls}
The flow space $\wt X$ has many effective walls.
\end{theorem}

This is proved in Theorem 6.16 of \cite{hagen_wise_irreducible}, to which we refer the reader for a detailed proof.
The main ingredients of their proof, which are available in our case as well, are the facts that periodic points are dense in $T$ (\cref{lem: periodic_dense}) and that flow rays starting from translates of a periodic point diverge (\cref{prop: periodic_flowrays}). {With these ingredients at hand, here is a brief sketch of the proof. The first condition of \cref{defn: effective_walls} can be verified by first choosing a periodic point in an $\epsilon$-neighbourhood of the given point $y$ and then choosing one periodic point in each edge of $D$ so that pairwise, their flow lines diverge. This will ensure that the ladder overlap property holds when primary busts are chosen in small enough neighbourhoods of these points, depending on the tunnel length. The family of walls is now constructed by choosing tunnel lengths to be large common multiples of the periods of the chosen periodic points, and choosing primary busts as above. To verify the second condition, we choose one periodic point per edge of $D$ as above, including one for the edge containing $a$, while ensuring that pairwise, the flow lines of the periodic points and of $a$ diverge. The family of walls is then constructed as done for the first condition.}

Let us conclude this section with another property that will be useful in the next section. Following \cite{hagen_wise_irreducible}, we say $\wt X$ is \emph{level-separated} if for each $N >0$ and $K\geq 0$, and each $N$-non-horizontal geodesic $\gamma$, there exists a point $y \in \wt X$ such that the backward flow of length $n$ meets $\gamma$ in an odd cardinality set for all large $n$, and the intersection is at distance at least $N + K$ from both $y$ and the leaves of the backward flow. 

\begin{lemma}\label{lem: level_separation}
The flow space $\wt X$ is level-separated.    
\end{lemma}

    In order to prove this statement, we will make use of an $\mathbb{R}$-tree $T_\infty$ obtained as a limit of the trees $T_m$ in the flow space with metrics $d_m = \frac{d|_{T_m}}{\lambda^m}$, where $\lambda$ is the (maximal) stretching factor of the train track map $f$. We will not go into details, but refer the reader to \cite{horbez_r_trees} and \cite[Section 6]{hagen_wise_irreducible}. {Formally, $T_\infty$ is an ultralimit, its points are equivalence classes of sequences of points in the $T_m$. In particular to each point in $T_0$ (and in $\wt X$), its flow ray defines a point in $T_\infty$. This defines a tautological map $\rho: \wt X \to T_\infty$.}

The main points we will use are {easy} facts, that the map $\rho: \wt X \to T_\infty$  is continuous, restricts to an isometric embedding on any vertical edge {(as $f$ is a train track map)} and that the preimage of any point in $T_\infty$ consists of points whose forward flow rays either intersect or remain at bounded Hausdorff distances from each other.

\begin{proof}
Let $\gamma : \mathbb{R}\to \wt X^1$ be an $N$-non-horizontal geodesic. Observe that the flow-like parts of $\gamma$ are horizontal segments starting and ending at vertices and that the vertical parts are paths in trees $T_k$. Since $\gamma$ is $N$-non-horizontal, it maps infinitely many unit subintervals of $\bbR$ to vertical edges, which we call `vertical' subintervals.

First, we make a slight change.  Let $\gamma_{\mathrm{vert}}$ be the union of paths defined on the vertical subintervals of $\gamma$, but now  
reparametrised on $\mathbb{R}$.

Let $\gamma^{\infty} = \rho \circ \gamma_{\mathrm{vert}} : \bbR \to T_\infty$ be the image of $\gamma_{\mathrm{vert}}$ (and also of $\gamma$, in fact) in $T_{\infty}$.
Call $\gamma_{\mathrm{vert}}(s)$ a ``flat" point if $s$ is a point of discontinuity for $\gamma_{\mathrm{vert}}$. Observe that this happens only when $\gamma(s)$ is the starting point in $\gamma$ of a horizontal segment, and that there are only countably many flat points. Since all points in any horizontal segment have the same flow, all such points have the same image in $T_\infty$. Therefore, the map $\gamma^{\infty}: \bbR \to T_\infty$ is continuous, and in fact, $1$-Lipschitz.

Further, there exists $D$ such that the preimage of 
a point in $\gamma^{\infty}$ is of diameter $\leq D$. Indeed, if $s_1$ and $s_2$ have the same $\gamma^\infty$-image, observe that $\gamma_{\mathrm{vert}}(s_1)$ and $\gamma_{\mathrm{vert}}(s_2)$ correspond to points in the image of $\gamma$ which fellow-travel a forward flow segment. Since $\gamma$ is $N$-non-horizontal, there exists a $D$ as required.

In fact, the same argument shows that $\gamma^\infty(\bbR_+)$ (respectively $\gamma^\infty(\bbR_-)$) has infinite 
diameter and goes to an end $\gamma^\infty_+$ (respectively $\gamma^\infty_-$) of $T_\infty$. Note that, by the bound on the diameter of preimages, $\gamma^\infty_+ \neq \gamma^\infty_-$.
In particular, there are uncountably many points that cut 
$\gamma^\infty(\bbR)$ in two unbounded components, with each such point contained in $\gamma^\infty(\bbR)$ (as $T_\infty$ is an $\bbR$-tree).

Let us now show that there exists $D'$ such that the preimage of a point under $\gamma^\infty$ has cardinality at most $D'$. Indeed, if there are many points in the preimage, there 
must be many in the same tree $T_k$, and their diameter being bounded, 
either they accumulate, or have to be placed so that there is a vertex 
and a huge angle to reach them. Accumulation is not permitted since 
all edges are stretched by the scaling factor (more precisely, we are 
guaranteed that two points in the same edge are flown to different 
points of $T_\infty$). {Also for two points separated by a huge angle, we already noticed that their flow lines cannot fellow-travel (see  \cref{prop: flowrays_fellow_travel_or_diverge}), hence they must diverge with speed defined by the stretch factor $\lambda$,  thus defining different points in the limit tree $T_\infty$.} 

%after a huge angle, passing to the limit 
%tree cannot collide for actually similar reasons.

Call $\gamma^\infty(s)$ a ``backtrack" point if 
there are two  intervals $[s_-, s), (s, s+]$ that are sent by 
$\gamma^\infty$ to the same component of $T_\infty \setminus 
\{ \gamma^ \infty (s) \}$.

Observe that since the map $\rho$ is injective on edges, if $\gamma^\infty(s)$ is a backtrack point, then $\gamma_{\mathrm{vert}}(s)$ is a vertex. 
Therefore, there are countably many backtrack points in $\gamma^\infty$.

There are two kinds of backtrack points: those such that the component of 
the intervals $[s_-, s), (s, s+]$ contains both the ends 
$\gamma^\infty_+$ and $\gamma^\infty_-$, and those that don't. For the former, take the 
centre of the tripod $(\gamma^\infty_+, \gamma^\infty_-, \gamma^\infty(s))$,  and call it a ``crossroad'' point.

There are countably many backtrack points, so the union of flat points, backtrack points and crossroad points is a countable set.

Pick a point $\xi$ in $\gamma^\infty (\bbR)$ that is not a backtrack 
point nor a crossroad point nor a flat point, and such that $\xi$ separates $\gamma^\infty (\bbR)$ in 
two infinite components. Since it is not a backtrack point, there is a well-defined direction of crossing $\xi$ at each preimage $s$ (a small 
interval around $s$ so that $(s_-, s)$ and $(s,s_+)$ are sent on 
different sides of $\xi$, each containing only one of the ends 
$\gamma^\infty_+, \gamma^\infty_-$).

Since preimages in $\gamma_{\mathrm{vert}}$ are finite, for $\xi$, there can only be 
an odd number of preimages in $\gamma_{\mathrm{vert}}$. Since $\xi$ is not a flat 
point, it has the same (odd) number of preimages in both $\gamma$ and $\gamma_{\mathrm{vert}}$. Then at least one of the points $x$ in the preimage is such that for a large $n$, if $y = f^n(x)$, then the backward flow of large enough length from $x$ meets $\gamma$ in an odd cardinality set, with the intersection at a large distance from $y$ and the leaves of the backward flow, as required.
\end{proof}

\section{Cutting principal flow lines and geodesics by wall saturations} \label{sec: wall_saturations_cut}

In this section, we will show that three types of subsets of $\wt X^1$ are separated by the saturations of walls. In \cite{hagen_wise_irreducible}, they show that all bi-infinite geodesics are separated by (unsaturated) walls. For relative cubulation (see \cref{sec: rel_cubul}), we will need more: that saturations of walls also separate principal flow lines from geodesic rays and pairs of principal flow lines.

{Let us begin with a tautological observation in preparation for the future coning-off of principal flow lines:}

{\begin{lemma}\label{lem: saturation flowline no cut}
 Let $W_u$ be a quasiconvex wall in $\wt X$ and $W$ an $M$-saturation of $W_u$. Then for each principal flow line $\Lambda$ of a singular vertex, either $\Lambda \subset W$ or the two points of $\partial \Lambda$ lie in a single component of $\partial \wt X^1 \setminus \partial W$.  \qed
\end{lemma}}

{We will also need the notion of ``cut'' in order to disconnect points in the Bowditch boundary:}
% In order to disconnect points in the Bowditch boundary, we will need to strengthen separation to a \emph{cut} in $\wt X$:

\begin{defn}\label{defn: saturation_cuts}
  Let $W_u$ be a quasiconvex wall in $\wt X$ and let $W$ be an $M$-saturation of $W_u$. Let $A$ be a quasiconvex subspace of $\wt X^1$ such that $\partial A$ has at least two points. We say that $W$ \emph{cuts} $A$ if \begin{enumerate}
      \item $\partial A \cap \partial W = \emptyset$ (as subsets of $\partial \wt X^1$), and
      \item $\partial W$ separates $\partial A$, i.e., $\partial A$ nontrivially intersects at least two components of $\partial \wt X^1 \setminus \partial W$. 
      % \item for each principal flow line $\Lambda$ of a singular vertex, either $\Lambda \subset W$ or the two points of $\partial \Lambda$ lie in a single component of $\partial \wt X^1 \setminus \partial W$. 
  \end{enumerate}
\end{defn}

% Observe that the third point of the definition does not depend on $A$. This is a preparation for the future coning-off of the flow lines.

\subsection{Cutting geodesics}
Recall that there are two types of geodesics (see \cref{defn: geodesic_types}) in $\wt{X}^1$. In this subsection, we will show that for each geodesic line that is not a principal flow line, there exists a wall whose saturation cuts the geodesic.

\begin{prop}\label{prop: walls_cut_non-horizontal_geodesics}
Let $\gamma : \mathbb{R} \to \wt{X}^1$ be an $N$-non-horizontal bi-infinite geodesic in the flow space. Then there exists a quasiconvex wall $W_u \to \wt X$ and an $M$-saturation $W$ of $W_u$ such that $W$ cuts $\gamma$.
\end{prop}

\begin{proof}
Proposition 5.18 of \cite{hagen_wise_irreducible} shows that in this case, there exists a quasiconvex wall $W_u$ which separates $\partial \gamma$. Let us quickly explain the idea behind their proof. We refer the reader to \cref{fig: nonhorizontal_cut} for an illustration of what follows.  We choose a wall $W_u$ from a regular effective set of walls (\cref{defn: effective_walls}(1)) with tunnel length $L$ sufficiently larger than $N$, while ensuring that the approximation $A(S)$ of some slope $S$ of $W_u$ intersects a vertical segment of $\gamma$ in an odd number of points (in the interior of some vertical edges), and far away from the endpoints of $A(S)$ ({arguing exactly as in \cite{hagen_wise_irreducible} while using \cref{lem: level_separation} for level separatedness}). Such a $W_u$ exists because of \cref{thm: effective_walls}. The fact that $\gamma$ is $N$-non-horizontal will ensure that $\gamma$ does not have long subpaths which $3\delta+2\lambda$-fellow-travel with $W_u$ (as the tunnel length $L$ of $W_u$ is much larger than $N$). This implies, by \cref{lem: quasiconvexity small overlaps}, that the union of $\gamma$ and $N(A(W_u))^1$ quasiisometrically embeds in $\wt X^1$ and thus $\partial W_u$  is disjoint from, and separates, $\partial \gamma$ in $\partial \wt X^1$. 

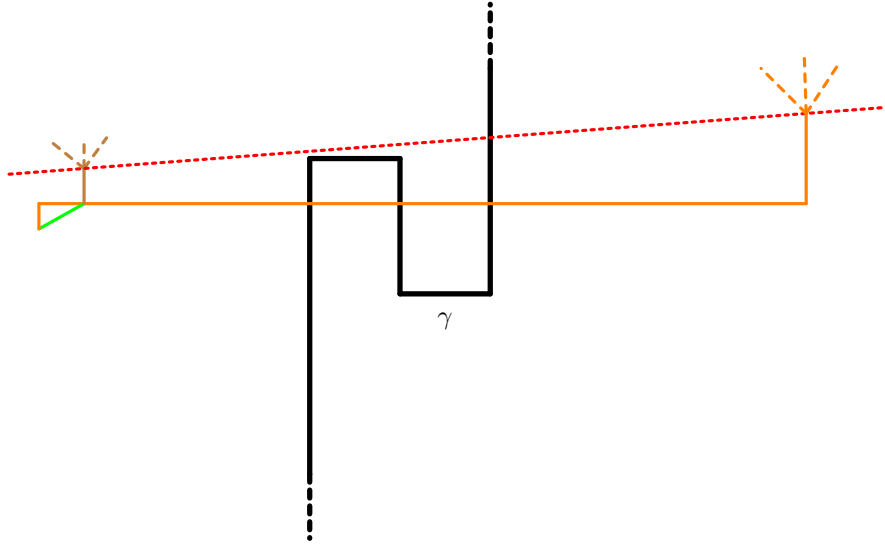
\begin{figure}
    \centering
    \begin{tikzpicture}[line cap=round,line join=round,>=triangle 45,x=.6cm,y=.6cm, dot/.style={circle,inner sep=1pt,fill,label={#1},name=#1},
 extended line/.style={shorten >=-#1,shorten <=-#1},
 extended line/.default=1cm]
% \clip(-25.017052975399984,-11.940152353799993) rectangle (19.725491640599994,17.36855283019999);
\draw [line width=2pt] (-5,-3)-- (-5,4);
\draw [line width=2pt] (-5,4)-- (-3,4);
\draw [line width=2pt] (-3,4)-- (-3,1);
\draw [line width=2pt] (-3,1)-- (-1,1);
\draw [line width=2pt] (-1,1)-- (-1,6);
\draw [line width=2pt,dashed] (-1,6)-- (-1,7.42);
\draw [line width=2pt,dashed] (-5,-3)-- (-5,-4.42);
\draw [line width=1.2pt,color=orange] (-11,3)-- (6,3);
\draw [line width=1.2pt,color=green] (-10,3)-- (-11,2.44);
\draw [line width=1.2pt,color=brown] (-10,3)-- (-10,3.78);
\draw [line width=1.2pt,dashed,color=brown] (-10,3.78)-- (-10.682644300000009,4.3330526800000015);
\draw [line width=1.2pt,dashed,color=brown] (-10,3.78)-- (-10,4.48);
\draw [line width=1.2pt,dashed,color=brown] (-10,3.78)-- (-9.5,4.46);
\draw [extended line, line width = 1.2pt, dotted, color = red] (-10,3.78) -- (6,5);
\draw [line width=1.2pt,color=orange] (-11,3)-- (-11,2.44);
\draw [line width=1.2pt,color=orange] (6,3)-- (6,5);
\draw [line width=1.2pt,dashed,color=orange] (6,5)-- (5,6);
\draw [line width=1.2pt,dashed,color=orange] (6,5)-- (5.96,6.22);
\draw [line width=1.2pt,dashed,color=orange] (6,5)-- (6.76,6.16);

\draw[color=black] (-2,0.4) node {$\gamma$};

\end{tikzpicture}
    \caption{Proof of \cref{prop: walls_cut_non-horizontal_geodesics}. The geodesic $\gamma$ hits the approximation of a slope of $W_u$. The red dotted line indicates a flow line in the saturation of $W_u$.}
    \label{fig: nonhorizontal_cut}
\end{figure}

{We now choose $M \in L\mathbb{N}$, sufficiently large, and let $W$ be the $M$-saturation of $W_u$.  Since $\partial W_u$ separates $\partial \gamma$, and $W_u \subset W$,  in order to show that $\partial W$ also separates $\partial \gamma$, it suffices to show that %provided that 
 $\partial \gamma$ is disjoint from $\partial W$, and for that it suffices to 
  %. The way to ensure disjointness is again to 
  show that the union of $\gamma$ and $N(A(W))^1$ quasiisometrically embeds in $\wt X^1$.  
  Since we already know that the union of $\gamma$ and $N(A(W_u))^1$ quasiisometrically embeds in $\wt X^1$
  and that the saturation consists in translates of $W_u$ that are at distance larger than $M$ from each other, and of principal flow lines between them,  a failure of quasiconvexity for the union of $\gamma$ and $N(A(W))^1$ would mean that $\gamma$ fellow travels with a principal flow line for a long distance (at least $M$ by definition of saturation). This contradicts the $N$-non-horizontality of $\gamma$. }
 % Since $M$ is much larger than $L$, and principal flow lines of $A(W)$ that intersect $A(W_u)$ 1-fellow-travel with $A(W_u)$ for distance at least $L$, subpaths of the $N$-non-horizontal $\gamma$ do not $3\delta+2\lambda$-fellow-travel with $A(W)$ longer than they do with $A(W_u)$ (
 %\cref{fig: nonhorizontal_cut}). An application of \cref{lem: quasiconvexity small overlaps} gives the desired result. 
 % )
% So far we have shown that $\partial W$ separates $\partial  \gamma$. To see that $W$ cuts $\partial \gamma$, simply observe that the definition of saturation guarantees the third point of \cref{defn: saturation_cuts}. 
\end{proof}

\begin{prop}\label{prop: walls_cut_horizontal_geodesics}
Let $\gamma : \mathbb{R} \to \wt{X}^1$ be a bi-infinite geodesic in the flow space such that \begin{enumerate}
    \item $\gamma$ is $N$-horizontal for all large $N$, and 
    \item no ray of $\gamma$ is asymptotic to any ray of a principal flow line.
\end{enumerate} Then there exists a quasiconvex wall $W_u \to \wt X$ and an $M$-saturation $W$ of $W_u$ such that $W$ cuts $\gamma$.
\end{prop}
\begin{proof}
Proposition 5.19 of \cite{hagen_wise_irreducible} ensures that there exists a quasiconvex wall $W_u$ which separates $\partial \gamma$; the wall $W_u$ is chosen from a periodic effective set of walls $\mathcal{W}_a$ (\cref{defn: effective_walls}(2)) so that, up to translation, the point $a$ in the nucleus of $W_u$ intersects $\gamma$ roughly in the middle of a long horizontal subpath of $\gamma$, whose length is much bigger than the tunnel length of $W_u$ (\cref{fig: horizontal_cut}). Such a $W_u$ exists, again, by \cref{thm: effective_walls}. By the way $W_u$ was chosen, slope approximations of $W_u$ do not have long $3\delta+2\lambda$-fellow-travelling subpaths with the flow line of $a$, and therefore with $\gamma$. \cref{lem: quasiconvexity small overlaps} now ensures that the union of $\gamma$ and $N(A(W_u)^1$ quasiisometrically embeds in $\wt X^1$. Therefore, $\partial W_u$ is disjoint from, and separates, $\partial \gamma$ in $\partial \wt X^1$.

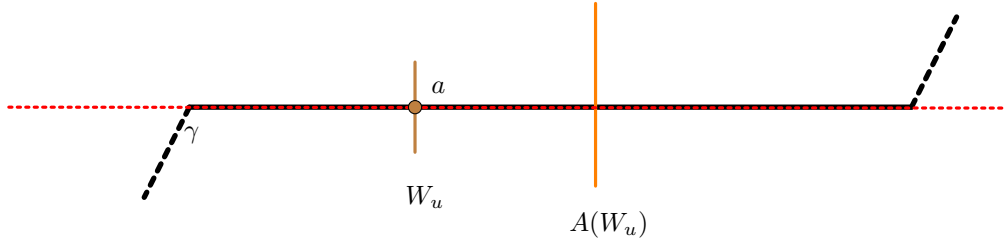
\begin{figure}
    \centering
    \begin{tikzpicture}[line cap=round,line join=round,>=triangle 45,x=.6cm,y=.6cm]

\draw [line width=2pt] (-7,2)-- (9,2);
\draw [line width=2pt,dashed] (-7,2)-- (-8,0);
\draw [line width=2pt,dashed] (9,2)-- (10,4);
\draw [line width=1.2pt,color=brown] (-2,2)-- (-2,3);
\draw [line width=1.2pt,color=brown] (-2,2)-- (-2,1);
\draw [line width=1.2pt,color=orange] (2,4.3)-- (2,2);
\draw [line width=1.2pt,color=orange] (2,2)-- (2,0.25);
\draw [line width=1.2pt,dotted,color=red] (-11,2)-- (11.18,1.98);
\draw[right] (-7.32,1.39) node {$\gamma$};
\draw [fill=brown] (-2,2) circle (2.5pt);
\draw [right](-1.84,2.43) node {$a$};
% \draw (1.12,1.85) node {$\Lambda_a$};
\draw (-1.78,0) node {$W_u$};
\draw (2.3,-.6) node {$A(W_u)$};
\end{tikzpicture}
    \caption{Proof of \cref{prop: walls_cut_horizontal_geodesics}. The geodesic $\gamma$ is cut along a nucleus of $W_u$ at $a$. The red dotted line indicates the principal flow line $\Lambda_a$.}
    \label{fig: horizontal_cut}
\end{figure}

In order to prove the proposition, we will now choose an $M$-saturation $W$ of $W_u$ with $M$ much larger than the length of the maximal subpath of $\gamma$ that $3\delta + 2\lambda$-fellow-travels with the principal flow line $\Lambda_a$ through $a$. Observe that this maximal subpath is bounded as $\partial \gamma$ is disjoint from $\partial \Lambda_a$. An application of \cref{lem: quasiconvexity small overlaps} now shows that the union of $N(A(W))^1$ with $\gamma$ is a quasiisometric embedding and thus $W$ cuts $\gamma$.
\end{proof}

\subsection{Separating geodesics from principal flow lines}

In this subsection, we will show that the union of a principal flow line and a geodesic ray is cut by a wall. More precisely:

\begin{prop}\label{prop: walls_cut_geodesic-parabolics}
Let $\Lambda$ be a principal flow line and $\gamma$ be a geodesic ray starting at some point on $\Lambda$. Suppose that $\gamma$ and $\Lambda$ are not asymptotic. Then there exists a quasiconvex wall $W_u \to \wt X$ and an $M$-saturation $W$ of $W_u$, such that \begin{enumerate}
    \item $W$ does not separate $\partial \Lambda$, and
    \item $W$ cuts $\Lambda \cup \gamma$.
\end{enumerate}
\end{prop}

\begin{proof}
Let $p \in \Lambda$ be the starting point of $\gamma$. Denote by $\Lambda_+, \Lambda_-$ the two geodesic subrays of $\Lambda$ starting at $p$. Denote the endpoints at infinity of these rays by $\Lambda_+(\infty), \Lambda_-(\infty), \gamma(\infty)$.   Let $\epsilon$ be the maximal Gromov product at $p = \gamma(0)$ of $\gamma(\infty)$ and 
$ \Lambda_{\pm}(\infty)$. Then both the concatenations $\gamma \cup \Lambda_{\pm}$ are $(1, 2\epsilon + 10\delta)$-quasigeodesic lines. We will now choose a quasiconvex wall $W_u$ (see \cref{fig:geodesic_principal_cut}) such that 
\begin{enumerate}
    \item $W_u \cap (\gamma \cup \Lambda_-) = W_u \cap (\gamma \cup \Lambda_+) \subset \gamma$,
    \item the union of the quasigeodesic lines $\gamma \cup \Lambda_{\pm}$ with $N(A(W_u)^1$ embeds quasiisometrically, and 
    \item $\partial W_u$ separates $\gamma(\infty)$ from $\Lambda_+(\infty)$ and $\Lambda_-(\infty)$.  
\end{enumerate} 
 
The choice of $W_u$ that will work for us is a wall given by either \cref{prop: walls_cut_non-horizontal_geodesics} or \cref{prop: walls_cut_horizontal_geodesics} depending on whether or not $\gamma$ is $N$-horizontal for all $N$. 

Let $\lambda$ be the quasiconvexity constant from \cref{prop: quasiconvexity_fwd_ladder}, and $\epsilon$ as before. Let $\kappa, L$ be respectively the quasiconvexity constant for $N(A(W_u))^1$ and the tunnel length of $W_u$  for such walls, as guaranteed by \cref{prop: approximations_quasiconvex}. We  choose $W_u$  so that $W_u$ intersects $\gamma$ at points at distance much larger than $2\epsilon + 20\delta + 2\kappa + \lambda + L$ from $p$.  

Thanks to the large distance between the point $p$ and $W_u \cap \gamma$, \cref{lem: quasiconvexity small overlaps} says that the union of $N(A(W_u)^1$ with $\gamma \cup \Lambda$ embeds quasiisometrically and both $W_u$ and $A(W_u)$ are disjoint from $\Lambda$. 
This implies that $\partial W_u$ does not separate $\partial \Lambda$. We now choose an $M$-saturation $W$ of $W_u$ with $M$ much larger than the maximal $3\delta + 2\lambda$-fellow-travelling length between pairs of principal flow lines. 
This choice ensures, yet again by \cref{lem: quasiconvexity small overlaps}, that $W$ is disjoint from $\Lambda$ and therefore cuts $\gamma \cup \Lambda$ but does not separate $\partial \Lambda$.
\end{proof}

\begin{figure}
    \centering
    \begin{tikzpicture}[line cap=round,line join=round,>=triangle 45,x=.6cm,y=.6cm]
\draw [shift={(0,-7)},line width=2pt]  plot[domain=0.9505468408120751:2.191045812777718,variable=\t]({1*8.602325267042627*cos(\t r)+0*8.602325267042627*sin(\t r)},{0*8.602325267042627*cos(\t r)+1*8.602325267042627*sin(\t r)});
\draw [line width=2pt] (0,1.6023252670426267)-- (0,4);
\draw [line width=2pt] (0,4)-- (3,4);
\draw [line width=2pt] (3,4)-- (3,5.54);
\draw [line width=2pt,dotted] (3,5.54)-- (4,7);
\draw [line width=1.2pt,color=orange] (-5,5)-- (8,5);
\draw [line width=1.2pt,color=orange] (-5,5)-- (-5,4.52);
\draw [line width=1.2pt,color=orange] (8,5)-- (8,6.14);
\draw [line width=1.2pt,dashed,color=orange] (8,6.14)-- (7.48,6.54);
\draw [line width=1.2pt,dashed,color=orange] (8,6.14)-- (8.64,6.54);
\draw [line width=1.2pt,dashed,color=orange] (8,6.14)-- (7.997272727272721,6.998363636363641);
\draw [line width=1.2pt,color=green] (-5,4.52)-- (-4.410727272727277,5);
\draw [line width=1.2pt,color=brown] (-4.410727272727277,5)-- (-4.406133371910058,5.709666507323151);
\draw [line width=1.2pt,dashed,color=brown] (-4.406133371910058,5.709666507323151)-- (-4,6);
\draw [line width=1.2pt,dashed,color=brown] (-4.406133371910058,5.709666507323151)-- (-4.410727272727277,6.184363636363641);
\draw [line width=1.2pt,dashed,color=brown] (-4.406133371910058,5.709666507323151)-- (-4.872727272727277,6.030363636363641);
\draw[color=black] (2.225607983978148,1.8140275619789166) node {$\Lambda$};
\draw[color=black, below] (1,4) node {$\gamma$};

\end{tikzpicture}
    \caption{The wall $W_u$ is chosen so as to hit the geodesic ray $\gamma$ far from the principal flow line $\Lambda$.}
    \label{fig:geodesic_principal_cut}
\end{figure}
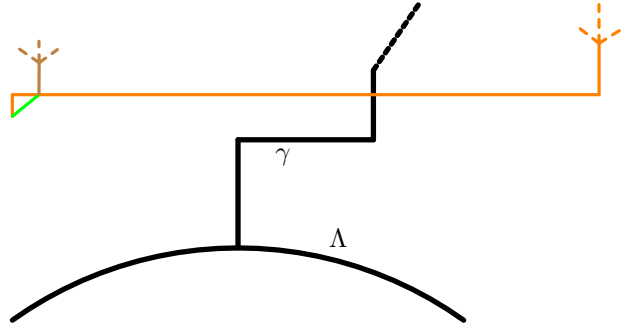

\subsection{Cutting pairs of principal flow lines}

\begin{prop}\label{prop: walls_cut_parabolic-parabolic}
Let $\Lambda_1 \neq \Lambda_2$ be two principal flow lines.  Then there exists a quasiconvex wall $W_u \to \wt X$ and an $M$-saturation $W$ of $W_u$ such that \begin{enumerate}
    \item $W$ cuts $\Lambda_1 \cup \Lambda_2$, and
    \item $\partial W$ does not separate $\partial \Lambda_i$.
\end{enumerate}
\end{prop}

\begin{proof}
 Let $\gamma$ be a geodesic segment between $\Lambda_1$ and $\Lambda_2$. Note that $\gamma$ contains at least one vertical edge $e$, since $\Lambda_1 \neq \Lambda_2$. By \cref{lem: periodic_divergence_principal_lines}, there exists a periodic point $x$ in the edge $e$ such that the periodic flow line $\Lambda_x$ is disjoint from the lines $\Lambda_i$ and diverges from $\Lambda_i$ in both the forward and backward directions. We now choose an immersed wall $W_u$ from a regular effective set (\cref{defn: effective_walls}(1)) and satisfying the following:\begin{enumerate}
     \item $W_u$ has tunnel length $L$ much larger than the $3\delta + 2\lambda$-fellow-travelling length of $\Lambda_x$ with each $\Lambda_i$,
     \item the intersection of $\Lambda_x$ with the approximation $A(W_u)$ is a segment of length $L$ with centre $x$,
     \item the approximation $A(W_u)$ intersects the edge $e$ at the singleton $\{x\}$ and the geodesic $\gamma$ in an odd cardinality set, and 
     \item the intersection of the union $\Lambda_1 \cup \Lambda_2 \cup \gamma$ with $A(W_u)$ coincides with the intersection of $\gamma$ with $A(W_u)$.
 \end{enumerate}
 Such a $W_u$ exists as $\wt X$ is level-separated (\cref{lem: level_separation}) and has many effective walls (\cref{thm: effective_walls}). Property (1) above allows us to apply \cref{lem: quasiconvexity small overlaps} to the union of $\Lambda_1 \cup \Lambda_2 \cup \gamma$ with $N(A(W))^1$ and conclude that it is a quasiisometric embedding in $\wt X^1$. Therefore, $\partial W_u$ is disjoint from, and separates, $\partial (\Lambda_1 \cup \Lambda_2)$. Further, by property (4), it does not separate $\partial \Lambda_i$. We now choose an $M$-saturation $W$ of $W_u$ for $M$ much larger than $L$. By applying \cref{lem: quasiconvexity small overlaps} again and arguing as above for $W_u$, we obtain the desired result.
\end{proof}

\section{Separating pairs of points in the Bowditch boundary} \label{sec: separation_bowditch}
The goal of this section is to show that the mapping torus $ \Gamma = G \rtimes_{\phi} \mathbb{Z}$ satisfies the hypothesis of \cref{thm: bdry_criterion}. Recall that $\Gamma$ is hyperbolic relative to the collection $\scrP$ consisting of the suspensions of the free factors $H_i$ of $G$.

Let us first make two observations connecting the Gromov boundary of the flow space $\wt X^1$ with the Bowditch boundary $\partial_B (\Gamma, \scrP)$. We denote by $\widehat{\wt X^1}$ the cone-off of the (1-skeleton of the) flow space $\wt X^1$ over the set of principal flow lines. Recall also that $\widehat{\Gamma}$ denotes the cone-off of $\Gamma$ relative to $\scrP$.

As a consequence of the relative \v{S}varc-Milnor Lemma (\cite[Theorem 5.1]{charney_crisp}), we obtain:

\begin{prop}\label{prop: rel_swarc_milnor}
The cone-off $\widehat{\wt X^1}$ is $\Gamma$-equivariantly quasiisometric to the cone-off $\widehat{\Gamma}$.
\end{prop}

Let $\beta: \wt X^1 \to \widehat{\wt X^1}$ denote the inclusion map from the flow space to its cone-off. We will denote by $\partial \beta$ the induced map from $\partial \wt X^1$ to the Bowditch boundary $\partial_B(\Gamma,\scrP)$. Observe that $\partial \beta$ maps the limit points of any principal flow line $\Lambda$ to the corresponding cone vertex $v_{\Lambda}$. In fact, 

\begin{cor}\label{cor: two-to-one_Bowditch}
The map $\partial \beta : \partial \wt{X}^1 \to  \partial_B (\Gamma, \mathscr{P})$ is {continuous and} surjective. Further, $\forall \xi \in \partial_B (\Gamma, \mathscr{P})$, $\partial \beta^{-1}(\xi)$ is either a singleton or contains two points. The latter arises if and only if $\xi$ is a cone vertex.
\qed
\end{cor}

\begin{prop}\label{prop;enough_visibility}
    If $\xi, \zeta$ are points in $\partial \wt X^1$ that are preimages of conical limit points of the Bowditch boundary,  then there exists a geodesic line in $\wt X^1$ joining them. 

    If $\xi \in \partial \wt X^1$ is the preimage of a conical limit point of the Bowditch boundary, and $z_0$ a vertex in $\wt X$, then there exists a geodesic ray from $z_0$ that converges to $\xi$.
\end{prop}
\begin{proof}

Let $(x_n), (y_n)$ be sequences of vertices in $\wt X$ respectively going  to $\xi, \zeta$, and let $[x_n, y_n]$ be a geodesic in $\wt X^1$. Let $z_0 \in T_0$ minimise the Gromov product $( \xi, \zeta)_{z_0} \in \bbN$, and  let $R_0$ be the minimum. 

Let $B(z_0, R)$ be the ball of $\wt X^1$ of radius $R$ centered at $z_0$. We aim to prove that for all $R>R_0$, there exists $m\in \bbN$ such that  $\{ B(z_0, R) \cap [x_n, y_n], n>m \}$ is a finite collection of segments of length $\geq 2R- 2R_0 -10\delta$.  

First, for $n$ large enough, $B(z_0, R) \cap [x_n, y_n]$ is indeed such a segment, otherwise one easily gets that $( \xi, \zeta)_{z_0} >R_0$. 

If for some $R$, the collection is infinite, let $k$ be such that in the tree $T_k$ the intersection $T_k\cap  B(z_0, R) \cap [x_n, y_n]$ is an infinite collection of segments as $n$ varies (since horizontal segments are determined by their endpoints, such a $k$ exists). Let $z_k \in T_k$, and let $u_n$ in  $T_k\cap  B(z_0, R) \cap [x_n, y_n]$. Assume that the maximal angle of $[z_k, u_n]_{T_k}$ goes to infinity with $n$. Let $v_k$ be the closest vertex to $z_k$ at which the angle goes to infinity. Then the same is true for the initial vertical segment of either $[z_k, x_n]$ or $[z_k, y_n]$. It follows that either $(x_n)$ or $(y_n)$ converges to a point that is an image of the parabolic point of the principal flow line of $v_k$, contrary to our assumption. 

The number of subsegments of $[x_n, y_n]$ around $z_0$ of a given length is therefore bounded when $n$ varies.  One can thus extract diagonally subsequences such that, inductively, for each given length, the subsegments of this length of $[x_n, y_n]$ around $z_0$ make a constant sequence. We thus obtain a bi-infinite geodesic between $\xi$ and $\zeta$ in $\wt X$. 

The second assertion is obtained with a similar argument.
\end{proof}

{\begin{prop}\label{prop: separation_flowspace_to_bowditch}
    Let $\xi, \zeta$ be two distinct points in the Bowditch boundary and let $W$ be the $M$-saturation of a quasiconvex wall $W_u$ {such that $W$} cuts a geodesic between $\beta^{-1}(\xi)$ and $\beta^{-1}(\zeta)$. Then $\partial \beta(\partial W) = \partial \mathrm{stab}(W)$ separates $\xi, \zeta$ and there exists a subgroup $K$ of index at most 2 of $\mathrm{stab}(W)$ such that $\xi$ and $\zeta$ are in $K$-distinct components of $\partial_B (\Gamma, \scrP) \setminus \partial \beta(\partial W)$.
\end{prop}}

Let us mention a word of caution here. Though $\mathrm{stab}(W)$ is a codimension-1 subgroup (\cref{prop: saturated_walls_stabilisers}) of $\Gamma$, the statement above is necessary because, 
in general, a codimension-1 subgroup need not separate the Bowditch boundary. As an example, take the cyclic subgroup $P = \langle aba^{-1}b^{-1}\rangle$ of the free group $F(a,b)$ with peripheral structure $\scrP = P$. Here $P$ is a maximal parabolic subgroup that is full, relatively quasiconvex and codimension-1, but does not separate the Bowditch boundary, a circle.

\begin{proof}
{Let $\xi, \zeta$, and $W_u$ be as in the statement, and $\wt \xi$, $\wt \zeta$ be two points in $\wt X^1$ in the preimages $\partial \beta^{-1}(\xi)$ and $\partial \beta^{-1}(\zeta)$ respectively. Let $\calU_\xi, \calU_\zeta$ be a clopen partition of   $\partial  \wt X^1 \setminus \partial W$ containing respectively $\wt \xi, \wt \zeta$. 
Consider points in $\partial  \wt X^1 $  identified in $\partial_B(\Gamma, \calP) $. By \cref{cor: two-to-one_Bowditch} they are at the end of the same principal flow line, and 
            % such that one is in  $\calU_\xi$, and the other is in  $\calU_\zeta$. 
            %Then, by 
by \cref{lem: saturation flowline no cut}, if one is in    $\calU_\xi$, so is the 
other one.  
        %end points of a principal flowline that meets $W_u$, and therefore they are 
        %in $\partial W$. 
        %This is a contradiction with the choice of $\calU_\xi, \calU_\zeta$.
  It follows that the  map $\partial \beta$ sends  $\calU_\xi, \calU_\zeta$ on disjoints subsets  $\partial \beta (\calU_\xi)$ and $\partial \beta (\calU_\zeta)$  of  $\partial_B(\Gamma, \calP) \setminus \partial \beta (\partial W) $. By surjectivity, they  provide a partition of   $\partial_B(\Gamma, \calP) \setminus \partial \beta (\partial W) $, each containing respectively $\xi$ and $\zeta$.}   
  
  {We need to check that both $\partial \beta (\calU_\xi)$ and $\partial \beta (\calU_\zeta)$  are closed. By symmetry, it is enough to check that one is closed. If a sequence $(x_n)$ in   $\partial \beta (\calU_\xi)$ converges in $\partial_B(\Gamma, \calP) \setminus \partial \beta (\partial W) $, lift it to a sequence $(x'_n)$ in  $\partial  \wt X^1\setminus \partial W$. Since $\partial \wt X^1$ is not compact, the sequence $(x'_n)$ may or may not have an accumulation point. If it does, then the fact that $\calU_\xi$ is closed, along with the continuity
  of $\partial \beta$, gives us the necessary conclusion.
  % , and up to extracting a subsequence, we only need to
  Now consider the case where it has no accumulation point in $\partial  \wt X^1$.  If $x'_n$ is a point represented by a ray $\rho'_n$ in $\wt X^1 \setminus W$ (from a given base point) that lives in the component of which $\xi$ is adherent, it means that, up to taking a subsequence, all the rays $\rho'_n$ pass through a singular vertex, and make an angle $\theta_n$ at that vertex such that $\theta_n \rightarrow \infty$ with $n$. Take $v$ to be the closest singular vertex to the base point with this property.  The limit of $(x_n)$ in $\partial_B(\Gamma, \calP) \setminus \partial \beta (\partial W) $ is then the parabolic point associated to the principal flow line of this vertex $v$.  Since the rays live in the component of  $\wt X^1 \setminus W$   adherent to $\xi$,    this principal flow line has at least one end point not in $\calU_\zeta$. By property of saturations, either both end points are in $\calU_\xi$, or both are in $\partial W$, and in that case the limit of $(x_n)$ is in $\partial \beta (\partial W)$, which we assumed otherwise. So both points of the line are in $\calU_\xi$, and their images, which is the limit of $(x_n)$, is in  $\partial \beta \calU_\xi$. Therefore this set is closed.}
    
{Therefore, $\xi, \zeta$ are separated by $\partial \beta(\partial W)$, which is $\partial \mathrm{stab}(W)$, since the action of the latter is cocompact on $W$.  }

\if0
{   Consider the image $\beta(W)$ of $W$. Let us denote by $W'$ the union of $\beta(W)$ with the cone vertices $v_{\Lambda}$ for the principal flow lines $\Lambda$ that intersect $W$. Observe that $W'$ separates $\widehat{\wt X^1}$ into several orbits of deep components (by \cref{prop: saturated_walls_stabilisers}). The result now follows by applying Theorem 5.7 of \cite{einstein_groves_ng}. We remark that though their definition of hyperset requires exactly two complementary components, their proof works as long as there are at least two orbits of complementary components under a finite index subgroup. Let us explain why this holds for us.}
\fi

{Fix a basepoint in the complement of $W$. Let $A$ be the union of components $C$ of $\wt X \setminus W$ such that there is a path from the basepoint to $C$ that crosses the set  $W'_u$ an even number of times. Recall that $W'_u$ is the complement in $W$ of the set of its principal flow lines. We note that $\wt X \setminus W$ is thus partitioned into two subsets: $A$ and its complement, which we denote here by $B$. Since the stabiliser of $W$ sends complementary components of $W$ to complementary components,  it preserves this partition of $W$ into $A$ and $B$. Therefore, there exists a subgroup of index at most 2 which preserves $A$ (and hence also $B$). }

{By the way the saturation $W$ of $W_u$ cuts the preimages of $\xi$ and $\zeta$ (as in \cref{prop: walls_cut_non-horizontal_geodesics}, \cref{prop: walls_cut_horizontal_geodesics},  \cref{prop: walls_cut_geodesic-parabolics} or \cref{prop: walls_cut_parabolic-parabolic}), it is easy to see that $\xi$ and $\zeta$ do not both lie in the limit set of $A$, or in the limit set of $B$, which gives us the desired result.}
\end{proof}

We now prove the main result \cref{thm: main_intro}, namely, in our current notations, that  the relatively hyperbolic group $(\Gamma, \scrP)$ is relatively cubulated. 

\begin{proof}[Proof of \cref{thm: main_intro}]
   { We will show that the boundary criterion \cref{thm: bdry_criterion} holds. Let $\xi \neq \zeta$ be two points in the Bowditch boundary $\partial_B (\Gamma,\scrP)$. We have three cases. \begin{enumerate}
        \item[Case 1.] Both $\xi$ and $\zeta$ are conical points.
        By abuse of notation, we denote by $\xi$ and $\zeta$ the unique preimages in $\partial \wt X^1$ of $\xi$ and $\zeta$ (\cref{cor: two-to-one_Bowditch}). Let $\gamma$ be a geodesic in $\wt X^1$ joining $\xi$ and $\zeta$ as given by \cref{prop;enough_visibility}. Then \cref{prop: walls_cut_non-horizontal_geodesics} and \cref{prop: walls_cut_horizontal_geodesics} ensure that there exists a wall saturation $W$ that cuts $\gamma$. An application of \cref{prop: separation_flowspace_to_bowditch} gives the desired result.
        \item[Case 2.] Without loss of generality, $\xi$ is conical while $\zeta$ is parabolic. We consider a geodesic ray $\gamma$ in $\wt X^1$, between a vertex of the principal flow line $\Lambda$ associated to $\zeta$, and the preimage of $\xi$, as given by \cref{prop;enough_visibility}. \cref{prop: walls_cut_geodesic-parabolics} ensures that the union of $\gamma$ and $\Lambda$ is cut by a wall saturation $W$. \cref{prop: separation_flowspace_to_bowditch} then gives the result.
        \item[Case 3.] Both $\xi$ and $\zeta$ are parabolic points. \cref{prop: walls_cut_parabolic-parabolic} ensures that the principal flow lines associated to $\xi$ and $\eta$ are cut by a wall saturation $W$. The result then follows from \cref{prop: separation_flowspace_to_bowditch}. \qedhere
    \end{enumerate} }
\end{proof}
\bibliographystyle{alpha}
\bibliography{cubulation_free_products}

\end{document}